\documentclass[12pt,a4paper]{article}
\usepackage{a4,a4wide}
\usepackage{amsfonts,amsmath,amssymb,amsthm}
\usepackage{color}
\usepackage{ifpdf} 
\ifpdf
    \usepackage[pdftex]{graphicx}
    \DeclareGraphicsExtensions{.png,.pdf,.jpg}
\else
   \usepackage[dvips]{graphicx}
   \DeclareGraphicsExtensions{.eps}
\fi
\graphicspath{{.}{figuresPL/}}
\usepackage[latin1]{inputenc}
\numberwithin{equation}{section}
\renewcommand\k{\kappa}
\newtheorem{theorem}{Theorem}[section]

\newtheorem{lemma}[theorem]{Lemma}
\newtheorem{proposition}[theorem]{Proposition}
\newtheorem{corollary}[theorem]{Corollary}
\newtheorem{definition}[theorem]{Definition}
\newtheorem{remark}[theorem]{Remark}

\newcommand\C{\mathbb{C}}

\newcommand{\gsim}{\stackrel{>}{\sim}}
\newcommand{\lsim}{\stackrel{<}{\sim}}

\newcommand\diag{\mathop{\mathrm {diag}}}

\renewcommand\tilde{\widetilde}

\newcommand\calM{\mathcal{M}}
\newcommand\calI{\mathcal{I}}
\newcommand\nab{\nabla}

\newcommand\EGL{E_\mathrm{GL}}
\def\a{\alpha}
\def\b{\beta}
\def\phimin{\varphi_\mathrm{min}}
\def\phimax{\varphi_\mathrm{max}}
\def\hal{\frac{1}{2}}
\def\div{\mathrm{div} \,}

\def\R{\mathbb{R}}
\newcommand\bext{b_{\mathrm{ext}}}
\newcommand\Bext{B_{\mathrm{ext}}}
\def\pa{\partial}

\newcommand\calA{\mathcal{A}}
\newcommand\calD{\mathcal{D}}

\newcommand{\step}[1]{{\it Step #1:}}

\newcommand\m{\mu}

\def\Om{\Omega}

\def\Z{\mathbb{Z}}
\def\N{\mathbb{N}}
\def\M{\mathcal{M}}
\def\H{\mathcal{H}}
\def\B{\mathcal{B}}

\def\cP{\mathcal{P}}

\def\LM#1{\hbox{\vrule width.2pt \vbox to#1pt{\vfill \hrule width#1pt
height.2pt}}}
\def\LL{{\mathchoice {\>\LM7\>}{\>\LM7\>}{\,\LM5\,}{\,\LM{3.35}\,}}}
\def\restr{{\LL}}
\renewcommand{\phi}{\varphi}
\def\p{\phi}
\def\Div{\textup{div}\,}

\def\dist{\textup{dist}}

\def\1{\mathbf{1}}
\def\loc{\mathrm{loc}}

\def\Ext{\mathrm{ext}}

\def\XXint#1#2#3{{\setbox0=\hbox{$#1{#2#3}{\int}$ }
\vcenter{\hbox{$#2#3$ }}\kern-.57\wd0}}

\def\eps{\varepsilon}
\renewcommand{\subset}{\subseteq}

\def\lt{\left}
\def\rt{\right}
\def\les{\lesssim}
\def\ges{\gtrsim}
\def\lsim{\lesssim}
\def\gsim{\gtrsim}
\def\Divp{\Div\!'}

\def\wE{\widetilde{E}_T}

\def\weaklim{\rightharpoonup}
\def\what{\widehat}
\def\wilde{\widetilde}
\def\Mreg{\mathcal{M}_\mathrm{reg}}
\def\MNreg{\mathcal{M}_\mathrm{reg}^N}
\def\wG{\wilde{G}_\b}
\def\wF{\what{F}_{\a,\b}}
\def\Lz{L_0}

\def\hL{\what{L}}

\def\tchi{\wilde{\chi}}
\def\tB{\wilde{B}}

\newcommand\dotX{\dot{X}}

\begin{document}

\title{{A branched transport limit of the Ginzburg-Landau functional}}
\author{Sergio Conti\footnote{Institut f\"ur Angewandte Mathematik, Universit\"at Bonn, Germany, email: sergio.conti@uni-bonn.de}\and Michael Goldman\footnote{ LJLL, Universit\'e Paris Diderot,  CNRS, UMR 7598,  France email: goldman@math.univ-paris-diderot.fr, part of his research was funded by a Von Humboldt PostDoc fellowship} \and Felix Otto\footnote{Max Planck Institute for Mathematics in the Sciences, Leipzig, Germany, email: otto@mis.mpg.de} \and Sylvia Serfaty\footnote{Courant Institute, NYU \& Institut Universitaire de France \& UPMC, Paris, France, email: serfaty@cims.nyu.edu}}
\date{}
\maketitle
\begin{abstract}\noindent
We study the  Ginzburg-Landau model of type-I superconductors in the regime of small external magnetic fields.
We show that, in an appropriate asymptotic regime, flux patterns are  described by a simplified 
branched transportation functional. We derive the simplified functional from the  full Ginzburg-Landau model
rigorously via $\Gamma$-convergence. The detailed analysis of the limiting procedure and the study of the limiting
functional lead to a  precise understanding 
of the multiple scales contained in the model.
\end{abstract}

\section{Introduction}
In 1911, K. Onnes discovered the phenomenon of superconductivity, manifested in the complete loss of resistivity of certain metals and alloys at very low temperature.
W. Meissner discovered in 1933 that this was coupled with the expulsion of the magnetic field from the superconductor at the critical temperature. This is now called the Meissner effect. 
After some preliminary works of the brothers F. and H. London, V. Ginzburg and L. Landau proposed in  1950
a phenomenological model describing the state of a superconductor. In their model (see \eqref{GLorigin} below), which belongs to Landau's general theory of second-order phase transitions, the state of the material is represented by the order parameter $u: \Omega\to \C$, where $\Omega$ is the material sample. The density of superconducting electrons is then given by $\rho:= |u|^2$. 
A microscopic theory of superconductivity was first proposed by  Bardeen-Cooper-Schrieffer (BCS) in 1957,
and the Ginzburg-Landau model was derived from BCS  by Gorkov in 1959 (see also \cite{frankseiringer} for a rigorous derivation). 

One of the main achievements of the Ginzburg-Landau theory is the 
prediction and the understanding of the mixed (or intermediate) state below the critical temperature.
This is a state in which, for moderate external magnetic fields, normal and superconducting regions coexist.
The behavior of the material  in the Ginzburg-Landau theory is characterized by   two physical parameters. 
The first is the coherence length $\xi$ which measures the typical length on which $u$ varies,
the second is the penetration length $\lambda$ which gives the typical length on which the magnetic field penetrates the superconducting regions. The Ginzburg-Landau parameter is then defined  as $\kappa:= \frac{\lambda}{\xi}$. 
The Ginzburg-Landau functional is  given by
\begin{equation}\label{GLorigin}
 \int_{\Omega} |\nabla_A u|^2 +\frac{\kappa^2}{2} (1-|u|^2)^2 dx +\int_{\R^3} |\nabla\times A- \Bext|^2 dx
\end{equation}
where $A: \R^3\to \R^3$ is the magnetic potential (so that $B:=\nabla\times A$ is the magnetic field), $ \nabla_A u:= \nabla u -i Au$ is the covariant derivative of $u$ and $\Bext$ is the external magnetic field.
In these units, the penetration length $\lambda$ is normalized to $1$.  As first observed by A. Abrikosov this theory predicts two types of superconductors. On the one hand, when $\kappa < 1/\sqrt{2}$, there is a positive surface tension which leads to the formation of normal and superconducting regions corresponding to $\rho\simeq 0$ and $\rho\simeq 1$ respectively, separated by interfaces.
These are the so-called type-I superconductors. On the other hand, when $\kappa>1/\sqrt{2}$, this surface tension is negative and one expects to see the magnetic field penetrating the domain through lines of vortices. These are the so-called type-II superconductors.
In this paper we are interested in better understanding the former type but we refer the interested reader to \cite{tinkham,SanSerf,Serfrev} for more information about 
the latter type. In particular, in that regime, there has been an intensive work on understanding the formation of regular patterns of vortices known as Abrikosov lattices.

In type-I superconductors, it is observed  experimentally \cite{ProzHo,Proz,Proal} that complex patterns appear at the surface of the sample. It is believed that these patterns are a manifestation of branching patterns inside the sample. 
Although the observed states are highly history-dependent, it is argued in \cite{ChokKoOt,Proal} that
the hysteresis is governed by low-energy configurations at vanishing external magnetic field. 
The scaling law of the ground-state energy was determined
in \cite{ChokConKoOt,ChokKoOt} for a simplified  sharp 
interface version of the Ginzburg-Landau functional \eqref{GLorigin} and in \cite{CoOtSer} for the full energy, 
these results indicate the presence of a regime with branched patterns at low fields. 

This paper aims at a better understanding of these branched patterns by going beyond the scaling law. Starting from the full Ginzburg-Landau functional, we prove that in the regime of vanishing external magnetic field, 
low energy configurations are made of nearly one-dimensional superconducting threads branching towards the boundary of the sample.  In a more mathematical language, we prove $\Gamma-$convergence \cite{braides,dalmaso} of the Ginzburg-Landau
functional to a kind of branched transportation functional in this regime. We focus on the simplest geometric setting by considering the sample $\Om$ to be a box $Q_{\Lz,T}:=(-\frac{\Lz}{2}, \frac{\Lz}{2})^2\times(-T,T)$ for some $T,\Lz>0$ and consider periodic lateral boundary conditions. 
The external magnetic field is taken to be perpendicular to the sample, that is $\Bext:= \bext e_3$ for some $\bext>0$ and where $e_3$ is the third vector of the canonical basis of $\R^3$.
After making an isotropic rescaling, subtracting the bulk part of the energy and dropping lower order terms (see \eqref{eqEGLE} and the discussion after it),  minimizing \eqref{GLorigin} can be seen as  equivalent to minimizing

\begin{multline}\label{Etilde}
  E_T(u,A):=\frac{1}{L^2}
\int_{Q_{L,1}} |\nabla_{TA} u |^2  + \left(B_3 - \a (1-\rho)\right)^2 + |B'|^2 dx
+ \| B_3-\a \b\|_{H^{-1/2}(\{x_3=\pm 1\})}^2,\,
\end{multline}
where we have  let $B:=(B',B_3):=\nabla \times A$,
\[ \kappa T:=\sqrt{2}\alpha,  \qquad \bext:=\frac{\beta \kappa}{\sqrt{2}} \qquad \textrm{ and } \qquad L:=\Lz/T.\]
If $u=\rho^{1/2} \exp(i\theta)$, since $|\nabla_{TA} u |^2 =|\nabla \rho^{1/2}|^2 +\rho|\nabla \theta-T A|^2$, in the limit $T\to +\infty$ we obtain, at least formally, that 
$A$ is a gradient field in the region where $\rho>0$ 
and therefore
the Meissner condition $\rho B=0$ holds. 
Moreover, in the regime $\alpha\gg 1$, from \eqref{Etilde} we see that $B_3 \simeq \a (1-\rho)$ and $\rho$ takes almost only  values in $\{0,1\}$. Hence $\Div B=0$ can be rewritten as 
\begin{equation*}
  \partial_3 \chi + \frac1\alpha\Divp \chi B'=0,
\end{equation*}
where $\chi:= (1-\rho)$ and $\Divp$ denotes the divergence with respect to the first two variables. Therefore,  from the Benamou-Brenier formulation of optimal transportation \cite{AGS,villani} and since from the Meissner condition, $B'\simeq \frac{1}{\chi} B'$,
the term 
\[\int_{Q_{L,1}} |B'|^2 dx\simeq \int_{Q_{L,1}} \frac{1}{\chi}|B'|^2 dx\]
in the energy \eqref{Etilde} can  be seen as a transportation cost. We thus expect that inside the sample 
(this is, in $Q_{L,1}$), superconducting domains where $\rho\simeq1$ and $B\simeq 0$ alternate with normal ones where 
$\rho\simeq 0$ and $B_3\simeq \alpha$. 
Because of the last term $\| B_3-\a \b\|_{H^{-1/2}(\{x_3=\pm 1\})}^2$ in the energy \eqref{Etilde}, one expects
$B\simeq \a \b e_3$ outside the sample. This implies that close to the boundary the normal domains have to refine. The interaction between the surface energy, the
transportation cost and the penalization of an $H^{-1/2}$ norm leads to the formation of complex patterns (see Figure \ref{figscales}).

It has been proven in \cite{CoOtSer} that in the regime $T\gg 1$, $\a \gg 1$ and $\b \ll 1$, 
\begin{equation}\label{scallawintro}
\min E_T(u,A)\sim \min\{ \a^{4/3} \b^{2/3}, \a^{10/7} \b\}\,.
\end{equation}
The scaling $\min E_T(u,A)\sim \a^{4/3} \b^{2/3}$ (relevant for $\a^{-2/7}\ll \b$) corresponds to uniform branching patterns whereas the scaling   $\min E_T(u,A)\sim \a^{10/7} \b$
corresponds to non-uniform branching ones. We focus here for definiteness
on the  regime $\min E_T(u,A)\sim \a^{4/3} \b^{2/3}$, although we believe that our proof can be extended to the other one.
Based on the construction giving the upper bounds in \eqref{scallawintro}, we expect that in the first regime there are multiple scales appearing (see Figure \ref{figscales}):
\begin{equation}\label{scalesepar}\stackrel{\displaystyle\textrm{penetration}}{\textrm{ length}}\,\stackrel{\displaystyle\ll}{\, } \, \stackrel{\displaystyle\textrm{ coherence}}{\textrm{ length}}\, \stackrel{\displaystyle\ll}{\,} 
\, \stackrel{\displaystyle\textrm{ diameter of the}}{\textrm{ threads in the bulk}} \, \,  \stackrel{\displaystyle\ll}{\,} \, \,  \stackrel{\displaystyle\textrm{distance between the}}{\textrm{ threads in the bulk}},\end{equation}
which amounts in our parameters to 
\[T^{-1} \ll \alpha^{-1}\ll \alpha^{-1/3}\beta^{1/3}\ll \alpha^{-1/3}\beta^{-1/6}.\]
\begin{figure}\label{figscales}\centering\resizebox{16cm}{!}{
 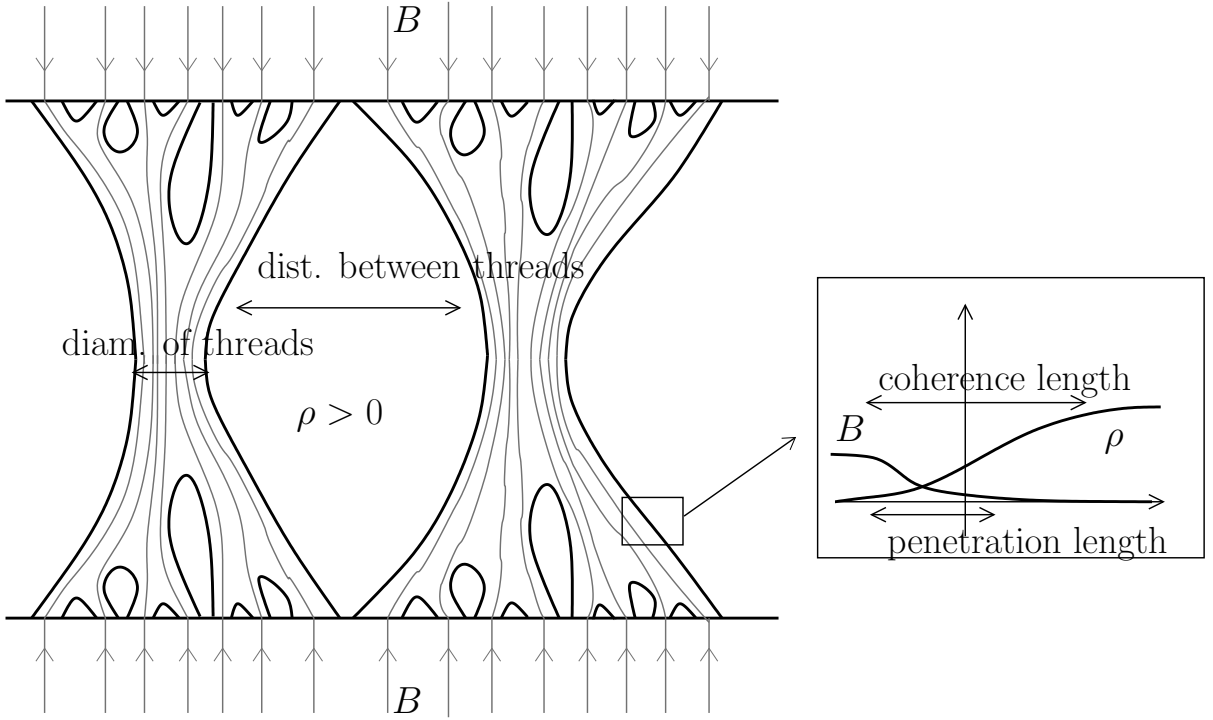
}
\caption{The various lengthscales}
\end{figure}
In order to better describe the minimizers we focus on the extreme region of the phase diagram
$T,T \a^{-1}, \beta^{-1}, \a \b^{7/2}\to +\infty$, with $L=\tilde L  \alpha^{-1/3}\beta^{-1/6}$ for some fixed $\tilde L>0$. In this regime, we have in particular $\a^{-1}\ll \a^{-1/3}\b^{1/3}$ so 
that the separation of scales \eqref{scalesepar} holds. We introduce an anisotropic  rescaling (see Section \ref{Sec:GL}) which leads to the functional

\begin{align}
\label{eqdefwe}
\wE(u,A)&:=\frac{1}{\widetilde{L}^2}\Big[\int_{Q_{\tilde L,1}} \a^{-2/3}\b^{-1/3}\lt|\nabla_{\a^{1/3}\b^{-1/3}TA}'u\rt|^2+ 
  \a^{-4/3}\b^{-2/3}\lt|(\nabla_{\a^{1/3}\b^{-1/3}TA} u)_3\rt|^2\\
\nonumber
 +  & \lt.\a^{2/3}\b^{-2/3}\left({B_3} -(1-|{u}|^2)\right)^2
+  \b^{-1}|{B'}|^2 dx +\a^{1/3}\b^{7/6}\|\b^{-1}{B_3}-1\|_{H^{-1/2}(x_3=\pm1)}^2\rt].
\end{align}
Our main result is a  $\Gamma-$convergence  result of the functional $\wE$  towards a 
functional defined on measures $\mu$ living on one-dimensional trees. These trees correspond to the normal regions 
in which $\rho\simeq 0$ and where  the magnetic field $B$ penetrates the sample. Roughly speaking, if for a.e. $x_3\in (-1,1)$
the slice of $\mu=\mu_{x_3}\otimes dx_3$ has the form $\mu_{x_3}=\sum_i \phi_i \delta_{X_i(x_3)}$ where the sum is at most countable, then we let (see Section \ref{Sec:limiting} for a precise definition)
\begin{equation}\label{eqdefIintro}
I(\mu):=\frac1{\widetilde{L}^2}\int_{-1}^1 K_* \sum_i \sqrt{\phi_i}+ \phi_i |\dotX_i|^2 dx_3, 
\end{equation}
 where $K_*=8\sqrt\pi/3$ and $\dotX_i$ denotes the derivative (with respect to $x_3$) of $X_i(x_3)$.  The $X_i$'s represent the graphs of each branch of the tree (parametrized by height) and the $\phi_i$'s represent the flux carried by the branch. 
We can now state our main result
\begin{theorem}\label{maintheo}
 Let $T_n, \a_n, \b_n^{-1}\to+\infty$ with $T_n \a_n^{-1}, \a_n \b_n^{7/2}\to +\infty$, $\tilde L>0$, then: 
 \begin{enumerate}
  \item\label{maintheolb}
  For every sequence  $(u_n, A_n)$ with $\sup_n \wE(u_n,A_n)<+\infty$, up to subsequence , $\beta_n^{-1}(1-|u_n|^2)$ weakly converges to a measure $\mu$ of the form $\mu=\mu_{x_3}\otimes dx_3$ with 
  $\mu_{x_3}=\sum_i \phi_i \delta_{X_i}$ for a.e. $x_3\in (-1,1)$, $\mu_{x_3} \weaklim dx'$ (where $dx'$ denotes the two dimensional Lebesgue measure on $Q_{\tilde L}$) when $x_3\to \pm 1$ and such that
  \[\liminf_{n\to +\infty}  \wE(u_n,A_n)\ge I(\mu).\]
  \item\label{maintheoub} If in addition $L_n^{2} \a_n \b_n T_n\in 2\pi \Z$, where $L_n:= \tilde L \a_n^{-1/3}\b_n^{-1/6}$,  then for every measure $\mu$ such that $I(\mu)<+\infty$ and $\mu_{x_3} \weaklim dx'$  as $x_3\to \pm1$ , there exists $(u_n,A_n)$ such that $\b_n^{-1}(1-|u_n|^2)\weaklim \mu$ and  
\[\limsup_{n\to +\infty} \, \wE(u_n,A_n)\le I(\mu).\]
 \end{enumerate}

\end{theorem}
\begin{proof} By scaling, 
it suffices to consider the case $\tilde L=1$.
The first assertion follows from Proposition  \ref{gammaliminf}, the second one from Proposition  \ref{gammalimsup-v2}.
\end{proof}
Let us stress once again that our result could have  been equivalently stated for the full Ginzburg-Landau energy  \eqref{GLorigin} instead of $\wE$ (see Section \ref{Sec:GL}).\\
Within our periodic setting, the quantization condition $L_n^2 \a_n \b_n T_n\in 2\pi \Z$ for the flux is a consequence of the fact that the phase circulation of the complex-valued function in the original problem is naturally quantized. 
It is necessary in order to make our construction but we believe that it is also a necessary condition for having sequences of bounded energy (see the discussion in Section \ref{Sec:GL} and the construction in Section \ref{recoveryGL}).
We remark that scaling back to the original variables
this condition is the physically natural one
$L_0^2\bext\in 2\pi\Z$.
\\

Before going into the proof of Theorem \ref{maintheo} we address
the limiting functional $I(\mu)$, which has many similarities with irrigation  (or branched transportation) models 
that have recently attracted a lot of attention (see \cite{BeCaMo}  and  more detailed comments in Section \ref{Sec:reliri} or the recent papers \cite{BraWirth, BraWir16} where the connection is also made to some urban planning models).  In Section \ref{Sec:limiting}, we first prove that the variational problem for this limiting functional is well-posed (Proposition \ref{existmu}) and show a scaling law for it (Proposition \ref{branchingmu} and Proposition \ref{branchinglowerbound}).
In Proposition \ref{def:subsystem}, we define the notion of
subsystems which allows us to remove part of the mass carried by the branching measure. This notion is at the basis of Lemma \ref{noloop} and Proposition \ref{finitebranch} which show that minimizers contain no loops and that
far from the boundary, they are made of a finite number of branches. 
From the no-loop property, we easily deduce  Proposition \ref{reg} which is a regularity result for minimizers of $I$.  The main result of Section \ref{Sec:limiting}
is Theorem \ref{theodens} which proves the density 
of ``regular'' measures in the topology given by the energy $I(\mu)$. As in nearly every $\Gamma-$convergence result, such a property is crucial in order  to implement the construction for the upper bound \ref{maintheoub}.

We now comment on the proof of Theorem \ref{maintheo}. Let us first point out that if the Meissner condition $\rho B =0$ were to hold, and $A$ could be written 
as a gradient field in the set $\{\rho>0\}$, then  $|\nabla_{TA} u|^2= |\nabla \rho^{1/2}|^2$ and we would have
\begin{equation}\label{introlowerbound}\int_{Q_{L,1}} |\nabla_{TA} u |^2  + \left(B_3 - \a (1-\rho)\right)^2dx=\int_{Q_{L,1}} |\nabla \rho^{1/2} |^2  +  \a^2 \chi_{\rho>0} (1-\rho)^2 +\chi_{\rho=0} (B_3-\a)^2 dx.\end{equation}
This is a Modica-Mortola \cite{ModMort} type of functional with a degenerate double-well potential given by $W(\rho):= \chi_{\rho>0} (1-\rho)^2$. Thanks to Lemma \ref{lemmameissner}, one can control how far we are from satisfying the Meissner condition. From this, we deduce that \eqref{introlowerbound} almost holds (see  Lemma \ref{lemmafirstlowerbound}).
This implies that the Ginzburg-Landau energy gives a control over the perimeter of the superconducting region $\{\rho>0\}$. In addition, $\b\ll 1$ imposes a small cross-area fraction for $\{\rho>0\}$.
Using then isoperimetric effects to get convergence to one-dimensional objects (see Lemma \ref{lowerbound2d}), we may 
use Proposition \ref{gammaliminf} to conclude the proof of \ref{maintheolb}.

In order to prove \ref{maintheoub}, thanks to the density result in Theorem \ref{theodens}, it is enough to consider regular measures. Given such a measure $\mu$, we first 
approximate it with quantized measures (Lemma \ref{lemmaquantize}).
Far from branching points the construction is easy (see Lemma \ref{lemmacurve3}). 
At a branching point, we need to pass from one disk to two (or vice-versa); this is 
done passing through rectangles 
(see Lemma \ref{Lembranch3} and Figure \ref{fig1}). 
Close to the boundary 
we use instead the construction from  \cite{CoOtSer}, which explicitly generates a specific branching pattern with the optimal energy scaling; since the 
height over which this is done is small the prefactor is not relevant here (Proposition \ref{propboundraylayer}).
The last step is to define a phase and a magnetic potential to get back to the full Ginzburg-Landau functional. This is possible since we made the
construction with  the Meissner condition and quantized
fluxes enforced, see Proposition \ref{thirdupperbound}.

{}From \eqref{Etilde} and the discussion around \eqref{introlowerbound}, for type-I superconductors, the Ginzburg-Landau functional can be seen as a non-convex, non-local (in $u$) functional favoring oscillations, 
regularized by a surface term which selects the lengthscales of the microstructures. The appearance of branched
structures for this type of problem is shared by many other functionals appearing in material sciences such as shape memory alloys \cite{KohnMuller92,KohnMuller94,Conti00,KnuepferKohnOtto2013,BelGol,Zwicknagl2014,ChanConti2015,ContiZwicknagl}, uniaxial ferromagnets \cite{ChoKoOtmicro,ViehOtt,KnMu}
and blistered thin films \cite{BCDM00,JinSternberg,BCDM02}.
Most of the  previously cited results on branching patterns (including \cite{ChokConKoOt,ChokKoOt,CoOtSer} for type-I superconductors) focus on scaling laws. Here, as in \cite{ViehOtt,CDZ}, we go one step further and
prove that, after a suitable anisotropic rescaling, configurations of low energy converge to branched patterns.  The two main  difficulties in our model with respect to the one studied in \cite{ViehOtt} are the presence of an additional lengthscale (the penetration length) and its sharp limit counterpart, the Meissner condition $\rho B=0$
which gives a nonlinear coupling between $u$ and $B$.
Let us point out that for the Kohn-M\"uller model \cite{KohnMuller92,KohnMuller94},  a much stronger result is known, namely that minimizers are  asymptotically self-similar \cite{Conti00} (see also \cite{Viehmanndiss,AlChokOt} for related results). In \cite{Gol}, the optimal microstructures for a two-dimensional analogue of $I(\mu)$ are exactly computed.

The paper is organized as follows. In section \ref{sec:not}, we set some notation and recall some notions from optimal transport theory. In Section \ref{Sec:GL}, we recall the definition
of the Ginzburg-Landau functional together with various important quantities such as the superconducting current. We also introduce there the anisotropic rescaling leading to the functional $\widetilde{E}_T$. 
In Section \ref{Sec:intermediate}, we introduce for the sake of clarity intermediate functionals corresponding to the different scales of the problem. Let us stress that we will not use them in the rest of the paper but strongly believe that they help understanding the structure of the problem. In Section \ref{Sec:limiting}, we carefuly define the
limiting functional $I(\mu)$ and study its properties. In particular we recover a scaling law for the minimization problem and prove regularity of the minimizers. We then prove the density in energy of 'regular' measures. This is a crucial result for the main $\Gamma-$ convergence result which is proven in the last two sections. As customary, we first prove the lower bound in Section \ref{Sec:gammaliminf} and then make the upper bound construction in Section \ref{Sec:upperbound}.

\section{ Notation and preliminary results}\label{sec:not} 
In the paper we will use the following notation. The symbols $\sim$, $\ges$, $\les$ indicate estimates that hold up to a global constant. For instance, $f\les g$ denotes the existence of a constant $C>0$ such that $f\le Cg$,
$f\sim g$ means $f\les g$ and $g\les f$. 
In heuristic arguments we use $a\simeq b$ to indicate that $a$ is close (in a not precisely specified sense) to $b$.
We use a prime to indicate the first two components of a vector in $\R^3$, and
identify $\R^2$ with $\R^2\times\{0\}\subset\R^3$.
Precisely, for $a\in \R^3$ we write $a'=(a_1,a_2,0)\in\R^2\subset\R^3$;  given
two vectors $a,b\in\R^3$ we write $a'\times b'=(a\times b)_3 =
(a'\times b')_3$. We denote by $(e_1,e_2,e_3)$ the canonical basis of $\R^3$. For $L>0$ and $T>0$, 
 $Q_L:=(-\frac{L}{2},\frac{L}{2})^2$ and $Q_{L,T}:=Q_L\times(-T,T)$. For a function $f$ defined on $Q_{L,T}$, we denote $f_{x_3}$ the function $f_{x_3}(x'):=f(x',x_3)$ and we analogously define for $\Omega \subset Q_{L,T}$, 
 the set $\Omega_{x_3}$. For $x=(x',x_3)$ and $r>0$ we let $\B_r(x)=\B(x,r)$ be the ball of radius $r$ centered at $x$ (in $\R^3$) 
and $\B'_r(x')=\B'(x',r)$ be the analogue two-dimensional ball centered at $x'$. Unless specified otherwise, all the functions and
measures we will consider are periodic in the $x'$ variable, i.e., we identify $Q_L$ with the torus $\R^2/L\Z^2$. In particular, for $x',y'\in Q_L$, $|x'-y'|$ denotes the distance for the metric of the torus, i.e., $|x'-y'|:=\min_{k\in \Z^2} |x'-y'+Lk|$. We denote by $\H^k$ the $k-$dimensional Hausdorff measure.
We let  $\cP(Q_L)$ be  the space of probability 
measures on $Q_L$ and $\M(Q_{L,T})$ be the space of finite Radon measures on $Q_{L,T}$, and similarly $\M(Q_L)$. Analogously, we define $\M^+(Q_{L,T})$ and $\M^+(Q_L)$ as the spaces of finite Radon measures which are also positive. For a measure $\mu$ and a function $f$, we denote by $f\sharp \mu$ the push-forward of $\mu$ by $f$. 

We recall the definition of the (homogeneous) $H^{-1/2}$ norm of a
function 
$f\in L^2(Q_{L})$ with $\int_{Q_{L}} f dx'=0$,
\begin{equation}\label{defHdem}
 \|f\|_{H^{-1/2}}^2:=\inf\lt\{ \int_{Q_{L}\times [0,+\infty)} |B|^2 dx \ : \ \Div B=0, \ B_3(\cdot,0)=f \rt\},
\end{equation}
which  can be alternatively given in term of the two-dimensional Fourier series as
(see, e.g., \cite{ChokKoOt})
\begin{equation*}
 \|f\|_{H^{-1/2}}^2=\frac{1}{2\pi} \sum_{ k'\in \lt(\frac{2\pi}{L}\Z\rt)^2\backslash\{0\}} \frac{ |\hat f (k')|^2}{|k'|}.
\end{equation*}
We shall write $\|f\|^2_{H^{-1/2}(Q_L\times\{\pm T\})}$ for 
 $\|f\|^2_{H^{-1/2}(Q_L\times\{T\})}+
 \|f\|^2_{H^{-1/2}(Q_L\times\{-T\})}$.

The $2$-Wasserstein distance between 
 two measures $\mu$ and $\nu\in \calM^+(Q_L)$ with $\mu(Q_L)=\nu(Q_L)$ is given by
\[W_2^2(\mu,\nu):=\min\lt\{ \mu(Q_L)\int_{Q_L\times Q_L} |x-y|^2 \, d\Pi(x,y)  \, : \, \Pi_1=\mu, \ \Pi_2=\nu\rt\},\] 
where the minimum is taken over  measures on  $Q_L\times Q_L$ and   $\Pi_1$  and $\Pi_2$ are   the first and  second marginal of
$\Pi$\footnote{ for $p\ge 1$, we analogously define $W_p^p(\mu,\nu):=\min\lt\{ \mu(Q_L)\int_{Q_L\times Q_L} |x-y|^p \, d\Pi(x,y)  \, : \, \Pi_1=\mu, \ \Pi_2=\nu\rt\}$.}, respectively.
For measures $\mu,\nu\in\calM^+(Q_{L,T})$, the $2$-Wasserstein distance is correspondingly defined.
We now introduce  some notions from metric analysis, see \cite{AGS,villani}
for more detail.
A curve 
$\mu:(a,b)\to \cP(Q_L)$, $z\mapsto \mu_z$ belongs to $AC^2(a,b)$ (where $AC$ stands for absolutely continuous) if there exists
$m\in L^2(a,b)$ such that 
\begin{equation}\label{metricderiv}
W_2(\mu_{z},\mu_{\tilde{z}})\le \int_{z}^{\tilde{z}} m(t) dt \ \qquad \forall a<z\le \tilde{z} <b.\end{equation}
For any such curve, {the {\it speed}}
\[|\mu'|(z):=\lim_{\tilde{z}\to z}\frac{W_2(\mu_{\tilde{z}},\m_{z})}{|z-\tilde{z}|}\]
exists for $\H^1-$a.e. $z\in (a,b)$ and $|\mu'|(z)\le m(z)$ for $\H^1$-a.e. $z\in (a,b)$ for every admissible $m$ in \eqref{metricderiv}. 
Further, there exists a Borel vector field $B$ such that 
\begin{equation}\label{eqexistsB}
 B(\cdot,z)\in L^2(Q_L,\mu_{z}), \qquad \| B(\cdot,z)\|_{L^2(Q_L,\mu_{z})}\le |\mu'|(z) \quad \textrm{for } \H^1\text{-a.e. } \ z\in (a,b)
\end{equation}
and the continuity equation
\begin{equation}\label{conteqini}
 \partial_3 \mu_{z}+\Divp (B \mu_{z})=0
\end{equation}
holds in the sense of distributions \cite[Th. 8.3.1]{AGS}.
Conversely, if a weakly continuous curve $\mu_{z} : (a,b)\to \cP(Q_L)$ 
satisfies the continuity equation \eqref{conteqini} for some Borel vector field $B$ with $\| B(\cdot,z)\|_{L^2( Q_L,\mu_{z})}\in L^2(a,b)$ then
$\mu\in AC^2(a,b)$ and  $|\mu'|(z)\le \| B(\cdot,z)\|_{L^2(Q_L,\mu_{z})}$ for $\H^1$-a.e. $z\in (a,b)$.
In particular, we have
\begin{equation}\label{eqw22b}
W_2^2(\nu,\hat\nu)= \min_{\mu,B} \lt\{ 2T\int_{Q_{L,T}} |B|^2 d\mu_{z} dz \ :\ \mu_{-T}=\nu, \ \mu_T= \hat\nu \textrm{ and \eqref{conteqini} holds}\rt\}\,,
 \end{equation}
where by scaling the right-hand side does not depend on $T$.\\

For a (signed) measure $\mu\in \M(Q_L)$, we define the Bounded-Lipschitz norm of $\mu$ as 
\begin{equation}\label{BLdefin}
 \|\mu\|_{BL}:=\sup_{\|\psi\|_{Lip}\le 1} \int_{Q_L} \psi d\mu,
\end{equation}
where for a $Q_L-$periodic and Lipschitz continuous function $\psi$, $\|\psi\|_{Lip}:=\|\psi\|_{\infty}+\|\nabla \psi\|_{\infty}$. 
By the Kantorovich-Rubinstein Theorem \cite[Th. 1.14]{villani}, the $1-$Wasserstein and the Bounded-Lipschitz norm are equivalent. 

\section{The Ginzburg-Landau functional}\label{Sec:GL}
In this section we recall some background material concerning the Ginzburg-Landau functional and introduce the anisotropic rescaling leading to $\wE$.\\
For a (non necessarily periodic) function  $u:Q_{\Lz,T}\to\C$, called the order parameter, and a vector potential $A: Q_{\Lz}\times\R\to \R^3$ (also not necessarily periodic), we define the covariant derivative
\[\nabla_A u:= \nabla u-iAu,\]
the magnetic field
\[B:=\nabla\times A,\]
and the {superconducting} current 
\begin{equation}\label{defj}
  j_A:= \frac12\left( -i\bar u (\nabla_A u)+ i u (\overline{\nabla_A  u})\right)=\textrm{Im}(iu \overline{\nabla_A u})\,.
\end{equation}
Let us first notice that $|\nabla_A u|^2$ and
the observable quantities $\rho$, $B$ and $j_A$ are invariant under change of gauge. That is,  if we replace $u$ by $u e^{i\varphi}$ and $A$ by $A+\nabla \varphi$
for any function $\varphi$, they remain unchanged.
  We also point out that
  if $u$ is written in polar coordinates as
 $u=\rho^{1/2} e^{i\theta}$, then
\begin{equation*}
 |\nabla_A u|^2=|\nabla \rho^{1/2}|^2+\rho|\nabla \theta-A|^2.
\end{equation*}
For any admissible pair $(u,A)$, that is  such that $\rho$, $B$ and $j_A$ are $Q_{\Lz}$-periodic, we define the Ginzburg-Landau functional as
\begin{equation*}
  \EGL(u,A):=\int_{Q_{\Lz,T}} |\nabla_A u|^2
+ \frac{\kappa^2}{2} (1-|u|^2)^2 dx
+ \int_{Q_{\Lz}\times \R} |\nabla\times A-\Bext|^2 dx\,.
\end{equation*}
We remark that $u$ and $A$ need not be (and, if $\Bext\ne 0$, cannot be) periodic. See \cite{CoOtSer} for more details on the functional spaces we are using.
Here $\Bext:=\bext e_3$  is the external magnetic field and 
$\kappa\in(0,1/\sqrt2)$ is a {material constant, called the Ginzburg-Landau parameter}.
From periodicity and $\Div B=0$ it follows that 
$ \int_{Q_{\Lz}\times\{x_3\}} B_3 \, dx'$ does not depend on $x_3$ and therefore,
if the energy is finite, necessarily
\begin{equation}
\label{B2bex}
\int_{Q_{\Lz}\times\{x_3\}} B_3 \, dx'=\Lz^2 
\bext\hskip1cm\text{ for all }x_3\in\R\,.
\end{equation}
We first remove the  bulk part from the 
energy $\EGL$. In order to do so, we   introduce  the quantity
\begin{equation*}
  \calD_A^3u :=   (\partial_2 u - i A_2 u) - i (\partial_1 u
 - i A_1 u) 
 = (\nabla_A u)_2 - i (\nabla_A u)_1 
\end{equation*}
and, more  generally,
\begin{equation*}
  \calD_A^ku := (\partial_{k+2} u - i A_{k+2} u) - i (\partial_{k+1} u
 - i A_{k+1} u)  = (\nabla_A u)_{k+2} - i (\nabla_A u)_{k+1} \, ,
\end{equation*}
where components are understood cyclically (i.e., $a_k=a_{k+3}$). 
The operator $\calD_A$  (which corresponds to a creation operator for a magnetic Laplacian) was used by Bogomol'nyi in the proof of the self-duality of the Ginzburg-Landau functional at $\k=\frac{1}{\sqrt{2}} $ (cf. e.g. \cite{jaffetaubes}). 
His proof relied on  identities similar to the next ones, which will be crucial in  enabling us to separate the leading order part of the energy. 

Expanding the squares, one  sees (for details see \cite[Lem. 2.1]{CoOtSer})
that (recall that $\rho:=|u|^2$)
\begin{equation}\label{magic}
  |\nabla'_A u|^2 =  |\calD_A^3u|^2 + \rho B_3 + \nabla'\times 
   j'_A \end{equation}
and, for any $k=1,2,3$,
\begin{equation}\label{magic2}
   |(\nabla_A u)^{k+1}|^2 +    |(\nabla_A u)^{k+2}|^2= |\calD_A^ku|^2
   + \rho B_k + (\nabla\times j_A)_k  \,.
\end{equation}
This implies
\begin{equation*}
|\nabla_A' u |^2 =   (1-\k \sqrt2) |\nabla_A' u |^2 + \k\sqrt2 
|\calD_A^3u|^2 
+\k\sqrt2\rho B_3 + \k\sqrt2\nabla'\times j_A'\,.
\end{equation*}
The last term integrates to zero by the
periodicity of $j_A$. Therefore, for 
each fixed $x_3$, using \eqref{B2bex}, we have 
\begin{equation*}
  \int_{Q_{\Lz}} |\nabla_A' u |^2\,  dx' =    \int_{Q_{\Lz}} (1-\k \sqrt2) |\nabla_A' u |^2 + \k\sqrt2 
|\calD_A^3u|^2  +\k\sqrt2(\rho-1) B_3 \, dx'+ \Lz^2 \k\sqrt2 \bext\,.
\end{equation*}
We substitute and obtain, using $\int_{Q_{\Lz}} (B_3-\bext)^2 dx'=\int_{Q_{\Lz}} B_3^2-\bext^2 dx'$ {and completing squares}, 
\begin{equation} \label{eqEGLE}
\EGL(u,A)=2T\Lz^2\left(\k\sqrt2 \bext  -\bext^2\right) + E(u,A)
  + \k\sqrt2
  \int_{Q_{\Lz,T}}|\calD_A^3u|^2-|\nabla_A'u|^2 dx,
\end{equation}
where
\begin{eqnarray*}
  E(u,A)&:=&
\int_{Q_{\Lz,T}} |\nabla_A u |^2  + \left(B_3 - \frac{\k}{\sqrt2} (1-\rho)\right)^2 dx 
\\&&
+  \int_{Q_{\Lz}\times \R} |B'|^2  dx+
   \int_{Q_{\Lz}\times(\R\setminus(-T,T))} (B_3 -\bext)^2 dx\,.
\end{eqnarray*}
In particular, the bulk energy is $2{\Lz}^2T(\k\sqrt2 \bext-\bext^2) $.
Since we are interested in the regime $\kappa\ll 1$ and since $ |\calD_A^3u|^2\le 2 |\nabla_A'u|^2$, the contribution of the last term 
in (\ref{eqEGLE}) to the energy is (asymptotically) negligible with respect to the first term in $ E$, and therefore it 
can be ignored in the following.

Applying \eqref{defHdem} to $B-\bext e_3  $ and  minimizing outside $Q_{\Lz} \times [-T,T]$ if necessary, the last two terms in
$E(u,A)$
can be replaced by
$$\int_{Q_{{\Lz},T}} |B'|^2 dx+ \| B_3-\bext
\|_{H^{-1/2}(\{x_3= T\})}^2+\| B_3-\bext
\|_{H^{-1/2}(\{x_3=-T\})}^2\,,$$ 
so that $ E(u,A)$ becomes
\begin{equation}\label{E}
  E(u,A)=
\int_{Q_{\Lz,T}}  |\nabla_A u |^2 + \left(B_3 - \frac{\k}{\sqrt2} (1-\rho)\right)^2 +  |B'|^2 dx+\| B_3-\bext
\|_{H^{-1/2}(\{x_3=\pm T\})}^2\,.
\end{equation}

Let us notice that the normal solution $\rho=0$, $B=\bext e_3$ (for which we can take $A(x_1,x_2,x_3)=\bext x_1 e_2$) is always admissible but has energy equal to 
\[\EGL(u,A)= \Lz^2 T \kappa^2\gg 2\Lz^2T(\k\sqrt2 \bext-\bext^2), \]
in the regime $\kappa\gg \bext$ that we consider here.

The following scaling law is  established in \cite{CoOtSer}.
\begin{theorem}\label{theoCoOtSer}
For $\bext<\kappa/8$, $\kappa\le 1/2$, $\kappa
T \ge 1$, $\Lz$ sufficiently large, if  the quantization condition
\begin{equation}\label{quantificationini}
 \bext \Lz^2 \in 2\pi \Z,
\end{equation}
holds then
\begin{equation}
\label{minE}
\min E(u,A) \sim \min \left\{\bext^{2/3} \kappa^{2/3} T^{1/3} , \bext \kappa^{3/7} T^{3/7} \right\}\Lz^2.\end{equation}
\end{theorem}
 
We believe that \eqref{quantificationini} is also a necessary condition for \eqref{minE} to hold. Indeed, we expect that if $(u,A)$ is such that 
\[ E(u,A) \les  \left\{\bext^{2/3} \kappa^{2/3} T^{1/3} \Lz^2, \bext \kappa^{3/7} T^{3/7} \Lz^2 \right\},\]
then the normal phase $\rho\simeq 0$ is the minority phase (typically disconnected on every slice) and there exist $x_3\in(-T,T)$ and (periodic) curves $\gamma_1$ and $\gamma_2$ such that
 $\Gamma_1:=\{(\gamma_1(s), s, x_3) : s\in[ 0, 1]\} \subset \{\rho\simeq 1\}$ and $\Gamma_2:=\{(s,\gamma_2(s), x_3) : s\in[ 0, 1]\} \subset \{\rho\simeq 1\}$,
 with $\gamma_i(1)=\gamma_i(0)+L_0 e_i$. If this holds then using  
Stokes Theorem on large domains the boundary of which is made of concatenations of the curves $\Gamma_i$, it is possible to prove that \eqref{quantificationini} must  hold. 
As in \cite{CoOtSer}, we will need to assume \eqref{quantificationini} in order to build the recovery sequence in Section \ref{recoveryGL}.\\

The first regime in \eqref{minE} corresponds to uniform branching patterns while the second corresponds to well separated branching trees (see \cite{ChokConKoOt,ChokKoOt,CoOtSer}). 
We focus here on the first regime, that is $\k^{5/7}\ll \bext T^{2/7}$, and replace $\kappa$ and $\bext$ by the variables
$\alpha$, $\beta$, defined according to 
$$\kappa T = \sqrt{2}\a \qquad \bext= \frac{\beta \kappa }{\sqrt{2}}=\frac{\alpha\beta}{T}\,,$$
and then  rescale
 \begin{center}
\begin{tabular}{ll}
$\hat{x}:= x/T$,   &$  L:=\Lz/T$,\\[6pt]
 $\hat{u}(\hat{x}):= u(x)$,  &$\hat{A}(\hat{x}):= A(x)$\\[6pt]
 $E_T(\hat u,\hat A):=\frac{1}{T L^2} E(u,A)$,& \end{tabular}
\end{center}
so that in particular $ \hat{B}(\hat{x})= \widehat{\nab}\times \hat{A}(\hat{x})= TB(x)$ and $\nab_A u(x)= T^{-1} \widehat{\nab}
\hat{u}( x/T)-i \hat{A}(x/T)\hat u (x/T)$. Changing variables 
 and removing the hats yields
\begin{equation*}
E_T(u,A)=  \frac{1}{L^2}\lt(\int_{Q_{L,1}}
|\nab_{T{A}} {u}|^2+ \left({B_3} - \a(1-\rho)\right)^2 +  |{B'}|^2 dx+\|{B_3}- \a\b\|_{H^{-1/2}(\{x_3=\pm1\})}^2\rt)\,.
\end{equation*}
as was anticipated in (\ref{Etilde}).
In these new variables, the  scaling law \eqref{minE} becomes  $E_T\sim \min\{\alpha^{4/3}\beta^{2/3},\a^{10/7}\beta \}$ and the uniform branching regime corresponds to $E_T\sim \alpha^{4/3}\beta^{2/3}$ which amounts to $\a^{-2/7}\ll \b\ll 1$,
see also (\ref{scallawintro}).
 Constructions (leading to the upper bounds in \cite{CoOtSer,ChokKoOt,ChokConKoOt}), suggest that in this regime, typically, the penetration length of the magnetic field inside 
the superconducting regions is of the order of $T^{-1}$, the coherence length (or domain walls) is of the order of $\a^{-1}$, the width of the normal domains in the bulk is of the order of $\a^{-1/3} \b^{1/3}$
and their separation of order $\a^{-1/3}\b^{-1/6}$. These various lengthscales
motivate the anisotropic rescalings that we will introduce in Section \ref{Sec:intermediate}.
\def\no{
For later reference, let us point out that in these new parameters, the quantization condition \eqref{quantificationini} reads
\begin{equation}\label{quantificationini2}
 TL^2 \a \b \in 2\pi \Z.
\end{equation}
Notice that since $T$ is very large compared to the other parameters, this condition is less 
and less restrictive as $T\to +\infty$. We further remark that the $L$ in (\ref{quantificationini2}) is $\hat L=L_0/T$.}

In closing this section we present the anisotropic rescaling that will lead to the functional defined in (\ref{eqdefwe}), postponing to the 
next section a detailed explanation of its motivation. We set for $x \in Q_{L,1}$, 
\begin{center}
\begin{tabular}{ll}
 $\begin{pmatrix}\tilde{x}'\\ \tilde{x}_3\end{pmatrix}:= \begin{pmatrix}\a^{1/3} \b^{1/6} x'\\x_3\end{pmatrix}$, &\hspace{2cm} 
 $ \tilde L:= \a^{1/3}\b^{1/6}L$,\\[8pt]
 $\begin{pmatrix}\tilde{A}'\\ \tilde{A}_3\end{pmatrix}(\tilde{x}):=\begin{pmatrix}\a^{-2/3}\b^{1/6} A'\\ \a^{-1/3} \b^{1/3}A_3\end{pmatrix}(x)$, & \hspace{2cm} $ \tilde{u}(\tilde{x}):=u(x),$
\end{tabular}
\end{center}
to get $\tilde B_3(\tilde x)=\alpha^{-1}B_3(x)$, $\tilde B'(\tilde x)=\alpha^{-2/3}\beta^{1/6}B'(x)$ inside the sample. 
Outside the sample, i.e. for $|x_3|\ge 1$, we make the isotropic rescaling $\tilde{x}:= \pm e_3+\a^{1/3}\b^{1/6}( x \mp e_3)$
to get $\tilde B(\tilde x)=\a^{-1} B(x)$. A straightforward computation leads to
 $\wE(\tilde u, \tilde A)=\a^{-4/3}\b^{-2/3} E_T(u,A)$, where
\begin{multline*}
  \wE(u,A):=\frac{1}{\widetilde{L}^2}\Big[\int_{Q_{\tilde L,1}} \a^{-2/3}\b^{-1/3}\lt|\nabla_{\a^{1/3}\b^{-1/3}TA}'u\rt|^2+ 
  \a^{-4/3}\b^{-2/3}\lt|(\nabla_{\a^{1/3}\b^{-1/3}TA} u)_3\rt|^2\\
+  \a^{2/3}\b^{-2/3}\left({B_3} -(1-|{u}|^2)\right)^2
+  \b^{-1}|{B'}|^2 dx +\a^{1/3}\b^{7/6}\|\b^{-1}{B_3}-1\|_{H^{-1/2}(x_3=\pm1)}^2\Big],\end{multline*}
with $\nab\times A=B$ (and in particular $\Div B=0$).
We assume that $\tilde L$ is a fixed quantity of order 1. For simplicity of notation, the detailed analysis is done
only for the case $\tilde L=1$.\\
Let us point out that in these units, the penetration length is of order $T^{-1}\a^{1/3}\b^{1/6}$, the coherence length  of order $\a^{-2/3}\b^{1/6}$, 
the width of the normal domains in the bulk of order $\b^{1/2}$ and the distance between the threads  of order one. That is, the scale separation \eqref{scalesepar} reads now
\begin{equation}\label{scalesepar2}
 T^{-1}\a^{1/3}\b^{1/6}\ll \a^{-2/3}\b^{1/6}\ll \b^{1/2}\ll1.
\end{equation}


\section{The intermediate functionals}\label{Sec:intermediate}
In this section we explain the origin of the rescaling leading from $E_T$ to $\tilde E_T$, 
and the different functionals which appear at different scales. This material is not needed
for the proofs but we think it is important to illustrate the meaning of our results.
We carry out the scalings in detail but the relations between the functionals are here discussed only at a heuristic level.

We want to successively  send $T\to+\infty$, $\alpha\to +\infty$ and $\beta\to 0$. For this we are going to introduce a hierarchy of models starting from $E_T(u,A)$ and finishing at $I(\mu)$. 
 When sending first $T\to +\infty$ with fixed $\alpha$ and $\beta$, the functional $E_T$ approximates
\begin{equation*}
 F_{\alpha,\beta}(\rho,B):=\frac{1}{L^2} \lt(\int_{Q_{L,1}} |\nabla \rho^{1/2}|^2 +(B_3-\alpha (1-\rho))^2 +|B'|^2 dx +\|{B_3}- \a\b\|_{H^{-1/2}(x_3=\pm1)}^2\rt),
\end{equation*}
with the constraints 
\begin{equation}\label{meissner4}
 \Div B=0 \qquad\textrm{and } \qquad \rho B=0.
\end{equation} The main difference between $E_T$ and $F_{\alpha,\beta}$ is that for the latter, since the penetration length (which corresponds to $T^{-1}$) was sent to zero, the Meissner condition \eqref{meissner4} is enforced. 
We now want to send the coherence length (of order $\a^{-1}$) to zero at fixed $\beta$, while keeping superconducting domains of finite size. 
Since the typical domain diameter is of order $\a^{-1/3}\b^{1/3}$ and their distance is of order $\a^{-1/3}\b^{-1/6}$, we are led to the anisotropic rescaling:

\begin{center}
\begin{tabular}{ll}
 $\begin{pmatrix}\hat{x}'\\ \hat{x}_3\end{pmatrix}:= \begin{pmatrix}\a^{1/3} x'\\x_3\end{pmatrix}$,  & \hspace{2cm} $\hat{L}:= \a^{1/3}L$,\\[6pt]
 $\begin{pmatrix}\hat{B}'\\ \hat{B}_3\end{pmatrix}(\hat{x}):=\begin{pmatrix}\a^{-2/3} B'\\ \a^{-1}B_3\end{pmatrix}(x)$,  & \hspace{2cm}$\hat{\rho}(\hat{x}):=\rho(x)$,\\[6pt]
 \hspace{1cm}$\wF=\a^{-4/3} F_{\a,\b}$.  &
\end{tabular}\end{center}
In these variables, the coherence length is of order $\a^{-2/3}\ll 1$  (at least horizontally) while the diameter of the normal domains is of order $\beta^{1/3}$ and their separation of order $\beta^{-1/6}$. Dropping the hats (we just keep them on the functional  and on $L$ to avoid confusion) we obtain
\begin{multline*}
 \wF(\rho,B):=\frac{1}{\hL^2} \lt(\int_{Q_{\hL,1}} \a^{-2/3}\left|\begin{pmatrix} \nabla' \rho^{1/2}\\ \a^{-1/3} \partial_3 \rho^{1/2}  
                                                                           \end{pmatrix}\right|^2
 +\a^{2/3}\lt(B_3- (1-\rho)\rt)^2 +|B'|^2 dx\rt.\\
\lt. +\a^{1/3}\|{B_3}- \b\|_{H^{-1/2}(x_3=\pm1)}^2\rt),
\end{multline*}
with the constraints \eqref{meissner4}. The scaling of Theorem \ref{theoCoOtSer}
indicates that $\wF$ behaves as $ \min\{\beta^{2/3},\a^{2/21}\beta \}$ which is of  order  $\beta^{2/3}$ if $\a\gg 1$ and $\beta$ is fixed.
We remark that, letting $\eta:= \alpha^{-1/3}$ and
 $\delta:=\eta^2=\alpha^{-2/3}$, one has
\begin{multline*}
 \wF(\rho,B)= \frac{1}{\hL^2} \left(\int_{Q_{\hL,1}} \delta \lt|\begin{pmatrix}\nabla'\rho^{1/2}\\ \eta \partial_3 \rho^{1/2}\end{pmatrix}\rt|^2+\frac{1}{\delta} \lt(B_3-(1-\rho)\rt)^2 
+|B'|^2 dx \rt.\\
\lt.+ \eta^{-1} \|B_3-\beta\|_{H^{-1/2}(x_3=\pm1)}^2\rt).
\end{multline*}
In this form, $ \wF(\rho,B)$ is very reminiscent of the functional studied in \cite{ViehOtt}. Notice however that besides the Meissner condition which makes 
our functional more rigid, the scaling $\delta=\eta^2$ is borderline for the analysis in \cite{ViehOtt}.

Recalling that $\lt(B_3-(1-\rho)\rt)^2 = \chi_{\rho>0} (1-\rho)^2$, the corresponding term in $\wF$ has the form of a double well-potential, and so in the limit $\alpha\to+\infty$ the functional $\wF$ approximates
\begin{equation*}
 G_\beta(\chi,B'):=\frac{1}{\hL^2} \lt( \int_{Q_{\hL,1}} \frac{4}{3}|D'\chi| +|B'|^2 dx \rt),
\end{equation*}
with the constraints $\chi\in \{0,1\}$, $\chi(\cdot,x_3) \weaklim \beta dx'$ when $x_3\to \pm1$ and
\begin{equation*}
 \partial_3 \chi +\Divp B'=0 \qquad \textrm{and } \qquad \chi B'=B'.
\end{equation*}
This is similar to the simplified sharp-interface functional that was studied in \cite{ChokKoOt,ChokConKoOt}.
In the definition of $G_\beta$, we used the notation
\begin{equation*}
 \int_{Q_{\hL,1}} |D' u|:= \sup_{\stackrel{\xi \in C^{\infty}(Q_{\hL,1}),}{ |\xi|_\infty \le 1}}  \, \int_{Q_{\hL,1}} u \Divp \xi \, dx,
\end{equation*}
for the horizontal $BV$ norm of a function $u\in L^1(Q_{\hL,1})$. By definition it is lower semicontinuous for the $L^1$ convergence and it is not hard to check that if we let $u_{x_3}(x'):= u(x',x_3)$, then
\[\int_{Q_{\hL,1}} |D'u|=\int_{-1}^1 \left(\int_{Q_{\hL}} |D' u_{x_3}|\right) dx_3,\]
where $\int_{Q_{\hL}} |D' u_{x_3}|$ is the usual $BV$ norm of $u_{x_3}$ in $Q_{\hL}$ \cite{AFP}. From this and the usual co-area formula \cite[Th. 3.40]{AFP}, we infer that
\begin{equation}\label{coarea}
 \int_{Q_{\hL,1}} |D'u|=\int_{\R} \int_{-1}^1 \H^1( \partial \{u_{x_3}>s\}) dx_3 ds.
\end{equation}
In \eqref{coarea}, $\partial \{u_{x_3}>s\}$ represents the measure-theoretic boundary of $\{u_{x_3}>s\}$ in $Q_{\hL}$. 

We finally want to send the volume fraction of the normal phase to zero and introduce the last rescaling in $\beta$ for which we let 
\begin{center}
\begin{tabular}{ll}
 $\begin{pmatrix}\tilde{x}'\\ \tilde{x}_3\end{pmatrix}:= \begin{pmatrix}\b^{1/6} x'\\x_3\end{pmatrix}$, & \hspace{2cm}$ 
 \tilde L:= \beta^{1/6}\hL$, \\[8pt]
 $\tB'(\tilde{x}):=\beta^{1/6} B'(x)$ & \hspace{2cm}$\tchi(\tilde{x}):=\beta^{-1}\chi(x)$,\\ [6pt]
$\wG:=\b^{-2/3} G_{\b}$. &
\end{tabular}
\end{center}
After this last rescaling, the domain width is of order $\b^{1/2}\ll1$, the separation between domains of order $1$. We obtain (dropping the tildas again) the order-one functional
\begin{equation*}
 \wG(\chi,B'):= \frac{1}{\widetilde{L}^2}\int_{Q_{\tilde L,1}} \frac{4}{3} \beta^{1/2}|D'\chi| +\chi|B'|^2 dx 
\end{equation*}
under the constraints $\chi\in \{0,\beta^{-1}\}$, $\chi(\cdot,x_3) \weaklim  dx'$ when $x_3\to \pm1$, and
\begin{equation*}
 \partial_3 \chi +\Divp (\chi B')=0 \qquad \textrm{and } \qquad \chi B'=\beta^{-1}B'.
\end{equation*}
This functional  converges to $I(\mu)$ as $\beta\to 0$.

Let us point out that since we are actually passing directly from the functional $\wE$ to $I$ in Theorem \ref{maintheo}, we are covering the whole parameter regime of interest. In particular,
our result looks at first sight stronger than passing first from $ E_T$ to $F_{\a,\b}$, then from $\wF$ to $G_\beta$ and finally from $\wG$ to $I$. However,
because of the Meissner condition, we do not have a proof of density of smooth objects for $F_{\a,\b}$ and $G_\b$. Because of this, we do not obtain the $\Gamma-$convergence of the intermediate functionals (the upper bound is missing).  

\section{The limiting energy}\label{Sec:limiting}

Before proving the $\Gamma$-limit we study the limiting functional $I$ that was mentioned in (\ref{eqdefIintro}) and motivated in the previous section.
We  give here a self-contained treatment of the functional $I$, which is motivated by the analysis discussed above,
and will be crucial in the proofs that follow. However, in this discussion we do not make use of the
relation to the Ginzburg-Landau functional.

\begin{definition}
For $L,T>0$ we denote by $\calA_{L,T}$ the set of  pairs of measures 
$\m\in\calM^+(Q_{L,T})$, $m\in \calM(Q_{L,T};\R^2)$ with $m\ll\mu$, satisfying the continuity equation
\begin{equation}\label{conteq0}
 \pa_3 \m +\Divp m=0 \qquad \textrm{in } Q_{L,T},
\end{equation}
and such that $\m=\m_{x_3}\otimes d{x_3}$ where, for a.e. $x_3\in(-T,T)$,
$\m_{x_3}=\sum_i \phi_i \delta_{X_i}$ for some $\phi_i> 0$ and $X_i\in Q_L$. 
We denote by $\calA_{L,T}^*:=\{\mu: \exists m, (\mu,m)\in\calA_{L,T}\}$ the
set of admissible $\mu$.

We define the functional  $I:\calA_{L,T}\to[0,+\infty]$ by
\begin{equation}\label{Imum}
 I(\mu,m):=
   \frac{K_*}{L^2} \int_{-T}^T  \sum_{x'\in Q_L} \left(\mu_{x_3}(x')\right)^{1/2}  \, dx_3 + 
   \frac{1}{L^2}\int_{Q_{L,T}}
   \left(\frac{dm}{d\m} \right)^2 d\m, 
\end{equation}
where $K_*:= \frac{8 \sqrt{\pi}}{3}$
and (with abuse of notation) $I:\calA_{L,T}^*\to[0,\infty]$ by
\begin{equation}\label{Imu}
I(\m):=\min \{ I(\mu,m)\ : \ m\ll\m, \ \pa_3 \m + {\Divp} m=0\}.
\end{equation}
\end{definition}
Condition (\ref{conteq0}) is understood in a $Q_L$-periodic sense, i.e., for any $\psi\in C^1(\R^3)$ which is
$Q_L$-periodic and vanishes outside $\R^2\times (0,T)$ one has
$\int_{Q_{L,T}} \partial_3\psi d\mu + \nabla'\psi \cdot dm=0$.
If $\mu_{x_3}=\sum_i \phi_i \delta_{X_i}$ then $\sum_{x'\in Q_L} \left(\mu_{x_3}(x')\right)^{1/2}=\sum_i \phi_i^{1/2}$.   Because of  
\eqref{conteq0},  $\mu_{x_3}(Q_L)$ does not depend on $x_3$.

Let us point out that the minimum in (\ref{Imu}) is attained thanks to (\ref{eqexistsB}).
Moreover, the minimizer is unique by strict convexity of $m\to\int_{Q_{L,T}} \left(\frac{dm}{d\m} \right)^2 d\m$. As proven in Lemma \ref{lemmacurves} below, if $\mu$ is made of a finite union of curves then there is actually only one  admissible 
measure $m$ for \eqref{Imu}. More generally, Since every measure $\mu$ with finite energy is rectifiable (see Corollary \ref{coroXia}), we believe that it is actually always the case. For $\mu$ an admissible measure and $z,\tilde{z}\in[-T,T]$, we let 
\begin{equation}\label{Iztz}
 I^{(z,\tilde{z})}(\mu):= \frac{K_*}{L^2} \int_{z}^{\tilde{z}}  \sum_{x'\in Q_L} \left(\mu_{x_3}(x')\right)^{1/2}  \, dx_3 + 
   \frac{1}{L^2}\int_{Q_{L}\times[z,\tilde{z}]}
   \left(\frac{dm}{d\m} \right)^2 d\m, 
\end{equation}
where $m$ is the optimal measure for $\m$ on $[z,\tilde{z}]$ (which coincides with the restriction to $[z,\tilde{z}]$ of the optimal measure on $[-T,T]$).

{}From (\ref{eqw22b}) one immediately deduces 
 for every measure $\mu$, and every $x_3, \tilde x_3\in[-T,T]$, the following estimate on the Wasserstein distance
\begin{equation}\label{HolderW2}
  W_2^2(\mu_{x_3}, \m_{\tilde x_3})\le L^2 I(\mu) |x_3-\tilde x_3|.
\end{equation}
In particular for  every measure $\mu$ with $I(\mu)<+\infty$, the curve $x_3\mapsto \mu_{x_3}$ is H\"older continuous with exponent $1/2$ in $\calM^+( Q_L)$ (endowed with the metric $W_2$) and the traces $\mu_{\pm T}$
 are well defined.

\subsection{Existence of minimizers}
Given two measures $\bar \m_{\pm}$ in $\M^+(Q_L)$ with $\bar \mu_+(Q_L)=\bar \mu_-(Q_L)$, we are interested in the  variational problem
\begin{equation}\label{limitProb}
 \inf\lt\{ I(\mu) \ : \ \mu_{\pm T}= \bar \m_{\pm} \rt\}.
\end{equation}
We first prove that any pair of  measures with equal flux can be connected with finite cost and that there always exists a minimizer.
The construction is a branching construction which gives the expected scaling (see \cite{ChokConKoOt,CoOtSer})
if the boundary data is such that $\bar \m_+=\bar\m_-$.

\def\no{
\begin{remark}\label{remarkbranchuni}
 Later on, we will make use of the uniform branching measures $\mu^{uni}$ constructed in the proof of the previous  proposition for $\bar \mu=\rho dx'$, a probability density with $c\ge \rho\ge \frac{1}{c}$. Notice that in this case, $\Phi_{ij,n}=\int_{Q_{ij,n}} \rho \, dx \sim \frac{1}{4^n}$.
\end{remark}
\begin{remark}\label{remarkscaling}
 By a simple rescaling of the construction leading to $\mu^{uni}$, we see that if $\bar \mu_{\pm T}(Q_L)=\Phi$, we have
\[ \inf\lt\{ I(\mu) \ : \ \mu_{\pm T}= \bar \m_{\pm T} \rt\}\les T\sqrt{\Phi}+\frac{L^2\Phi}{T},\]
In particular, if $ \bar \m_{\pm T}=dx'$, letting $N:=\frac{L}{T^{2/3}}$ and  taking $\m_0$ equal to $N^2$ uniformly distributed  Dirac masses with weight $\frac{L^2}{N^2}$, we find (provided $L\gg T^{2/3}$)
\[\inf I(\m)\les LTN+\frac{L^4}{N^2 T}\les L^2 T^{1/3}\]
which is the expected scaling (see \cite{ChokConKoOt,CoOtSer}).
\end{remark}}

\begin{proposition}\label{branchingmu}
 For every pair of  measures $ \bar \m_{\pm}\in \calM^+(Q_L)$ with $\bar\mu_+(Q_L)=\bar\mu_-(Q_L)=\Phi$, 
 there is  $\mu\in \calA^*_{L,T}$ such that
 $\mu_{\pm T}=\bar\mu_\pm$ and
 \begin{equation*}
  I(\mu)\les  \frac{T\Phi^{1/2}}{L^2}+\frac{\Phi}{T}. 
 \end{equation*}
 If $\bar\mu_+=\bar\mu_-$, then there is a construction with 
 \begin{equation*}
  I(\mu)\les  \frac{T \Phi^{1/2}}{L^2} +\frac{  T^{1/3} \Phi^{2/3} }{L^{4/3}}, 
 \end{equation*}
and such that the slice at $x_3=0$ is  given by $\mu_0=\Phi N^{-2}\sum_j \delta_{X_j}$, with $X_j$ the $N^2$ points in $[-L/2,L/2)^2\cap ((L/N)\Z)^2$, and $N:=\lfloor1+ \Phi^{1/6}L^{2/3}/T^{2/3}\rfloor$.
The measure $\mu$ is supported on countably many segments, which only meet at triple points.
 \end{proposition}
\begin{proof}
By periodicity we can work on $[0,L)^2$ instead of $[-L/2,L/2)^2$.
We first perform the construction for $x_3\ge 0$.
 The idea is to approximate $\bar\mu_+$ by linear combinations of Dirac masses, which become finer and finer as $x_3$
 approaches $T$. Fix $N\in\N$, chosen below. For $n\in\N$, fix 
 $x_{3,n}:=T(1-3^{-n})$, and let
 $T_n:= x_{3,n}-x_{3,n-1}=\frac{2}{3^{n}}T$ be the distance between two consecutive planes. 
At level $x_{3,n}$ we partition  $Q_L$ into squares of side length $L_n:= \frac{L}{2^nN}$. 
More precisely, for $i,j=0,...,2^nN-1 $, we let  $x'_{ij,n}:=\lt(L_n\, i, L_n \, j\rt)$ be a corner of the square $Q_{ij,n}:= x'_{ij,n}+ [0,L_n)^2$, 
and we let $\Phi_{ij,n}:=\bar\mu_+(Q_{ij,n})$ be the flux associated to this square.
  
We define the measures $\mu^\mathrm{br}$ and $m^\mathrm{br}$ (here the suffix $\mathrm{br}$ stands for branching) by
 $$ \mu^\mathrm{br}_{x_3}:= \sum_{ij} \Phi_{ij,n} \delta_{X_{ij,n}(x_3)} \text{ and }
m^\mathrm{br}_{x_3}:= \sum_{ij} \frac{d X_{ij,n}}{dx_3}(x_3) \Phi_{ij,n} \delta_{X_{ij,n}(x_3)} \quad 
\text{for }  x_3\in [x_{3,n-1}, x_{3,n}),$$ 
where $X_{ij,n}:[x_{3,n-1}, x_{3,n}]\to Q_L$ is a piecewise affine function such that
 $X_{ij,n}(x_{3,n})=x'_{ij,n}$, 
 $X_{ij,n}(x_{3,n}-\frac12 T_n)=x'_{i_*j,n}$, 
 and $X_{ij,n}(x_{3,n-1})=x'_{i_*j_*,n}$, where
$i_*=2\lfloor i/2\rfloor$, $j_*=2\lfloor j/2\rfloor$. Four such curves end in every $i_*$, $j_*$ (which corresponds to the pair
$i_*/2$, $j_*/2$ at level $n-1$), but they are pairwise superimposed for $x_3\in [x_{3,n-1},x_{3,n}-\frac12 T_n]$,
therefore all junctions are triple points (one curve goes in, two go out).

Using that $\sum_{i j} \Phi_{ij,n}=\Phi$ and $\sum_{ij}\sqrt{\Phi_{ij,n}} \le
(\sum_{ij} \Phi_{ij,n})^{1/2}(\sum_{ij} 1)^{1/2} = \Phi^{1/2} 2^nN$, 
we get that the energy of $\mu^\mathrm{br}$ is given by
\begin{align*}
 I(\mu^\mathrm{br})&=\frac{1}{L^2} \sum_{n=1}^{+\infty} \sum_{ij} \left( K_* T_n   \sqrt{\Phi_{ij,n}}+ \Phi_{ij,n} T_n \frac{2 L_n^2}{ T_n^2}\rt)\\
& \les  L^{-2} TN\Phi^{1/2} \sum_{n=0}^{+\infty} \lt(\frac{2}{3}\rt)^n + \frac{\Phi}{TN^2}\sum_{n=0}^{+\infty} \lt(\frac{3}{4}\rt)^n\,.
\end{align*}
If we choose $N=1$, then there is only one point in the central plane, $\mu_0=\Phi \delta_0$. Therefore the top and bottom
constructions can be carried out independently, since by assumption the total flux is conserved, and we obtain the first assertion.

If $\bar\mu_+=\bar\mu_-$, we can choose the value of $N$ which makes the energy minimal. Up to constants this is the value given in the statement. Inserting in the estimate above gives the second assertion.
\end{proof}

If the boundary densities are maximally spread, in the sense that they are given by the Lebesgue measure,
the scaling is optimal, as the following lower bound shows.

\begin{proposition}\label{branchinglowerbound}
 For every measure  $\mu\in \calA^*_{L,T}$ such that
 $\mu_{\pm T}=\Phi L^{-2} dx'$ one has
 \begin{equation}\label{lower}
  I(\mu)\ges  \frac{T \Phi^{1/2}}{L^2}+ \frac{ T^{1/3}\Phi^{2/3}}{ L^{4/3}}\, .
 \end{equation}
 \end{proposition}
\begin{proof}
The bound $I(\mu)\ge L^{-2} T\Phi^{1/2}$ follows at once from the subadditivity of the square root. Hence we only need to prove the other one.
 We give two proofs of this bound. The first uses only elementary tools while the second is based on an interpolation inequality.
 \smallskip
 
 {\it First proof:}
Let $I:=I(\mu,v\mu)$, where $v:=dm/d\mu$.  
Fix $\lambda>0$, chosen below.
 Choose $x_3\in (-T,T)$ such that $\mu_{x_3}=\sum_i\varphi_i \delta_{X_i}$ obeys
 \begin{equation}\label{eqlbchoicex3}
  \sum_i \varphi_i^{1/2} \le \frac{L^2 I}{T}\,.
 \end{equation}
 For some set $\calI\subset\N$ to be chosen below, 
  let $\psi:Q_L\to \R$ be a mollification of  the function $\max\{ (\lambda-\dist(x',X_i))_+: i\in \calI\}$,
  where as usual the distance is interpreted periodically.
 By the divergence condition,
 \begin{equation*}
  \int_{Q_L} \psi d\mu_{x_3} = \frac{\Phi}{L^{2}} \int_{Q_L} \psi dx'+\int_{-T}^{x_3} \int_{Q_L} \nabla'\psi\cdot  v d\mu
 \end{equation*}
 (to prove this, pick $\xi_\eps\in C^1_c((-T,x_3))$ which converge pointwise to $1$ and use $\xi_\eps\psi$ as a
 test function in (\ref{conteq0}) and then pass to the limit).
  Since $|\nabla'\psi|\le 1$, 
 \begin{equation*}
    \sum_{i\in \calI} \lambda \varphi_i \le  \frac{\Phi}{L^{2}} \sum_{i\in \calI}\frac{ \pi}{3} \lambda^3+ \int_{-T}^{x_3} \int_{Q_L}  |v| d\mu
    \le \frac{\Phi}{L^2}\sum_{i\in \calI} \pi \lambda^3+  L (T\Phi)^{1/2} I^{1/2},
 \end{equation*}
  where in the second step we used H\"older's inequality and flux conservation. 
  We choose $\calI=\{i\in \N: \varphi_i\ge 4 \Phi \lambda^2/L^2\}$. 
   From the definition of $\calI$, we have 
   \[\frac{\Phi}{L^2}\sum_{i\in \calI} \pi \lambda^3\le \frac{\pi}{4} \sum_{i\in \calI} \lambda \phi_i.  
   \]
Therefore, since $\pi<4$, we obtain 
 \begin{equation}\label{E1}
    \sum_{i\in \calI} \lambda \varphi_i \lsim L(T\Phi I)^{1/2}\,.
 \end{equation}
 At the same time, again by the definition of $\calI$ and (\ref{eqlbchoicex3}), 
  \begin{equation}\label{E2}
  \sum_{i\not\in \calI} \varphi_i\le \frac{2\lambda\Phi^{1/2}}{L} \sum_{i\not\in \calI} \varphi_i^{1/2} \le 2\lambda L \Phi^{1/2}\, \frac{I}{T}\,.
 \end{equation}Adding \eqref{E1} and \eqref{E2}, we obtain
 \begin{equation*}
 \sum_{i\in \calI} \varphi_i \lsim \frac{1}{\lambda} L(T\Phi I)^{1/2}+ \lambda L \Phi^{1/2}\, \frac{I}{T}\, , 
 \end{equation*}
 hence \begin{equation*}
 \Phi^{1/2}\lsim \frac{1}{\lambda} L(T I)^{1/2}+ \lambda L \, \frac{I}{T}\, ,\end{equation*}
 and optimizing over $\lambda$ by choosing $\lambda = \frac{T^{3/4}}{I^{1/4}}$ yields  
  $I\ges \Phi^{2/3} T^{1/3}L^{-4/3}$.
  \smallskip
  
  {\it Second Proof:} As before, let $x_3\in (-T,T)$  be such that $\mu_{x_3}=\sum_i\varphi_i \delta_{X_i}$ obeys \eqref{eqlbchoicex3}.
 By Young's inequality and \eqref{HolderW2}, we have
 \[I\ges  \frac{T}{L^2} \left( \sum_i \varphi_i^{1/2} \right) +\frac{W_2^2(\mu_{x_3}, \Phi L^{-2} dx')}{L^2 T}\ges L^{-2} T^{1/3}\left( \sum_i \varphi_i^{1/2} \right)^{2/3} \left(W_2^2(\mu_{x_3}, \Phi L^{-2} dx')\right)^{1/3}.\]
The desired lower bound would then follow if we can show that for every measure $\mu\in\M^+(Q_L)$ with $\mu=\sum_i \phi_i \delta_{X_i}$ and $\sum_i \phi_i = \Phi$, 
\begin{equation}\label{interpolBVW2}
  \left( \sum_i \varphi_i^{1/2} \right)^{2/3} \left(W_2^2(\mu, \Phi L^{-2} dx')\right)^{1/3}\ges \Phi^{2/3} L^{-4/3}.
\end{equation}
By rescaling it is enough considering $\Phi=L=1$. 
The optimal transport map is necessarily of the form  $\psi(x')=X_i$ if $x'\in E_i$, where
 $E_i$ is a partition of $Q_1$ with $|E_i|=\phi_i$ (the corresponding transport plan is $(Id\times \psi)\sharp dx'$). By definition, it holds
\[W_2^2(\mu, dx')\ge \sum_{i} \int_{E_i} |x'-X_i|^2 dx'.\]
But since $|E_i|=\phi_i=|\B'(X_i,(\phi_i/\pi)^{1/2})|$,
\[\sum_{i} \int_{E_i} |x'-X_i|^2 dx'\ge \sum_{i} \int_{\B(X_i,(\phi_i/\pi)^{1/2})} |x'|^2 dx'\ges \sum_{i} \phi_i^2 .\]
By H\"older's inequality, we conclude that
\[1=\sum_i \phi_i \le \left(\sum_i \phi_i^{1/2}\right)^{2/3} \left(\sum_i \phi_i^2\right)^{1/3}\les \left( \sum_i \varphi_i^{1/2} \right)^{2/3} \left(W_2^2(\mu, dx')\right)^{1/3},\]
as desired.
\end{proof}

\begin{remark}
 The lower bound \eqref{lower}  can also be obtained as a consequence of the scaling law proven 
 in \cite{CoOtSer} 
 for  the  Ginzburg-Landau model combined with our lower bound in Section 
\ref{Sec:gammaliminf}   (which does not use this lower bound). 
 However, since the proof here is much simpler and contains some of the main ideas behind the proofs of \cite{ChokConKoOt,CoOtSer}, we decided to include it.
 Similarly, the interpolation inequality \eqref{interpolBVW2} can be obtained by approximation from a similar inequality proven in \cite{CintiOt} (where it is used in the same spirit as here to re-derive the lower bounds of \cite{ChokConKoOt}). 
\end{remark}
We end this section by proving existence of minimizers.
\begin{proposition}\label{existmu}
 For every pair of measures $ \bar \m_{\pm }$ with $\bar \m_+(Q_L)=\bar \m_-(Q_L)$, the infimum in \eqref{limitProb} is finite and attained.
\end{proposition}
\begin{proof}
In this proof we assume $L=T=1$ and $\bar\mu_+(Q_{1})=1$.
By Proposition \ref{branchingmu} the infimum is finite.
Let now $\mu^n$ be a minimizing sequence for $I$.  Since $\sup_n  I(\mu^n)<+\infty$, thanks to \eqref{HolderW2}, the functions $x_3\mapsto \mu^n_{x_3}$ are equi-continuous in $\cP(Q_1)$ (recall that $W_2$ 
metrizes the weak convergence in $\mathcal{P}(Q_1)$)
 hence by the  Arzel\`a-Ascoli theorem there exists a subsequence, still denoted $\mu^n$, uniformly converging (in $x_3$) to some measure $\mu$ which also satisfies the given
 boundary conditions. Moreover, if $m^n$ is an optimal measure in 
\eqref{Imu} for $\mu^n$,  since by the Cauchy-Schwarz inequality we have
\[\int_{Q_{1,1}} d|m^n|\le \lt(\int_{Q_{1,1}} \lt(\frac{dm^n}{d\mu^n}\rt)^2 d\mu^n\rt)^\hal \lt(\int_{Q_{1,1}} d\mu^n\rt)^\hal\les 1,\]
there also exists  a subsequence $m^n$ converging to some measure $m$ satisfying \eqref{conteq0}. 
By \cite[Th. 2.34 and Ex. 2.36]{AFP} we deduce that  $m\ll \mu$ and
\[\liminf_{n\to +\infty} \int_{Q_{1,1}} \lt(\frac{dm^n}{d\mu^n}\rt)^2 d\mu^n\ge \int_{Q_{1,1}} \lt(\frac{dm}{d\mu}\rt)^2 d\mu.\]
It remains to prove that $\mu_{x_3}=\sum_i \phi_i(x_3) \delta_{X_i(x_3)}$ for a.e. $x_3$ and that 
\begin{equation}\label{semicontperi}
 \liminf_{n\to +\infty} \int_{-1}^1 \sum_{x'\in Q_{1}} \left(\mu^n_{x_3}(x')\right)^{1/2} \, dx_3 \ge \int_{-1}^1 \sum_{x'\in Q_{1}} \left(\mu_{x_3}(x')\right)^{1/2} \, dx_3. 
\end{equation}
 If $\mu^n_{x_3}=\sum \phi_i^n(x_3) \delta_{X^n_i}(x_3)$, with $\phi_i^n$ ordered in a decreasing order, we let $f_n(x_3):= \sum_i \sqrt{\phi_i^n(x_3)}$ and observe that $\int_{-1}^1 f_n dx_3 \le I(\mu^n)\les 1 $. Hence, by Fatou's lemma, 
\begin{equation}\label{fatouf}
1\ges \liminf_{n\to +\infty} \int_{-1}^1 f_n(x_3) dx_3\ge \int_{-1}^1 \liminf_{n\to +\infty} f_n(x_3) dx_3,
\end{equation}
from which we infer that $g(x_3):=\liminf_{n\to+\infty} f_n(x_3)$ is finite for a.e. $x_3$. 
Consider such an $x_3$ and let $\psi(n)$ be a subsequence (which depends on $x_3$) such that $g(x_3)=\lim_{n\to+\infty} f_{\psi(n)}(x_3)$. 
Up to another subsequence, still denoted $\psi(n)$, we may assume that 
for every $i\in \N$, $\phi_i^{\psi(n)}(x_3)$ converges to some $\phi_i(x_3)$ and $X_i^{\psi(n)}(x_3)$ converges to some $X_i(x_3)$.   By Lemma~\ref{lemsqrt} (see below), for every $N\in \N$,
\[
\sum_{i\le N} \phi_i^{\psi(n)}(x_3)\ge 1- \frac{f_{\psi(n)}(x_3)}{\sqrt{N}}.
\]
This implies, by tightness, $\mu_{x_3}^{\psi(n)}\weaklim \sum_i \phi_i(x_3) \delta_{X_i(x_3)}$
and $\sum_i (\phi_i(x_3))^{1/2}\le g(x_3)$.
But since $\mu_{x_3}^{\psi(n)}\weaklim \m_{x_3}$, we have $\m_{x_3}=\sum_i \phi_i(x_3) \delta_{X_i(x_3)}$. 
Finally, by the subadditivity of the square root, \eqref{fatouf} and the definition of $g$ we obtain \eqref{semicontperi}.
\end{proof}

\begin{lemma}\label{lemsqrt}
 If a nonincreasing sequence of  positive numbers
$\gamma_i$ is such that $$ \sum_i \gamma_i = c_0 \qquad \text{ and }\qquad \sum_i
\sqrt{\gamma_i} \le C_0,$$ then for all $N\in \N$ one has
$$\sum_{i\le N} \gamma_i \ge c_0- C_0 \sqrt{\frac{c_0}{N}}.$$
\end{lemma} \begin{proof} Indeed $\sum_{i >N} \gamma_i \le \sqrt{\gamma_N} \sum_{i >N}
\sqrt{\gamma_i} \le C_0 \sqrt{\gamma_N},$ while $c_0 \ge \sum_{i\le
N} \gamma_i \ge N \gamma_N$.\end{proof}

\subsection{Regularity of minimizers}
We now want to prove regularity of the minimizing measures $\mu$. In order to prove that we can restrict our attention to measures containing no loops,
we first define the notion of subsystem.
\begin{proposition} [Existence of subsystems]\label{def:subsystem}
Given a point $x:=(X, x_3) \in Q_{L,T}$ and $\mu\in\calA^*_{L,T}$ with $I(\mu)<+\infty$, there exists a subsystem $\tilde{\m}$ of $\m $ emanating from $x$,
meaning that there exists  $\tilde{\m} \in \calA^*_{L,T}$ such that
\begin{enumerate} 
\item\label{def:subsystempos} $ \tilde{\m}\le \m$  in
the sense that $\m-\tilde{\m}$ is a positive measure,
\item\label{def:subsystemdelta} $\tilde{\m}_{x_3}=a\delta_{X}$, where $a=\m_{x_3}(X)$,
\item\label{def:subsystemconti} \begin{equation*}
\pa_3 \tilde{\m}+ \div' \lt(\frac{dm}{d\m} \tilde{\m}\rt)=0.\end{equation*}
\end{enumerate}
\end{proposition}
In particular, \ref{def:subsystemdelta} implies that $(\m_{x_3} - \tilde{\m}_{x_3}) \perp  \delta_{X}$ in the sense of  the Radon-Nikodym decomposition.

\begin{proof} Let us for notational simplicity assume that $x_3=0$, $L=T=1$, $\mu(Q_{1,1})=2$. Let us denote
$v= \frac{d m}{d\m}$. According to 
\cite[Th. 8.2.1 and (8.2.8)]{AGS},  since $v\in L^2(Q_{1,1},\mu)$, there exists a
positive measure $\sigma$ on $ C^0([-1, 1];Q_1)$
 (endowed with the sup norm), whose disintegration \cite[Th. 5.3.1]{AGS} with respect to $\mu_0$, i.e. $\sigma=\int_{Q_1} \sigma_{x'} d\mu_0(x')$, is made of probability 
 measures $\sigma_{x'}$ concentrated on the set of curves $ \gamma$ solving
$$\left\{ \begin{array}{l}
\dot{\gamma}(x_3)= v(\gamma(x_3))\\
\gamma(0)= x',\end{array}\right.$$
and  such that for every $x_3\in [-1,1]$, $\m_{x_3}= (e_{x_3})_\# \sigma$, where
$e_{x_3}$ denotes the evaluation at $x_3$, in the sense  that
\begin{equation*}
 \int_{Q_1}\varphi d\mu_{x_3} = \int_{  C^0([-1, 1];Q_1)} \varphi(\gamma(x_3)) d\sigma (\gamma)
 \hskip1cm \text{ for all }\varphi\in C^0(Q_1)\,.
\end{equation*}
Then, the  measure $\tilde{\m}=   \tilde{\m}_{x_3}\otimes dx_3$ with $\tilde{\mu}_{x_3}= (e_{x_3})_\# (a \sigma_{X})$, where $a=\mu_{0}(X)$,  satisfies all the required properties.

\end{proof}

\begin{lemma}[No loops] \label{noloop} Let $\m$ be a minimizer for the Dirichlet problem (\ref{limitProb}), $\bar x_3\in (-T,T)$.
Let $x_1=(X_1,\bar{x}_3)$, $x_2=(X_2,\bar{x}_3)$ be two points in the plane $\{x:x_3=\bar{x}_3\}$.  Let $\m_1$ and $\m_2$ be subsystems  of $\m$ emanating from
$x_1$ and $x_2$. Let $x_+:=(X_+,z_+)$ be a point with $ z_+  > \bar{x}_3$ and $x_-:=(X_-,z_-)$ a point with $ z_-<\bar{x}_3$, and such that $\m_{1,z_+}, \m_{2,z_+}$ both have Diracs at $X_+$, 
and $\m_{1,z_-}, \m_{2,z_-}$ both have Diracs at $X_-$ (with nonzero mass). Then $X_1=X_2$.
\end{lemma}\begin{proof} 
 Let $\phi_1:=\mu_{\bar x_3}(X_1)$ be the mass of $\m_1$ and  $\phi_2$ be the mass of $\mu_2$.
Let $\phi_{1, +}:=\m_{1,\bar x_3}(X_+)$ be the mass of $\m_1$ at  $x_+$, $\phi_{2, +}$ the mass of $\m_2$  at $x_+$, $\phi_{1, - } $ the mass of $\m_1$ at $x_-$, $\phi_{2,-}$ the mass of $\m_2$ at $x_-$. 
Let $\phi:= \min \{\phi_{1, +},\phi_{1, -}, \phi_{2, +}, \phi_{2,-}\}$ which by assumption is positive.

We define $\m_{1, +}$ as the subsystem of $\m_1$ coming  from $x_+$, it is thus of mass $\phi_{1, +}$, and at level $\bar{x}_3$ all its mass is at $X_1$ (since it is a subsystem  of $\m_1$ for which this is the case). Similarly with $\m_{1, -}$, $\m_{2, +}, \m_{2,-}$.
 We can now define
  $\tilde{\m}_1:= \frac{\phi}{\phi_{1, +} } \m_{1, +}$ for $x_3\ge\bar{x}_3$  and $\tilde{\m}_1:=\frac{\phi}{\phi_{1, -} }\m_{1, -}$ for $x_3 <\bar{x}_3$, and the same with  $\tilde{\m}_2$.
  The measures $\tilde{\m}_1$ and $\tilde{\m}_2$ are ``systems'' of mass $\phi$ that join $X_-$ and $X_+$. By construction, we have 
  $$\partial_{x_3}  ( \tilde{\m}_1- \tilde{\m}_2 ) \restr[z_-,z_+]+\div' \lt( ( \tilde{\m}_1- \tilde{\m}_2 ) \restr[z_-,z_+]\frac{dm}{d\mu}\rt)=0$$
  and 
  $$-\frac{\varphi}{\min(\varphi_{2,+}, \varphi_{2,-})}\mu \le  ( \tilde{\m}_1- \tilde{\m}_2 ) \restr[z_-,z_+]\le \frac{\varphi}{\min(\varphi_{1,+}, \varphi_{1,-})}\mu.$$
    
 We now define 
$\hat{\m}_\eta:=   \m + \eta ( \tilde{\m}_1- \tilde{\m}_2 ) \restr[z_-,z_+]$, which 
  is admissible for $\eta$ small enough {(and different from $\mu$ unless $X_1=X_2$)},  and evaluate
\begin{multline*}
I(\hat{\m}_\eta)- I(\m) =
 K_*\int_{z_-}^{z_+}
 \sum_{x'\in Q_1}\left( \mu_{x_3} (x') +\eta (\tilde{\mu}_1)_{x_3}(x')-\eta (\tilde{\mu}_2)_{x_3}(x')  \right)^\hal
   -  \sum_{x'\in Q_1}\left( \mu_{x_3} (x') \right)^\hal \, dx_3 \\
+\eta  \int_{Q_1 \times [z_-, z_+] }
     \left(\frac{dm}{d\m} \right)^2
     (   d\tilde{\m}_1- \, d\tilde{\m}_2)\,.
     \end{multline*}
But the function $\eta \mapsto  \sqrt{a+\eta b}$ is strictly concave for $a>0$ and $b\ne0$, therefore
$I(\hat{\m}_\eta)+I(\hat{\m}_{-\eta})<2I(\m)$ for any $\eta\ne 0$, 
a contradiction with the minimality of $\m$.
\end{proof}

A consequence of this lemma is the following. Consider a minimizing measure $\mu$  of $I$. Let $z_-$ and $z_+$ be any two slices and
let $X_-$ be one of the Diracs at slice $z_-$. Let $\tilde{\m}$ be a  subsystem
emanating from $(X_-,z_-)$.  Let $X_+$ be any point in the slice $z_+$
where $\tilde{\m}$ carries mass. Then, there is a unique ``path" connecting
$X_-$ to $X_+$ (otherwise there would be a loop). Since this is true
for any couple of ``sources" in two different planes, this means
that there are at most a countable number of absolutely continuous
curves (absolutely continuous because of the transport term) on which $\m\restr[z_-, z_+]$ is concentrated.  So we  have a representation of the form
\begin{equation}\label{representationmu}
\m= \sum_i \frac{\varphi_i}{\sqrt{1+|\dotX_i|^2}} \, \mathcal{H}^1\restr\Gamma_i   \, ,\end{equation}
 where the sum is countable and 
 $\Gamma_i=\{ (X_i(x_3),x_3) : x_3\in [a_i,b_i]\} $ with $X_i$ absolutely continuous and almost everywhere 
 non overlapping.

Another consequence is that if there are two levels at which $\m$ is a finite sum
of Diracs, then it is the case for all the levels in between.
So, if there is a slice with an infinite number of points, then
either it is also the case for all the slices below or for all the
slices above.

For measures which are concentrated on finitely many curves we obtain a simple representation formula for $I(\mu)$.

\begin{lemma}\label{lemmacurves}
 Let $\mu=\sum_{i=1}^N \frac{\varphi_i}{\sqrt{1+|\dotX_i|^2}} \, \mathcal{H}^1\restr \Gamma_i \in\calA^*_{L,T}$ with $\Gamma_i=\{ (X_i(x_3),x_3) : x_3\in [a_i,b_i]\}$ for some  absolutely continuous  curves $X_i$, almost everywhere non overlapping. 
 Every $\phi_i$ is then constant on $[a_i,b_i]$ and we have conservation of mass. That is, 
 for $x:=(x',x_3)$, letting 
\begin{align*}\calI^+(x)&:=\{ i\in[1,N] \ : \ x_3=b_i, \ X_i(b_i)=x'\},  \\ \calI^-(x)&:=\{ i\in[1,N] \ : \ x_3=a_i, \ X_i(a_i)=x'\},  \end{align*} 
it holds
\begin{equation*}
\sum_{i\in\calI^-(x)} \phi_i=\sum_{i\in\calI^+(x)} \phi_i.
\end{equation*}
  Moreover, $m=\sum_i \frac{\varphi_i}{\sqrt{1+|\dotX_i|^2}}
\dotX_i \, \mathcal{H}^1\restr\Gamma_i$ and 
\begin{equation}\label{Iparticul}
 I(\mu)=\frac{1}{L^2}\sum_i \int_{a_i}^{b_i} K_* \sqrt{\phi_i}+ \phi_i |\dotX_i|^2 dx_3.
\end{equation}
  
\end{lemma}
\begin{proof}
 Let $\bar x=(\bar x',\bar x_3)$ with $\bar x_3 \in (-T,T)$  be such that $\mu_{\bar x_3}(\bar x')\neq 0$. Then, by continuity of the $X_i$'s,  there exist  $\delta>0, \eps>0$  such that every curve $\Gamma_i$ with 
$\Gamma_i\cap\lt( \B_\eps'(\bar x')\times [\bar x_3-\delta,\bar x_3+\delta]\rt) \neq \emptyset$ satisfies $X_i(\bar x_3)=\bar x',$ and such that $\mu\restr (\B'_{2\eps}(\bar x')\backslash \B_\eps'(\bar x'))\times [\bar x_3-\delta,\bar x_3+\delta]=0$
(and thus also $m\restr (\B'_{2\eps}(\bar x')\backslash \B_\eps'(\bar x'))\times [\bar x_3-\delta,\bar x_3+\delta]=0$ since $m\ll \m$). Consider then $\psi_1\in C^\infty_c(\B'_{2\eps}(\bar x'))$ with $\psi_1=1$ in $\B_\eps(\bar x')$ and 
$\psi_2\in C^\infty_c( \bar x_3-\delta,\bar x_3+\delta)$ and test \eqref{conteq0} with $\psi:=\psi_1 \psi_2$ to obtain
\begin{align*}
\int_{\bar x_3-\delta}^{\bar x_3+\delta} \frac{d\psi_2}{dx_3} (x_3)\lt(\sum_{X_i(x_3)\in \B_\eps'(\bar x')} \phi_i\rt) dx_3&=\int_{\bar x_3-\delta}^{\bar x_3+\delta}\int_{\B'_\eps(\bar x')} \frac{d\psi_2}{dx_3}(x_3)\psi_1(x') d\mu\\
&=\int_{Q_{1,1}} \frac{\partial \psi}{\partial x_3} d\mu=-\int_{Q_{1,1}} \nabla'\psi \cdot dm\\
 &=-\int_{\bar x_3-\delta}^{\bar x_3+\delta}\int_{\B_{2\eps}(\bar x')\backslash \B_\eps(x')} \psi_2 \nabla' \psi_1 \cdot dm =0,
\end{align*}
from which  the first two assertions follow. It can be easily checked that this implies that $\bar m:= \sum_i \frac{\varphi_i}{\sqrt{1+|\dotX_i|^2}}
\dotX_i \, \mathcal{H}^1\restr\Gamma_i$ satisfies \eqref{conteq0}. Let $m$ be any other measure satisfying \eqref{conteq0} and let us prove that $\nu:=\bar m -m=0$.
 Since $\Divp \nu_{x_3}=0$, we have for every $\psi\in C^{\infty}(Q_1)$  
\[\sum_i \nabla' \psi(X_i(x_3))\cdot \nu_i(x_3)=0,\]
where $\nu_{x_3}=\sum_i \nu_i (x_3) \delta_{X_i(x_3)}$, from which the claim follows. 
\end{proof}

  We remark that Corollary~\ref{coroXia} below will imply that representation \eqref{representationmu}  holds for every measure $\mu$ with $I(\mu)<+\infty$.

The previous results lead to the following.
\begin{proposition}\label{reg}
A minimizer of the Dirichlet problem (\ref{limitProb}) with boundary conditions
$\bar{\mu}_+= \sum_{i =1}^N \phi_i^+\delta_{X_i^+}$ and $\bar{\mu}_{-} =
\sum_{i=1}^N \phi_i^- \delta_{X_i^-} $ (some $\phi_i$ may be zero)
satisfies \begin{enumerate}
\item
$\mu= 
\sum_{i=1}^M \frac{\varphi_i}{\sqrt{1+|\dotX_i|^2}} \, \mathcal{H}^1\restr\Gamma_i  $  for some $M\in\N$, where  $\Gamma_i=\{ (X_i(x_3),x_3) : x_3\in [a_i,b_i]\}$ are
 disjoint up to the endpoints, and the $X_i$ are absolutely continuous.
\item Each $X_i$ is affine.
\item If $\bar{\mu}_{-}= \bar{\mu}_{+} $ then there exists a symmetric minimizer with respect to the $x_3=0$ plane.\end{enumerate}
    \end{proposition}

\begin{proof}
Let $\mu^{ij}$ be the subsystem emanating from $(X_i^-,-T)$ of the subsystem emanating from $(X_j^+,T)$ of $\mu$, so that $\mu=\sum_{ij}\mu^{ij}$ by $\mu^{ij}_{x_3}\le \mu_{x_3}$ and conservation of mass. By Lemma \ref{noloop} we have
$\mu^{ij}_{x_3}=\phi_{ij}(x_3)\delta_{X^{ij}}(x_3)$ for all $x_3$, otherwise there would be loops. 
By Lemma \ref{lemmacurves}, $\phi_{ij}(x_3)$ does not depend on $x_3$. By 
(\ref{HolderW2}), if $\phi_{ij}>0$ then $X^{ij}$ is absolutely continuous. After a relabeling, (i) is proven.

Assertion (ii) follows from minimizing $I(\m)$ as given by \eqref{Iparticul} with respect to $\dotX_i$. 

Let now $\bar{\mu}_{-}= \bar{\mu}_{+}$. If $I(\mu,(-T,0))\le I(\mu,(0,T))$
we obtain a symmetric minimizer $\hat\mu$ by reflection of $\mu\LL(-T,0)$ across $\{x_3=0\}$,
and analogously in the other case. 
This proves (iii).
\end{proof}

We now show that for symmetric minimizers, at arbitrarily small distance from the
boundary we have a finite number of Diracs. We already know that at
arbitrarily small distance we have a countable number, and then that
we have a representation of $\m$ of the form 
\eqref{representationmu}. Let us point out that we will not use this proposition but rather include it for its own interest.

\begin{proposition}\label{finitebranch} 
Fix $\bar\mu\in \calM^+(Q_L)$. Let $\mu$ be a symmetric minimizer 
of $I$ subject to $\mu_{\pm T}=\bar\mu$. Then for  any
$\delta>0$ sufficiently small, the number of Diracs in each slice $x_3 \in [-T +
\delta T , T- \delta T]$ is $\lsim \delta^{-4}$.
\end{proposition}
\begin{proof}
We may assume $L=T=1$, $\mu(Q_{1,1})=2$.
By symmetry, we need only to consider the interval $[0,1-\delta]$. If $\mu_{1-\delta}=\sum_{i} \p_i \delta_{X_i}$, 
 it suffices to prove that $\phi_i\gsim \delta^{4}$ for every $i$. For the rest of the proof we fix a point  $X_i$ and  
in order to ease notation we write $\phi:=\phi_i$ and $X:= X_i$.   Let  $\tilde{\m}$ be the subsystem emanating from $(X,1-\delta)$. 
Thanks to the symmetry of $\m$ and to the no-loop condition,
$\mu$ and  $\mu-\tilde{\m}$ are disjoint for  $x_3> 1-\delta$.  Indeed, if this was not the case, 
by symmetry they would meet also for $x_3<- 1+\delta$, and there would be a loop, which is excluded by
Lemma \ref{noloop}. Therefore
\begin{equation}\label{secondestimI}
  I(\m)- I(\m- \tilde{\m})\ge \int_{1-\delta}^1 \sum_{x'\in Q_1} \left( \tilde{\m}_{x_3}(x')\right)^\hal  \, dx_3\ge \delta \sqrt{\p},
 \end{equation}
 where in the second step we used subadditivity of the square root.
 
 Let now $\tilde{\mu}_{1}$ be the trace of $\tilde{\m}$ on $x_3={1}$ and
 for $z:=\phi^{1/4}$, let $\hat\mu$ be the symmetric comparison measure constructed as follows: 
\begin{itemize}
\item in $[0, 1-z]$, $\hat\mu_{x_3}:= (1+\frac{\p}{1-\p})(\mu_{x_3}-\tilde{\mu}_{x_3})$ and
\item in   $[1-z,1]$,  $\hat\mu_{x_3}:=\mu_{x_3}-\tilde{\mu}_{x_3}+\nu_{x_3}$ where $\nu$ is a measure connecting $\frac{\p}{1-\p} (\mu-\tilde{\m})_{1-z}$ to $\tilde{\m}_{1}$ constructed in Proposition \ref{branchingmu}, so that (recall \eqref{Iztz})
\begin{equation*}
 I^{(1-z,1)}(\nu)\les z\sqrt{\p}+\frac{\p}{z}\sim \p^{3/4}.
\end{equation*}
\end{itemize}
  Since $\m$ is a minimizer it follows by subadditivity of the energy that, for some universal (but generic) constant $C$,
 \begin{align*}
  I(\mu)\le I(\hat\mu)&\le \lt(1+\frac{\phi}{1-\p}\rt) I(\mu-\tilde{\m})+ C\p^{3/4}\\
 &\le I(\mu-\tilde{\m})+C\p^{3/4}.
 \end{align*} Indeed, $I(\mu-\tilde{\m}) \le I(\mu) \lsim 1$ while $\varphi \ll 1$ without loss of generality.
Recalling \eqref{secondestimI}, we deduce  that $C\p^{3/4}\ge \delta \p^{1/2}$,
which yields the result.
\end{proof}

\begin{definition}\label{defpolygonal}
We say that a measure $\mu$ is polygonal if 
\begin{equation*}
\mu= \sum_i \frac{\varphi_i}{\sqrt{1+|\dotX_i|^2}} \, \mathcal{H}^1\restr\Gamma_i,   
\end{equation*}
 where the sum is countable, $\Gamma_i$ are segments of the form 
 $\Gamma_i=\{ (X_i(x_3),x_3) : x_3\in [a_i,b_i]\}$ disjoint up to the endpoints,
 and for any $z\in (0,T)$ only finitely many segments intersect  $Q_L\times (-z,z)$. We say it is  finite polygonal
 if the total number of segments is finite.
\end{definition}
 For any polygonal measure, the representation formula \eqref{Iparticul} holds.

\begin{lemma}
  Let $\bar \m \in \calM^+(Q_L)$. Then, every symmetric minimizer of $I$  with boundary data $\mu_{\pm T}=\bar \m$ is polygonal.
\end{lemma}
\begin{proof}
 It suffices to show that for any $z\in (0,T)$ the measure $\mu$ is polygonal in $Q_L\times (-z,z)$. 
 By Proposition \ref{finitebranch} the measures $\mu_{z}$ and $\mu_{-z}$ are finite sums of Diracs. Since $\mu$
 is a minimizer, it minimizes $I$ restricted to $(-z,z)$ with boundary data $\mu_{\pm z}$. By Proposition~\ref{reg}
 we conclude.
\end{proof}

\subsection{Density of regular and quantized measures}
In this section we want to prove that when $\bar \mu_{\pm T }=  \Phi L^{-2} dx'$, the set of ``regular" measures is dense in energy.

 \begin{definition}\label{defregular}
  We denote by $\Mreg(Q_{L,T})\subset \calA^*_{L,T}$ the set of regular measures, i.e., of measures $\mu$ such that: 
\begin{itemize}
\item[(i)] The measure $\mu$ is finite polygonal, according to Def. \ref{defpolygonal}.
\item[(ii)] All branching points are triple points. This means that any $x\in Q_{L,T}$ belongs
to the closures of no more than three segments.
\end{itemize}
For $N\in \N$, we say that $\mu$ is $N$-regular,  $\mu\in \MNreg(Q_{L,T})\subset \Mreg(Q_{L,T})$, if in addition
\begin{itemize}
\item[(iii)] The traces obey
$\mu_{T}=\mu_{-T}=\Phi N^{-2}\sum_{j} \delta_{X_j}$,
where the $X_j$ are $N^2$ points on a square grid, spaced by $L/N$, and $\Phi\ge0$.
\end{itemize}
\end{definition} 

We  can now state the main  theorem of this section:
\begin{theorem}\label{theodens}
 For every measure $\mu\in \calA^*_{L,T}$ with $I(\mu)<+\infty$
  and $ \mu_{\pm T}= \Phi L^{-2} dx'$, there exists a sequence of measures $\mu_N
  \in\MNreg$, with $\mu_N \weaklim \mu$ and such that  $\limsup_{N \to +\infty} I(\mu_N)\le I(\mu)$,   $\mu_N(Q_{L,T})=\mu(Q_{L,T})$.
  \end{theorem}
The proof will be based on the following intermediate result.
\begin{lemma}\label{lemmamudensityinside}
 For every measure $\mu\in \calA^*_{L,T}$ with $I(\mu)<+\infty$, and such that the traces 
 $\mu_T$ and $\mu_{-T}$ are finite sums of Diracs,
 there exists a sequence of regular measures $\mu^{(N)}\in \Mreg(Q_{L,T})$ 
 with $\mu^{(N)} \weaklim \mu$, $\mu^{(N)}_{\pm T}=\mu_{\pm T}$, 
 and such that  $\limsup_{N \to +\infty} I(\mu^{(N)})\le I(\mu)$.
\end{lemma}
\begin{proof}
We shall modify $\mu$ in two steps to make it polygonal: first on finitely many
layers, to have finitely many Diracs on each of them, and then in the rest of the volume, using local
minimization. 

Fix $N\in\N$ and $\delta>0$, both chosen later. 
We  choose levels $z_j\in (Tj/N, T(j+1)/N)$, for $j=-N+1,\dots, N-1$,
with the property that $\mu_{z_j}=\sum_{k\in\N} \varphi_{j,k}\delta_{x'_{j,k}}$, 
with $\sum_{k\in \N} \varphi_{j,k}^{1/2}<+\infty$ for every $j$. 
We shall iteratively truncate the measure at these levels so that it is supported on finitely many points; 
for notational simplicity we also define
$z_{\pm N}:=\pm T$. The measures $\mu^j$, $j=-N,\dots, N$ will all satisfy
$\mu^j\ll\mu$ and $I(\mu^j)\le I(\mu)$.

 We start with $\mu^{-N}:=\mu$. In 
order to construct $\mu^{j+1}$ from $\mu^j$, we first choose $K_j$ such that
$\sum_{k\ge K_j} \varphi_{j,k}^{1/2}\le \delta/N$.  Then we define 
$\mu^{j+1}$ as
the sum of the subsystems of $\mu^j$ originating from the points $(x'_{j,k},z_j)$ with $k<K_j$.
Clearly $I(\mu^{j+1})\le I(\mu^j)\le I(\mu)$. At the same time, since $\sum_{k=K_j}^{+\infty} \phi_{j,k} \le \Phi^{1/2} \sum_{k=K_j}^{+\infty} \phi_{j,k}^{1/2}$,
\[\mu^{j+1}_{z_j}(Q_L)
=  \mu^{j}_{z_j}(Q_L)- \sum_{k=K_j}^{+\infty} \phi_{j,k}\ge \mu^{j}_{z_j}(Q_L)-\frac{ \Phi^{1/2} \delta}{N},\]
so that  $|\mu^{j+1}-\mu^j|(Q_{L,T})\le 2\Phi^{1/2}\delta T/N$.
Therefore $|\mu^N-\mu|(Q_{L,T})\les \delta T \Phi^{1/2}$.

We define $\hat\mu^N$ as the minimizer with boundary data $\mu^N_{z_j}$ and
 $\mu^N_{z_{j+1}}$ in each stripe $Q_L\times (z_j, z_{j+1})$. Then
 $I(\hat\mu^N)\le I(\mu^N)\le I(\mu)$, and $\hat\mu^N_{\pm T}=\mu^N_{\pm T}$.

At this point we fix the boundary data. For this, we let $\tilde{\mu}$ be the minimizer on $Q_{L,T}$ with 
boundary data $\tilde{\mu}_\pm:=\mu_{\pm T}-\mu^N_{\pm T}$. These boundary data are finite sums of Diracs,
and their flux is $|\tilde{\mu}_\pm|(Q_L)=|\mu_T-\mu^N_T|(Q_L)\les \delta \Phi^{1/2}$. By Proposition \ref{reg} the minimizer is  finite polygonal,
by Proposition \ref{branchingmu} it has energy no larger than a constant times  $\delta \Phi^{1/2}T^{-1} + \delta^{1/2} T \Phi^{1/4} L^{-2}$.
Finally we set $\mu^{(N)}:=\hat\mu^N+\tilde{\mu}$. Then $\mu^{(N)}_{\pm T}=\mu_{\pm T}$, and
\begin{equation*}
 I(\mu^{(N)})\le I(\hat\mu^N)+I(\tilde{\mu}) \le I(\mu) +  C\left(\delta \Phi^{1/2}T^{-1} + \delta^{1/2} T \Phi^{1/4} L^{-2}\rt) .
\end{equation*}

Up to a small perturbation, we may further assume that all junctions are triple.
We can now choose for instance $\delta=1/N$. It only remains to show that $\mu^{(N)}\weaklim\mu$ as $N\to+\infty$.
Recalling that
$|\mu^N-\mu|(Q_{L,T})+|\tilde{\mu}|(Q_{L,T})\les \delta \Phi^{1/2}T$, we only need to show that $\hat\mu^N-\mu^N\weaklim 0$.

For 
 $x_3\in (z_j,z_{j+1})$ we have by definition of $\hat \mu^N$
 $$W_2(\mu_{x_3}^N, \hat \mu_{x_3}^N) \le W_2(\mu_{x_3}^N, \mu_{z_j}^N) +W_2(\hat \mu_{x_3}^N, \hat\mu_{z_j}^N),$$ and by \eqref{HolderW2}
 $$W_2^2(\mu_{x_3}^N, \mu_{z_j}^N)+W_2^2(\hat \mu_{x_3}^N, \hat \mu_{z_j}^N) \le L^2 (z_{j+1}-z_j) \lt( I(\mu^N)+I(\hat \mu^N)\rt),$$
 so that 
 $$\max_{x_3} W_2^2(\mu_{x_3}^N, \hat \mu_{x_3}^N) \lsim \frac{L^2T}{N}I(\mu).$$
For $x_3\in (-T,T)$, let $\Pi_{x_3}$ be an optimal transport plan from $\mu^N_{x_3}$ to $\hat{\mu}^N_{x_3}$. Considering then the transport plan $\Pi:=\Pi_{x_3}\otimes dx_3$ between
$\mu^N$ and $\hat{\mu}^N$, we get
\begin{equation*}
 W_2^2(\mu^N,\hat{\mu}^N)\le \int_{-T}^T W^2_2(\mu_{x_3}^N,\hat\mu_{x_3}^N) dx_3\les \frac{L^2T^2}{N}I(\mu)
\,,
  \end{equation*}
which yields that indeed $\hat\mu^N-\mu^N\weaklim 0$.
 \end{proof}

\begin{proof}[Proof of Theorem \ref{theodens}]
Let $\eps\in (0,1/4)$, chosen such that it tends to zero as $N\to\infty$.
We define $\hat\mu$ in $Q_{L,(1-2\eps)T}$ as a rescaling by $(1-2\eps)$ in the vertical direction, $\hat\mu_{(1-2\eps)x_3}=\mu_{x_3}$.
An easy computation shows that $I(\hat\mu, -(1-2\eps)T, (1-2\eps)T)\le \frac{1}{1-2\eps} I(\mu)$.
In particular, $\hat\mu_{(1-2\eps)T}=\Phi L^{-2} dx'$.
For $x_3\in ((1-2\eps)T, T)$ we define $\hat\mu$ as the result of Proposition \ref{branchingmu}.
Then we set
$\tilde{\mu}=\hat\mu$ on $(-(1-\eps)T,(1-\eps)T)$ and extend it constant outside, in the sense that
$\tilde{\mu}_{x_3}=\hat\mu_{(1-\eps)T}$ for $x_3\in ((1-\eps)T,T)$, and
the same on the other side. Since $\hat \mu_{(1-\eps)T}$ is the midplane configuration of the branching measure constructed in Proposition \ref{branchingmu},
$\tilde{\mu}_{x_3}$ is a finite sum of Diracs for $|x_3|\ge (1-\eps) T$.
We obtain
\begin{equation*}
 I(\tilde{\mu}) \le \frac{1}{1-2\eps} I(\mu) + C \eps^{1/3} \left(\Phi^{2/3}T^{1/3} L^{-4/3}+\eps^{2/3} \Phi^{1/2} T  L^{-2} \right)\,.
\end{equation*}
By Lemma \ref{lemmamudensityinside} applied to the inner domain
$(-(1-\eps)T,(1-\eps)T)$ there is a finite polygonal measure $\check\mu$ which is close to $\tilde{\mu}$
and has the same boundary data at $x_3=\pm T(1-\eps)$. The measure given by $\check{\mu}$ inside, 
and $\tilde{\mu}$ outside, has the required properties. 
\end{proof}

We now turn to the quantization of the measures.
\begin{definition}
 We say that a regular measure $\mu\in\Mreg(Q_{L,T})\subset\calA_{L,T}^*$ is $k$-quantized, for $k>0$, if for all $(x',z)\in Q_{L,T}$ one has $k\mu_z(\{x'\})\in 2\pi\N$.
\end{definition}
\begin{lemma}\label{lemmaquantize}
  Let $\mu\in \MNreg(Q_{L,T})$ and  $\Phi=\mu_T(Q_L)$. 
For any $k>0$ such that $k\Phi\in 2\pi\N$
  there is a $k$-quantized regular measure
  $\mu^{k}\in \Mreg$ such that 
  $\mu^{k}(Q_{L,T})=\mu(Q_{L,T})$ and
  \begin{equation*}
   \left(1-\frac{ C(\mu)}k\right) \mu \le \mu^k \le  \left(1+\frac{ C(\mu)}k\right)\mu\,.
  \end{equation*}
This implies in particular $\mu^k\ll\mu$, $\mu^k\to\mu$ strongly, $W_2^2(\mu,\mu_k)\les C(\mu)k^{-1}$ and $I(\mu^k)\to I(\mu)$ as $k\to\infty$.
  \end{lemma}
\begin{proof}
The measure $\mu$ consists of finitely many segments, each with a flux. To prove the assertion
it suffices to round up or down the fluxes to integer multiples of $2\pi/k$ without breaking the divergence condition, and without changing the total flux.

Since $\mu\in\MNreg(Q_{L,T})$, we have
$\mu_{T}=\Phi N^{-2}\sum_i \delta_{X_i}$.
We select $\varphi_i^k$ as $2\pi \lfloor k \Phi /(2\pi N^2)\rfloor /k$ 
or  $2\pi \lfloor k \Phi /(2\pi N^2)+1\rfloor /k$ , depending on $i$. 
Precisely, we choose the first value for $i=0$ and then, at each $i$, we choose
the lower one if $\sum_{j<i} (\phi_j^k- \Phi N^{-2})>0$, and the upper one otherwise.
This concludes the definition of $\mu^k_T$.

The fluxes in the interior of the sample are defined by propagating
the rounding. At each point where a bifurcation occurs,
if there is more then one outgoing branch we distribute the rounding as discussed for $\mu_T^k$.
This increases the maximal error by at most $2\pi/k$, at each branching point.
Since $\mu$ is finite polygonal, there is a finite number of branching points, hence
the total error is bounded
by a constant times $1/k$. Precisely
$|\phi_i^k-\phi_i|\le C(\mu)/k$ for any segment $i$.
Since $\phi_i$ only takes finitely many values, 
$|\phi_i^k-\phi_i|\le \phi_i C(\mu)/k$ for any segment $i$, which concludes the proof.
\end{proof}

\subsection{Relation with irrigation problems}\label{Sec:reliri}
The functional $I(\mu)$ bears  similarities with the so-called irrigation problems which have attracted a lot of interest
 (see for instance \cite{Xia, BeCaMo}). Besides their applications to the modeling of communication networks and other branched patterns (see again \cite{BeCaMo} and the references therein),
they have also been recently used in the study of Sobolev spaces   between manifolds  \cite{bethuel}. 
Let us recall their definition and for this,  follow the notation of \cite{BeCaMo}. 
For $E(G)$ a set of oriented straight edges and $\p: E(G)\to (0,+\infty)$  we define the irrigation graph $G$ as the vector measure
\[G:=\sum_{e\in E(G)} \phi(e) \mathbf{e} \, \H^1\restr e \, \]
 where $\mathbf{e}$ is the unit  tangent vector to $e$. For   $\a\in [0,1]$, we then define the Gilbert energy of $G$ by
\[M^{\a}(G):=\sum_{e\in E(G)} \p(e)^\a \H^1(e).\] 
Given two atomic  probability measures $\mu^+=\sum_{i=1}^k a_i \delta_{X_i}$ and $\mu^-=\sum_{j=1}^l b_j \delta_{Y_j}$, we say that $G$ irrigates $(\mu^+,\mu^-)$ 
if  $\div G= \mu^+-\mu^-$ in the sense of distributions (this implies in particular that $G$ satisfies Kirchoff's law).
 If we are now given any two probability measures $(\mu^+,\mu^-)$ and a vector measure $G$, with $\div G=\mu^+-\mu^-$ (sometimes called an irrigation path between $\mu^+$ and $\mu^-$), we define 
\[M^{\a}(G):=\inf \{ \liminf_{i\to +\infty} M^{\a}(G_i)\},\]
where the infimum is taken among all the sequences of irrigation graphs  $G_i$ with $G_i\weaklim G$ in the sense of measures and such that
 $\div G= \mu^+_i-\mu^-_i$ for some atomic measures $\mu^{\pm}_i$ tending to $\mu^{\pm}$. If no such sequence exists then we set $M^\a(G)=+\infty$. 
The irrigation problem then consists in minimizing $M^{\a}(G)$ among all the transport paths $G$ between $\mu^+$ and $\mu^-$. For $\a=0$ this is a 
generalization of the famous Steiner problem while for $\a=1$ it is just the Monge-Kantorovich problem.

Using some powerful rectifiability criterion of B. White,  the following theorem  was proven by Q. Xia \cite{Xia}.
\begin{theorem}\label{theoXia}
 Given $0<\a<1$, any transport path $G$ with $M^{\a}(G)+M^1(G)<+\infty$ is rectifiable in the sense that 
\[G=\phi \tau \, \H^1\restr\Gamma\]               
for some density function $\p$ and some $1-$rectifiable set $\Gamma$ having $\tau$ as tangent vector.
\end{theorem}
For minimal irrigation paths, much more is known about their interior and boundary regularity \cite{BeCaMo}. For instance, as for our functional $I(\mu)$ (see Proposition \ref{reg}), it can also be proven that minimal irrigation paths contain no loops 
and that for $\a> 1-\frac{1}{n}$  (where $n$ is the dimension of the ambient space i.e. $n=3$ for us),  any two probability measures $\mu^{\pm}$ can be irrigated at a finite cost (compare with Proposition \ref{existmu}).

Using Theorem \ref{theoXia} and Lemma \ref{lemmamudensityinside}, we can obtain the following rectifiability result.
\begin{corollary}\label{coroXia}
 Every measure $\mu$ for which $I(\mu)<+\infty$ is rectifiable.
\end{corollary}
\begin{proof}
 Using the construction of Lemma \ref{lemmamudensityinside}, we can find a sequence $\mu^n$ such that $\mu^n\weaklim \mu$, $\limsup_{n\to +\infty} I(\mu^n)\le I(\mu)$ and 
$\mu^n=\sum_{i=1}^N \frac{\phi_i}{\sqrt{1+|\dotX_i|^2}} \H^1\restr \Gamma_i$ for some straight edges $\Gamma_i=\{(x_3,X_i(x_3)) \, :\, x_3\in(a_i,b_i)\}$.
 Letting $\widetilde{\mu}^n:=\sum_{i=1}^N \frac{\phi_i}{\sqrt{1+|\dotX_i|^2}} \begin{pmatrix} \dotX_i\\ 1\end{pmatrix} \H^1\restr \Gamma_i$, we have for $\a\ge\frac{3}{4}$, 
\begin{align*}M^\a(\tilde \mu^n)&=\sum_i \int_{a_i}^{b_i} \phi_i^\alpha \sqrt{1+ |\dotX_i|^2} dx_3\\
&\les \sum_i \int_{a_i}^{b_i} \phi_i (1+|\dotX_i|^2) + \phi_i^{2\alpha-1} dx_3\\
&\les \sum_i \int_{a_i}^{b_i} \phi_i (1+|\dotX_i|^2) + \sqrt{\phi_i} dx_3\les I(\mu^n) +1
\end{align*}
so that $\liminf_{n\to +\infty} M^{\a}(\widetilde{\mu}^n)\les I(\mu)+1<+\infty$ and by Theorem \ref{theoXia},  the claim follows.
\end{proof}
In \cite{OudSant}, an approximation of the functional $M^\a$ in the spirit of the Modica-Mortola \cite{ModMort} approximation of the perimeter was proposed. Even though their proofs and constructions are completely different from ours, 
this approach bears some similarities with our derivation of the functional $I(\m)$ from the Ginzburg-Landau functional $E_T(u,A)$.

\section{Lower bound}\label{Sec:gammaliminf}

In the rest of the paper we consider sequences with
\begin{equation}\label{eqlbassumptcoeff}
T_n\to +\infty, \hskip5mm \alpha_n\to +\infty,\hskip5mm  \beta_n\to 0, \hskip5mm \frac{T_n}{\alpha_n} \to +\infty,\hskip5mm \alpha_n \beta_n^{7/2}\to +\infty\,.
\end{equation}
No constant appearing in the sequel will depend on the specific choice of the sequence. We observe that (\ref{eqlbassumptcoeff}) immediately implies
$\alpha_n\beta_n^2\to+\infty$ and $\alpha_n^2\beta_n\to+\infty$. 
Let us recall that in this proof we set $\tilde L=1$ and that (see (\ref{eqdefwe}))
\begin{multline*}
  \wE(u,A)=\int_{Q_{1,1}} \a^{-2/3}\b^{-1/3}\lt|\nabla_{\a^{1/3}\b^{-1/3}TA}'u\rt|^2+ 
 \a^{-4/3}\b^{-2/3}\lt|(\nabla_{\a^{1/3}\b^{-1/3}TA} u)_3\rt|^2\\
+  \a^{2/3}\b^{-2/3}\left({B_3} -(1-|{u}|^2)\right)^2
+  \b^{-1}|{B'}|^2 dx +\a^{1/3}\b^{7/6}\|\b^{-1}{B_3}-1\|_{H^{-1/2}(x_3=\pm1)}^2.
\end{multline*}
In this section, we prove the following compactness and lower bound result.


\begin{proposition}\label{gammaliminf}
Fix sequences of positive numbers $\alpha_n$, $\beta_n$, $T_n$ such that (\ref{eqlbassumptcoeff})  holds,  and
let  $(u_n, A_n)$ be such that  $\sup_{n} \wE(u_n,A_n)<+\infty$. Then up to a subsequence,  the following holds :
\begin{enumerate}
\item $\b_n^{-1}(1-\rho_n)\weaklim \mu$ for some measure $\mu$, $\b_n^{-1} B'_n\weaklim m$ for some  vector-valued measure $m\ll \m$ satisfying the continuity equation \eqref{conteq0}.
\item For almost every $x_3\in (-1,1)$, there exists some probability measure $\mu_{x_3}$ on $Q_1$ with $\mu=\mu_{x_3}\otimes dx_3$ and such that $\mu_{x_3}\weaklim dx'$ as $x_3\to \pm 1$.
\item For almost every $x_3\in (-1,1)$,
 $\m_{x_3}=\sum_{i\in \calI} \p_i \delta_{X_i}$ with $\calI$ at most countable and $\p_i>0$.
\item One has $(\mu,m)\in \calA_{1,1}$ with
\[\liminf_{n\to +\infty} \wE(u_n,A_n)\ge I(\m,m).\] 
\end{enumerate}
\end{proposition}

 Let us first show that the energy gives a quantitative control on the failure of the Meissner
 condition $\rho B=0$ in a weak sense.

\begin{lemma}\label{lemmameissner}
For every $Q_1$-periodic test function $\psi\in H^1_\mathrm{per}(Q_{1,1})$,
if $\|\rho\|_\infty\le 1$ then
\begin{equation} \label{estimmeissner} 
\lt|\int_{Q_{1,1}} \rho B_3 \psi dx\rt|\les 
\frac{\a^{1/3}\b^{2/3}}{T}\wE(u,A)\|\psi\|_{L^\infty}+
\frac{\beta^{1/2}}{T}\wE(u,A)^\hal   \|\nab'
\psi\|_{L^2}, 
\end{equation}
and, if additionally $\psi(x',\pm1)=0$, for $k=1,2$ and $\alpha^2\beta\ge 1$,
\begin{equation} \label{estimmeissnerprime} 
\lt|\int_{Q_{1,1}} \rho B_k \psi dx\rt|\les 
\frac{\a^{2/3}\b^{5/6}}{T}\wE(u,A)\|\psi\|_{L^\infty}+
\frac{\alpha^{1/3}\beta^{2/3}}{T}\wE(u,A)^\hal   \|\nab
\psi\|_{L^2}.
\end{equation}
Moreover, if $\xi\in H^1_0(-1,1)$ and $\psi$ is a periodic Lipschitz continuous function on $Q_1$ then
\begin{equation} \label{estimmeissnerter} 
\lt|\int_{Q_{1,1}} \rho B'\cdot \nabla' \psi \, \xi dx\rt|\les 
\frac{\a^{2/3}\b^{5/6}}{T}\wE(u,A)\|\xi \nabla'\psi\|_{L^\infty}+
\frac{\b^{1/2}}{T}\wE(u,A)^\hal   \| \partial_3 \xi \nabla'\psi \|_{L^2}.
\end{equation}
\end{lemma}
\begin{proof}
Let $\lambda:=\a^{1/3}\b^{-1/3}T$. For (\ref{estimmeissner}) we use formula \eqref{magic} with $A$ substituted by $\lambda A$, that is  
 $|\nab'_{\lambda A} u|^2 = |\calD_{\lambda A}^3 u|^2 +\rho \lambda B_3+
\nab'\times j'_{\lambda A}$. We integrate against a test function $\psi$,
\begin{eqnarray*}
\lt|\int_{Q_{1,1}} \rho B_3 \psi dx\rt| & = & \frac{\a^{1/3}\b^{2/3}}{T}\lt|\int_{Q_{1,1}}\a^{-2/3}\b^{-1/3} \left(|\nab_{\lambda A}'u|^2 - |\calD_{\lambda A}^3 u|^2 - \nab'\times j'_{\lambda A}\right) \psi dx\rt|\\
& \les &  \frac{\a^{1/3}\b^{2/3}}{T}\lt(\wE(u,A)\|\psi\|_{L^\infty} + \a^{-1/3}\b^{-1/6}\int_{Q_{1,1}} \a^{-1/3}\b^{-1/6}|j'_{\lambda A}| |\nab' \psi|dx\rt)\\
& \les & \frac{\a^{1/3}\b^{2/3}}{T}\lt(\wE(u,A)\|\psi\|_{L^\infty}+ \a^{-1/3}\b^{-1/6}\wE(u,A)^\hal   \|\nab'
\psi\|_{L^2} \rt),
\end{eqnarray*}
{where we have used that $|j'_{\lambda A}|\le |\nabla'_{\lambda A} u| $ in view of the definition \eqref{defj} and the upper bound $\rho\le 1$.}
We obtain \eqref{estimmeissnerprime} similarly : One first checks
from the definition of $\calD_{\lambda A}u$ that
\begin{equation*}
 \left| |\calD^1_{\lambda A} u|^2- |(\nabla_{\lambda A} u)_3|^2 - |(\nabla_{\lambda A} u)_2|^2\right|
 \le 2 |(\nabla_{\lambda A} u)_2| \, |(\nabla_{\lambda A} u)_3|.
\end{equation*}
Testing (\ref{magic2}) with $\psi$ and integrating by parts the term with $j_{\lambda A}$ as above gives
\begin{equation*}
\lambda \lt|\int_{Q_{1,1}} \rho B_k \psi dx\rt|\les 
\int_{Q_{1,1}} 2 |\psi|  \, |(\nabla_{\lambda A} u)_2| \, |(\nabla_{\lambda A} u)_3| 
+ |\nabla \psi| \, |\nabla_{\lambda A} u| dx\,.
\end{equation*}
Estimating 
$ 2|(\nabla_{\lambda A} u)_2| \, |(\nabla_{\lambda A} u)_3| \le  \alpha^{1/3}\beta^{1/6}|(\nabla_{\lambda A} u)_2|^2 +   \alpha^{-1/3}\beta^{-1/6}|(\nabla_{\lambda A} u)_3|^2 $
and $\|\nabla_{\lambda A} u\|_2^2\le \alpha^{4/3}\beta^{2/3} \wE(u,A)$
concludes the proof of (\ref{estimmeissnerprime}).\\
The proof of \eqref{estimmeissnerter} is very similar to the proof of \eqref{estimmeissnerprime}. Arguing as above with $\xi\partial_{k}\psi$ playing the role of $\psi$, we get 
\[
 \lambda \lt|\int_{Q_{1,1}} \rho B'\cdot \nabla'\psi \xi dx\rt|\les 
\int_{Q_{1,1}} 2 |\xi \nabla' \psi |  \, |(\nabla_{\lambda A} u)_2| \, |(\nabla_{\lambda A} u)_3|  dx
+ \lt|\int_{Q_{1,1}} (\nabla \times j_{\lambda A})\cdot \nabla \psi \xi dx\rt|\,.
\]
The first term is estimated exactly as before, while the second one gives after integration by parts of $\nabla \times $  and using $\nabla \times \nabla=0$,
\[
 \lt|\int_{Q_{1,1}} (\nabla \times j_{\lambda A})\cdot \nabla \psi \xi dx\rt|= \lt|\int_{Q_{1,1}} \nabla \xi\cdot  (j_{\lambda A}\times \nabla \psi) dx\rt|\le \|\partial_3 \xi \nabla' \psi\|_{L^2} \lt(\int_{Q_{1,1}} |j'_{\lambda A}|^2\rt)^{1/2},
\]
from which we conclude the proof.
\end{proof}
We now prove that for admissible pairs $(u,A)$ of bounded energy, the corresponding curves $x_3\mapsto \beta^{-1}B_3(\cdot,x_3)$ satisfy a sort of uniform H\"older continuity. 
This is the analog of \eqref{HolderW2} for the limiting energy. 
\begin{lemma}\label{HolderE}
 For every admissible  pair $(u,A)$ with $\|\rho\|_\infty\le 1$,   and  every $x_3,\tilde{x}_3\in(-1,1)$, letting $E:=\wE(u,A)$ it holds (recall \eqref{BLdefin})
 \begin{multline}\label{HolderE2}
   \|\b^{-1}B_3(\cdot,x_3) - \b^{-1}B_3(\cdot,\tilde{x}_3) \|_{BL}\les E^{1/2} |x_3-\tilde{x}_3|^{1/2}
   + \sigma(\alpha,\beta,T)(E^{1/2}+E),
 \end{multline}
where $\sigma(\alpha,\beta,T):= \lt(\frac{\a}{T}\rt)^{1/2} \lt(\a^2\b^{5/2}\rt)^{-1/6}+ (\a^{1/2}\b)^{-1/3} + \lt(\frac{\alpha^{1/3}}{T\b^{5/6}} \rt)^{1/2}$, which goes to zero in the regime \eqref{eqlbassumptcoeff}.
In particular, in that regime, 
if $E\les 1$, for every $x_3,\tilde{x}_3\in(-1,1)$ with $|x_3-\tilde{x}_3|\ge \sigma^{1/2}(\alpha,\b,T)$, there holds
\begin{equation}\label{equiconB}
  \|\beta^{-1}B_3(\cdot,x_3)-\beta^{-1}B_3(\cdot,\tilde{x}_3)\|_{BL}\les |x_3-\tilde{x}_3|^{1/2}.
\end{equation}

\end{lemma}

\begin{proof}
 The proof resembles that of \cite[Lem. 3.13]{CoOtSer}.  First, we show that for every $Q_1-$periodic and Lipschitz continuous function $\psi$ with $\|\psi\|_{Lip}\le 1$,
 \begin{equation}\label{firstestimHolder}
  \lt|\int_{Q_1\times\{x_3\} } \b^{-1}B_3 \psi dx' -\int_{Q_1\times\{\tilde{x}_3\} } \b^{-1}B_3 \psi dx'\rt|\le|x_3-\tilde{x}_3|^{1/2} \b^{-1/2} E^{1/2}.
 \end{equation}
 This follows from $\Div B=0$ and integration by parts, which yields
\begin{align*}
 \lt|\int_{Q_1\times\{x_3\} } \b^{-1}B_3 \psi dx' -\int_{Q_1\times\{\tilde{x}_3\} } \b^{-1}B_3 \psi dx'\rt|&=\lt|\int_{Q_1\times(x_3,\tilde{x}_3)} \b^{-1}\partial_3 B_3 \psi  dx\rt|\\
 &=\lt|\int_{Q_1\times(x_3,\tilde{x}_3)}\beta^{-1} B' \cdot \nabla' \psi  dx\rt|\\
 &\le |x_3-\tilde{x}_3|^{1/2} \beta^{-1/2} \lt(\int_{Q_1\times(x_3,\tilde{x}_3)} \beta^{-1}|B'|^2 dx\rt)^{1/2}\\
 &\le |x_3-\tilde{x}_3|^{1/2} \beta^{-1/2}E^{1/2}.
\end{align*}
For $|x_3-\tilde{x}_3|\le T^{-1} \a^{1/3}\b^{1/6}$, this implies that 
\[\|\b^{-1}B_3(\cdot,x_3) - \b^{-1}B_3(\cdot,\tilde{x}_3) \|_{BL}\le \lt(\a^{1/3} T^{-1} \b^{-5/6}\rt)^{1/2} E^{1/2}\le \sigma(\a,\b,T) E^{1/2},\]
and \eqref{HolderE2} is proven. Letting $\hat{\sigma}(\a,\b,T):= \lt(\frac{\a}{T}\rt)^{1/2} \lt(\a^2\b^{5/2}\rt)^{-1/6}+ (\a^{1/2}\b)^{-1/3}$, we are left to prove that 
 for   $Q_1-$periodic and  Lipschitz continuous $\psi$ with $\|\psi\|_{Lip}\le 1$ and $|x_3-\tilde{x}_3|\ge T^{-1}\a^{1/3}\b^{1/6}$, 
 \begin{equation}\label{toproveHolder}
  \lt|\int_{Q_{1}\times\{x_3\}} \beta^{-1}B_3 \psi dx' -\int_{Q_{1}\times\{\tilde{x}_3\}} \beta^{-1}B_3 \psi dx' \rt|\les |x_3-\tilde{x}_3|^{1/2}E^{1/2}+   \hat{\sigma}(\alpha,\beta,T)(E^{1/2}+E).
 \end{equation}
 Up to translation we may assume $\tilde{x}_3=0$ and $x_3>0$. Let $\delta\le x_3/2$ and define $\xi: \R\to \R$ by
\[
 \xi(z):=\begin{cases}
             \frac{z}{\delta} & \textrm{if } 0<z<\delta\\
             1 & \textrm{if } \delta\le z \le x_3-\delta\\
             \frac{x_3-z}{\delta} & \textrm{if } x_3-\delta\le z\le x_3\\
             0 & \textrm{otherwise}.
            \end{cases}
\]
We then have using again $\Div B=0$ and integration by parts
\begin{equation}\label{integpart}
 \int_{Q_{1,1}} \beta^{-1} B_3 \psi \partial_3\xi dx=-\int_{Q_{1,1}} \b^{-1}\rho  B'\cdot \nabla'\psi \, \xi dx -\int_{Q_{1,1}} \b^{-1}(1-\rho) B'\cdot\nabla'\psi\,  \xi dx.
\end{equation}
The first term on the right-hand side of \eqref{integpart} is estimated by \eqref{estimmeissnerter}. For the second term, we  now estimate
\begin{equation}\label{firsttermright}
 \int_{Q_{1,1}} \b^{-1}(1-\rho)|B'|\xi dx\le \lt(\int_{Q_1\times(0,x_3)} \b^{-1}(1-\rho) dx\rt)^{1/2} \lt(\int_{Q_{1,1}}\b^{-1}(1-\rho)|B'|^2 dx \rt)^{1/2}.
\end{equation}
We rewrite the first factor as 
 \[\int_{Q_1\times(0,x_3)} \b^{-1}(1-\rho) dx=\int_{Q_1\times(0,x_3)} \b^{-1}B_3 dx  +\int_{Q_1\times(0,x_3)} \b^{-1}(B_3-(1-\rho)) dx,\]
from which we obtain
 \[\int_{Q_1\times(0,x_3)} \b^{-1}(1-\rho) dx\le \lt|\int_{Q_1\times(0,x_3)} \b^{-1}B_3 dx \rt| +|x_3|^{1/2} \lt(\int_{Q_{1,1}} \b^{-2} (B_3-(1-\rho))^2 dx\rt)^{1/2}.\]
This allows to make use of
 \[\int_{Q_{1,1}} \b^{-2} (B_3-(1-\rho))^2 dx\le \a^{-2/3}\b^{-4/3} E,\] 
and $\int_{Q_1\times\{z\}} \beta^{-1} B_3 dx' =1$, yielding
 \begin{equation}\label{firstestimrightbis}
  \int_{Q_1\times(0,x_3)} \b^{-1}(1-\rho) dx\le |x_3|+|x_3|^{1/2}\a^{-1/3}\b^{-2/3} E^{1/2}.
 \end{equation}
The second factor in \eqref{firsttermright} is directly estimated by
\[
 \int_{Q_{1,1}}\b^{-1}(1-\rho)|B'|^2 dx\le E,
\]
so that inserting this and \eqref{firstestimrightbis} into \eqref{firsttermright} gives
\begin{equation}\label{firstestimrightter}
\lt| \int_{Q_{1,1}} \b^{-1}(1-\rho) B'\cdot\nabla'\psi \xi dx\rt|\les |x_3|^{1/2} E^{1/2} +|x_3|^{1/4}  \a^{-1/6}\b^{-1/3}E^{3/4}.
\end{equation}

Letting $f(z):=\int_{Q_1\times\{z\}}\beta^{-1} B_3 \psi dx'$, we  thus obtain from \eqref{integpart}, \eqref{firstestimrightter} and \eqref{estimmeissnerter},
\begin{equation}\label{firstestimf}
\lt|\int_0^{x_3} f \partial_3 \xi dz\rt|\les \frac{\a^{2/3}\b^{-1/6}}{T} E +\frac{\a^{1/3}\b^{-1/3}}{T} E^{1/2} \|\partial_3\xi\|_{L^2}+ |x_3|^{1/2} E^{1/2} +|x_3|^{1/4} E^{3/4} \a^{-1/6}\b^{-1/3}.
\end{equation}
Since by definition of $\xi$,  $\int_0^{x_3} f \partial_3\xi dz=\frac{1}{\delta} \int_0^\delta f dz-\frac{1}{\delta}\int_{x_3-\delta}^{x_3} f dz$, we have 
\[
 f(x_3)-f(0)=\int_0^{x_3} f \partial_3 \xi dz  +\frac{1}{\delta}\int_{0}^\delta (f-f(0)) dz+\frac{1}{\delta}\int_{x_3-\delta}^{x_3} (f(x_3)-f) dz,
\]
so that 
\[|f(x_3)-f(0)|\le \lt|\int_0^{x_3} f \partial_3 \xi dz\rt|+\sup_{(0,\delta)}|f-f(0)| +\sup_{(x_3-\delta,x_3)}|f-f(x_3)|.\]
In view of this elementary inequality, the estimates  \eqref{firstestimHolder} and \eqref{firstestimf} combine to 
\begin{multline*}
|f(x_3)-f(0)|\les |x_3|^{1/2} E^{1/2} +\frac{\a^{1/3}\b^{-1/3}}{T} E^{1/2} \delta^{-1/2} + \delta^{1/2}\b^{-1/2}E^{1/2} +\frac{\a^{2/3}\b^{-1/6}}{T} E\\ + |x_3|^{1/4} E^{3/4} \a^{-1/6}\b^{-1/3}.  
\end{multline*}
We now optimize in $\delta $ by  choosing $\delta=T^{-1}\a^{1/3}\b^{1/6}$, which combined with  $\frac{\a^{2/3}\b^{-1/6}}{T}\ll \a^{-1/6}\b^{-1/3}$, yields \eqref{toproveHolder} in the form of 
\[
 |f(x_3)-f(0)|\les |x_3|^{1/2} E^{1/2} + (E^{1/2}+E)\lt( \lt(\frac{\a}{T}\rt)^{1/2} \lt(\a^2\b^{5/2}\rt)^{-1/6}+ \a^{-1/6}\b^{-1/3}\rt). 
\]

\end{proof}
\begin{remark}\label{remarkequi}
 We notice that thanks to the Kantorovich-Rubinstein Theorem \cite[Th. 1.14]{villani}, if $B_3$ is non-negative then we can substitute the Bounded-Lipschitz norm in  \eqref{HolderE2} by a $1-$Wasserstein distance.
 In particular, it would imply that if $(\alpha_n,\beta_n,T_n)$ satisfy \eqref{eqlbassumptcoeff} and if $(u_n,A_n)$ are admissible with $\|\rho_n\|_{\infty}\le 1$, $\wE(u_n,A_n)\les 1$, and $B^n_3\ge 0$,
 then the corresponding curves $x_3\mapsto \beta_n^{-1} B^n_3(\cdot,x_3)$ would be in some sense equi-continuous in the space of probability measures endowed with the Wasserstein metric.
\end{remark}

For $\eps>0$ fixed, we define the following regularization of the  singular double well potential $\chi_{\rho>0}(1-\rho)^2$ :
 \begin{equation}\label{doublewell}
 W_\eps(\rho):=\eta_\eps(\rho)(1-\rho)^2  \ \text{with } \ \eta_\eps(\rho):=\min\{\rho/\eps,1\} 
  \,, 
 \end{equation}
see \eqref{introlowerbound} and  Figure \ref{figweps}.
We next show that the energy controls $W_\eps(\rho)$. Similar ideas have  been used in the context of Bose-Einstein condensates \cite{GolMer}.
\begin{figure}
\centerline{\includegraphics[height=4cm]{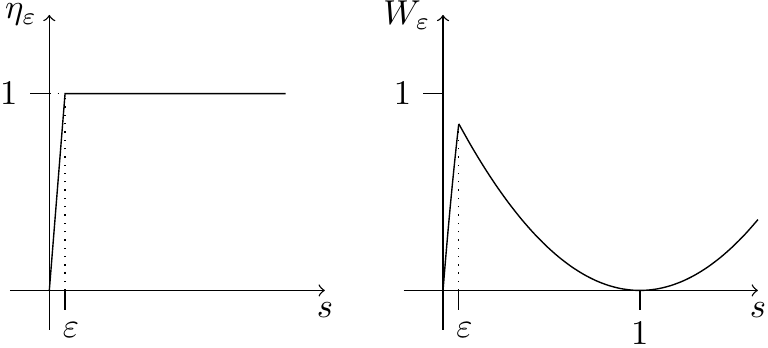}}
\caption{The cutoff function $\eta_\eps$ and the two-well potential $W_\eps$ used in Lemma \ref{lemmafirstlowerbound}.}
\label{figweps}
\end{figure}
\begin{lemma}\label{lemmafirstlowerbound}
 For every $\eps>0$ there exists $C_\eps>0$ such that for every  $(u,A)$ with $\|\rho\|_\infty\le 1$ it holds, 
\begin{equation}\label{estimlower}
 \int_{Q_{1,1}} \a^{2/3}\b^{-2/3}W_\eps(\rho) dx
 \le \int_{Q_{1,1}} \a^{2/3}\b^{-2/3}\left(B_3 -(1-\rho)\right)^2 dx
 +C_\eps \wE(u,A) \frac{\a}{T}.
\end{equation}
\end{lemma}

\begin{proof}
As above, to lighten notation, we let $E:= \wE(u,A)$.  Writing $(1-\rho)= B_3 -(B_3-(1-\rho))$, we obtain by Young's inequality 
\begin{align*}
 (1-\rho)^2&= B_3 (1-\rho)-(B_3-(1-\rho))  (1-\rho)\\
&\le  B_3 (1-\rho) +\frac{1}{2}  (1-\rho)^2+\frac{1 }{2} (B_3-(1-\rho))^2.
\end{align*}
Multiplying by 
 $2\eta_\eps(\rho)$ and  using that $0\le\eta_\eps\le 1$ we obtain for $W_\eps(\rho)$ the estimate
\begin{equation}\label{wer}
  W_\eps(\rho) =\eta_\eps{(\rho)}   (1-\rho)^2\le (B_3-(1-\rho))^2+2\eta_\eps{(\rho)}  (1-\rho)B_3.
\end{equation}
Let $\psi_\eps(s):=2\frac{\eta_\eps(s)}{s}(1-s)= 2\min \{ \frac{1}{\eps}, \frac{1}{s}\} (1-s) $ then $\psi_\eps$  is bounded by $1/\eps$ and is Lipschitz continuous in $s^{1/2} $
with constant of  order  $\eps^{-3/2}$ i.e. $\sup_t |(\psi_\eps(t^2))'|\les \eps^{-3/2}$.
Since 
 $2\eta_\eps {(\rho)} (1-\rho)B_3= \rho B_3 \psi_\eps(\rho)$, using \eqref{estimmeissner} with $\psi=\psi_\eps{(\rho)} $, we get
\begin{align*}
\left| \int_{Q_{1,1}} 2\eta_\eps {(\rho)} (1-\rho)B_3 dx\right|&\les \frac{\a^{1/3}\b^{2/3}}{T}\lt(\eps^{-1} E+  \a^{-1/3}\b^{-1/6}  E^{1/2}  \|\nab' (\psi_\eps(\rho))\|_{L^2} \rt)  \\
&\les \frac{\a^{1/3}\b^{2/3}}{T}\lt(\eps^{-1}E+ \eps^{-3/2} \a^{-1/3}\b^{-1/6}  E^{1/2}  \|\nab'\rho^{1/2}\|_{L^2} \rt) \\
&\les   C_\eps \frac{\a^{1/3}\b^{2/3}}{T}\lt(E+ \int_{Q_{1,1}} \a^{-2/3}\b^{-1/3} |\nabla' \rho^{1/2}|^2 dx\rt) \\
&\les C_\eps \frac{\a^{1/3}\b^{2/3}}{T} E,
\end{align*}
where we used that $|\nabla' \rho^{1/2}|\le |\nabla'_{\a^{1/3}\b^{1/3}TA} u|$ and thus $\int_{Q_{1,1}} \a^{-2/3}\b^{-1/3} |\nabla' \rho^{1/2}|^2 dx\le E$. Estimate  \eqref{estimlower} follows from inserting this estimate into \eqref{wer}.
\end{proof}

 To prove the lower bound, we will need the following two dimensional result.

\begin{lemma}\label{lowerbound2d}
Let $\chi_n\in BV(Q_1, \{0,\b_n^{-1}\})$ be such that $\lim_{n\to +\infty}\int_{Q_1}\chi_n dx'=1$ and 
\[\sup_{n} \int_{Q_1} \b_n^{1/2} |D' \chi_n| <+\infty.\]
 Then, up to a subsequence, $\chi_n\weaklim \sum_i \phi_i \delta_{X_i}$ for some at most countable family of $\phi_i> 0$ and $X_i \in Q_1$, and 
\begin{equation}\label{liminfperibeta}
\liminf_{n\to +\infty}  \int_{Q_1} \b_n^{1/2} |D' \chi_n|\ge 2\sqrt{\pi}\sum_i \sqrt{\phi_i}.
\end{equation}
\end{lemma}

\begin{proof}
\step{1 (Compactness)} 
For each $n$ we split the cube $Q_1$ into
 small cubes of side length
$3\b_n^{1/2}$. 
Let $Q_i^{n}$ be an enumeration of theses cubes such
that
$$\varphi_i^{n} :=\int_{Q_i^{n}} \chi_n dx'$$ 
is nonincreasing in $i$. Since $|\{\chi_n=\b_n^{-1}\}|=\b_n \int_{Q_1} \chi_n dx'=\b_n+o({\beta_n})$, we have $|{Q_i^{n}}\cap\{\chi_n=\b_n^{-1}\}|\le\b_n+o(\b_n)\le \frac{1}{2}|{Q_i^{n}}|$ 
and thus by the relative isoperimetric inequality \cite[Th. 3.46]{AFP}, we have on each ${Q_i^{n}}$
$$\varphi_i^{n}=\int_{{Q_i^{n}}}  \chi_n  dx'\les\left( \int_{{Q_i^{n}}} \b_n^{1/2} |D'
\chi_n|\right)^2\,. $$   It follows
that
$$\sum_i \sqrt{\varphi_i^{n}} \les \int_{Q_1} \b_n^{1/2}|D' \chi_n|  \le C,$$
by the energy bound.
Arguing as in the proof of Proposition~\ref{existmu}, we deduce from Lemma \ref{lemsqrt} that up to extracting a subsequence, $\chi_n \weaklim \sum_{i\in \calI }\phi_i \delta_{X_i}$ for some $\phi_i> 0$ and $X_i\in Q_1$, pairwise distinct.

\medskip

\step{2 (Lower bound)} Assume now that the $\phi_i$ are labeled in a decreasing order and fix $N\in \N$. 
Choose $r\in(0,1/4)$ sufficiently small so that 
\begin{equation}\label{bxr}
\B'(X_i,r)\cap \B'(X_j,r)=\emptyset \qquad \qquad \forall i,j \le N \text{ with } i\neq j, 
\end{equation}
and let  $\psi\in C^\infty_c(\B'(X_i, r);[0,1])$ be a smooth function such that
$\psi=1$ in $\B'(X_i, r/2)$ and $|\nab \psi|\le C/r$.
Let $C_\mathrm{iso}=(2\sqrt{\pi})^{-1}$ be the isoperimetric constant in dimension 2, then  we may write
\begin{eqnarray*}\nonumber
\left(\int_{\B'(X_i, r/2)
} \b_n^{-1} \chi_n dx'\right)^\hal &
 \le &\left( \int_{Q_1} ( \psi \chi_n)^2 dx'\right)^{1/2} 
 \le  C_\mathrm{iso} \int_{Q_1} |D' (\psi \chi_n)|\\
 \nonumber & \le & C_\mathrm{iso} \left( \int_{\B'(X_i, r) }
 |D' \chi_n| + \frac{C}{r} \int_{\B'(X_i,r)\backslash \B'(X_i, r/2)}
  \chi_n dx'\right)\\
   \nonumber & \le & C_\mathrm{iso} \int_{\B'(X_i, r) }
 |D' \chi_n| + \frac{C}{r} 
 \end{eqnarray*}
since $\int_{Q_1} \chi_n dx'\to 1$.
 Multiplying by $\beta_n^{1/2}$ and summing over $N$, we get
 \begin{equation}\label{estimisop}
 \sum_{i=1}^N \left(\int_{\B'(X_i, r/2)
}  \chi_n dx'\right)^{1/2} \le C_\mathrm{iso} \int_{Q_1} \beta_n^{1/2}|D' \chi_n| + \beta_n^{1/2}\frac{CN}{r}.\end{equation}
 
 Next observe that since $\chi_n\rightharpoonup 
 \sum_{i} \phi_i \delta_{X_i}$, we have for every $i=1,\dots,N$,
 \[\liminf_{n\to +\infty}\lt(\int_{\B'(X_i, r/2)}  \chi_n dx'\right)^{1/2}=   \phi_i^{1/2}
\,.\]
Therefore, passing to the limit in \eqref{estimisop}, we obtain
\[\sum_{i=1}^N   \phi^{1/2}_i\le  C_\mathrm{iso} \liminf_{n \to +\infty} \int_{Q_1} \beta_n^{1/2}|D' \chi_n|.\]
Since $N$ was arbitrary this implies \eqref{liminfperibeta}.

\end{proof}

With this lemma at hand, we can prove the compactness and lower bound result.
\begin{proof}[Proof of Proposition \ref{gammaliminf}]
We fix for the proof a sequence $(u_n,A_n)$ with $\wE(u_n,A_n)\les 1$. We then let $B_n:=(B'_n,B_3^n):=\nabla \times A_n$ and $\rho_n:=|u_n|^2$.\\

\step{1 (Compactness)} Notice first that 
\begin{equation}\label{bounddoublewell}
 \int_{Q_{1,1}} \lt( \frac{B^n_3-(1-\rho_n)}
 {\beta_n} 
 \rt)^2 dx\le \a_n^{-2/3} \b_n^{-4/3} \wE(u_n,A_n)
 \to0\,.
\end{equation}
By   \cite[Lem. 3.7]{CoOtSer} there is $\hat u_n$ with $\hat\rho_n:=|\hat u_n|^2=\min\{\rho_n,1\}$ such that $\wE(\hat u_n,A_n)\le (1+2\alpha_n/T_n)\wE(u_n,A_n)$ (the error comes from the last two terms
in (\ref{eqEGLE})).
In particular, also $ (B^n_3-(1-\hat\rho_n))/\beta_n\to0$ in $L^2(Q_{1,1})$.
   Using this and  $|B^n_3|\le |B^n_3-(1-\hat\rho_n)|$ on $\{ B^n_3\le 0\}$ we obtain
\begin{equation}\label{convneg}\lt|\int_{\{B^n_3<0\}} \b_n^{-1} B^n_3 dx\rt|\les \lt( \b_n^{-2} \int_{Q_{1,1}} (B^n_3-(1-\hat\rho_n))^2 dx\rt)^\hal   \le \frac{1}{(\alpha_n\beta_n^2)^{2/3} } \wE(u_n,A_n)   \to 0,\end{equation}
 and since $\int_{Q_{1,1}} \b_n^{-1}B^n_3 dx=2$ the sequence $\b_n^{-1}B^n_3$ is bounded in $L^1$
 and, after extracting  a subsequence, $\b_n^{-1} B^n_3 \weaklim \mu$ for some measure $\mu$.  From \eqref{bounddoublewell}  we also get $\b_n^{-1}(1-\rho_n)\weaklim \mu$, and the same for $\hat\rho_n$. 
  It also follows from \eqref{convneg} that 
$$\int_{Q_{1,1}} \b_n^{-1} (1- \hat\rho_n) dx= 2+\left|\beta_n^{-1} \int_{Q_{1,1}} \left(B_3^n-(1-\hat \rho_n) \right) dx\right|\to 2.$$
Moreover, since $\int_{Q_{1,1}} \b_n^{-1}(1-\hat\rho_n) |B_n'|^2 dx\le\int_{Q_{1,1}} \b_n^{-1} |B'_n|^2 dx\le \wE(u_n,A_n)$ it holds
\[\int_{Q_{1,1}} \b_n^{-1}(1-\hat\rho_n) | B_n'| dx\le \lt(\int_{Q_{1,1}} \b_n^{-1}(1-\hat\rho_n) |B_n'|^2 dx\rt)^\hal \lt(\int_{Q_{1,1}} \b_n^{-1}(1-\hat\rho_n) dx\rt)^{\hal}\les 1,\] 
thus (up to a subsequence) $\b_n^{-1}(1-\hat\rho_n) B_n'\weaklim m$ for some vector-valued measure $m$. By 
\cite[Th. 2.34]{AFP}, 
\begin{equation}\label{liminftrans}
 \liminf_{n\to +\infty} \int_{Q_{1,1}} \b_n^{-1}|B_n'|^2 dx\ge\liminf_{n\to +\infty} \int_{Q_{1,1}} \b_n^{-1}(1-\hat\rho_n)|B_n'|^2 dx\ge \int_{Q_{1,1}} \lt(\frac{dm}{d\m}\rt)^2 d\mu,
\end{equation}
and $m\ll \mu$.
Moreover, from Lemma \ref{lemmameissner} we have that $\b_n^{-1}\hat\rho_n B'_n\weaklim 0$ in a distributional sense and therefore $\b^{-1}_n B_n'$  itself converges to $m$. Letting $n\to +\infty$ in 
$$ \b_n^{-1} \Div B_n=\partial_3 \left[\b_n^{-1} B^n_3\right]+\Divp \lt[\b_n^{-1}B_n'\rt]=0,$$ 
 we obtain $\partial_3 \mu +\Divp m=0$. This proves (i). \\
 
 We now prove that $\mu=\mu_{x_3}\otimes dx_3$, that $\beta^{-1}_nB^n_3(\cdot,x_3)\weaklim \mu_{x_3}$ for a.e. $x_3\in(-1,1)$ and that $\mu_{x_3}\weaklim dx'$ as $x_3\to \pm 1$. 
By \eqref{convneg}  we have that, up to a subsequence in $n$, for a.e. $x_3\in (-1,1)$, 
\[
 \lim_{n\to +\infty} \int_{Q_1\cap \{B^n_3<0\}} \beta^{-1}_n B^n_3 dx=0.
\]
Let $\mathcal{G}\subset (-1,1)$ be the set of $x_3$ for which this hold. For every $x_3\in \mathcal{G}$, the $L^1$ norm of $\beta^{-1}_nB^n_3(\cdot,x_3)$ is bounded thus if we fix a countable dense set $\mathcal{G}_d\subset \mathcal{G}$, we can assume up to extraction that for every $x_3\in \mathcal{G}_d$, 
$\beta^{-1}_n B^n_3(\cdot,x_3)\weaklim \nu_{x_3}$ for some probability measure $\nu_{x_3}$. For $x_3,\tilde{x}_3\in \mathcal{G}_d$, thanks to the weak lower semi-continuity of $\|\cdot \|_{BL}$ and \eqref{HolderE2},
\begin{align*}
 \|\nu_{x_3}-\nu_{\tilde{x}_3}\|_{BL}&\le \limsup_{n\to +\infty} \|\beta_n^{-1} B_3^n (\cdot,   x_3) - \beta_n^{-1} B_3^n(\cdot, \tilde{x}_3)\|_{BL} \\
 &\les \lim_{n\to +\infty}\lt( |x_3-\tilde{x}_3|^{1/2} +\sigma(\alpha_n,\beta_n,T_n)\rt) \\
 &= |x_3-\tilde{x}_3|^{1/2}. 
\end{align*}
Therefore, there exists a unique H\"older-continuous extension of $\nu_{x_3}$ to $(-1,1) \ni x_3 \mapsto \nu_{x_3} \in \mathcal{P}(Q_1)$. We claim that $\mu=\nu_{x_3}\otimes dx_3$.  For $K\to +\infty$,  
let $\{z^K_j\}_{j=1}^{K+1} \in \mathcal{G}_d$ be an increasing sequence such that $|z^K_j-z^K_{j+1}|\les K^{-1}$, $|z_1^K+1|\les K^{-1}$ and $|z_{K+1}^K-1|\les K^{-1}$. Notice that for $n$ large enough,
we have for every $j$ that $|z^K_j-z^K_{j+1}|\ge \sigma^{1/2}(\a_n,\b_n,T_n)$ (where $\sigma$ is defined in Lemma \ref{HolderE}) so that \eqref{equiconB} applies.  Let $\psi$ be a $Q_1-$periodic and Lipschitz continuous function on $Q_{1,1}$ with $\|\psi\|_{Lip}\le 1$. By the continuity of $x_3 \mapsto  \nu_{x_3}$,
we have 
\begin{align*}
 \int_{Q_{1,1}} \psi (d\mu- d\nu_{x_3}\otimes dx_3)&=\lim_{K\to +\infty}  \int_{Q_{1,1}} \psi d\mu-\sum_{j=1}^K  (z^K_{j+1}-z^K_j) \int_{Q_1} \psi(\cdot ,z_j^K)  d\nu_{z_j^K}\\
 &=\lim_{K\to +\infty} \lim_{n\to +\infty} \sum_{j=1}^K \int_{z_j^K}^{z_{j+1}^K} \int_{Q_1} \psi(x',x_3) \beta_n^{-1} B^n_3(x',x_3)\\
 &\qquad  \qquad \qquad \qquad \qquad  -\psi(x',z_j^K)\beta_n^{-1} B^n_3(x',z_j^K) dx' dx_3.
\end{align*}
Using the finite difference version of Leibniz' rule, 
\begin{multline*}\psi(x',x_3)  B^n_3(x',x_3)-\psi(x',z_j^K) B^n_3(x',z_j^K)\\
=   B^n_3(x',x_3)(\psi(x',x_3)-\psi(x', z_j^K))+  \psi(x',z_j^K)( B^n_3(x',x_3) - B^n_3(x',z_j^K))\end{multline*}
and using that $\|\psi\|_{Lip}\le 1$, we can estimate for fixed $K,j$ and $n$ large enough, 
\begin{align*}
\lt| \int_{z_j^K}^{z_{j+1}^K} \rt.& \lt.\int_{Q_1} \psi(x',x_3) \beta_n^{-1} B^n_3(x',x_3)-\psi(x',z_j^K)\beta_n^{-1} B^n_3(x',z_j^K) dx' dx_3\rt|\\
&\le  \int_{z_j^K}^{z_{j+1}^K} \int_{Q_1} \beta_n^{-1} |B^n_3| |x_3-z^K_j| dx + \int_{z_j^K}^{z_{j+1}^K} \|\beta^{-1}_n( B^n_3(x',x_3) - B^n_3(x',z_j^K))\|_{BL}\\
&\les K^{-1} \int_{z_j^K}^{z_{j+1}^K} \int_{Q_1} \beta_n^{-1} |B^n_3|dx +K^{-1} K^{-1/2},
\end{align*}
where in the last line we have used that $|x_3-z^K_j|\les K^{-1}$ and \eqref{equiconB}. Summing this estimate over $j$, we obtain
 \[\left|\int_{Q_{1,1}} \psi (d\mu- d\nu_{x_3}\otimes dx_3)\right| \les \lim_{K\to +\infty} K^{-1}\lt[\lim_{n\to +\infty} \int_{Q_{1,1}}\beta_n^{-1} |B^n_3| dx + K^{-1/2}\rt]=0.\]
This establishes that $\mu= \nu_{x_3}\otimes dx_3$. Moreover, this proves that for every $x_3\in \mathcal{G}_d$,  the whole sequence $\beta_n^{-1} B^n_3(\cdot,x_3)$ weakly converges to $\mu_{x_3}$.
Since the set $\mathcal{G}_d$ was arbitrary, this proves the above convergence for all $x_3\in \mathcal{G}$. \\

We finally show that the boundary conditions hold. For this we focus on $x_3=1$.  For $x_3\in \mathcal{G}$, it holds by the weak lower semi-continuity of $\|\cdot \|_{BL}$,
\begin{align*}
 \|1-\mu_{x_3}\|_{BL}&\le \liminf_{n\to +\infty} \|1-\beta_n^{-1}B_3^n(\cdot, x_3)\|_{BL}\\
 & \le \liminf_{n\to +\infty} \|1-\beta_n^{-1}B_3^n(\cdot, 1)\|_{BL}+\limsup_{n\to +\infty} \|\beta_n^{-1}B_3^n(\cdot, 1)-\beta_n^{-1}B_3^n(\cdot, x_3)\|_{BL}.
\end{align*}
By \eqref{HolderE2}, the second right-hand side term is controlled by $|1-x_3|^{1/2}$. For the first right-hand side term we note that because of $\|\psi\|_{H^{1/2}}\lsim \|\psi\|_{Lip}$, we have 
$$\|\beta_n^{-1}B_3^n (\cdot, 1)- 1\|_{BL}^2  \lsim \|\beta_n^{-1} B_3^n (\cdot, 1)-1\|_{H^{1/2}}^2 \le \alpha_n^{-1/3}\beta_n^{-7/6} \wE(\hat u_n, A_n).$$
Hence, it is the last assumption in \eqref{eqlbassumptcoeff} that ensures that this term vanishes in the limit $n\to +\infty$. We thus obtain the desired estimate
$$\| 1- \mu_{x_3}\|_{BL}\lsim |1-x_3|^{1/2}.$$

 \medskip
\step{2 (Lower bound and structure of $\mu$)} 
The starting point is an application of 
 the usual Modica-Mortola trick.
 In this step we only deal with $\hat u_n$ and $\hat\rho_n$, and drop the hats for brevity.
By \eqref{estimlower} and $|\nabla' \rho^{1/2}|\le |\nabla'_{\lambda A} u|$ we obtain from the Cauchy-Schwarz inequality :
\begin{equation}\label{estimMM}
 \left(1 + C_\eps \frac{\alpha_n}{T_n}\right)\wE(u_n,A_n)\ge \int_{Q_{1,1}} \b_n^{-1/2}2\sqrt{W_\eps(\rho_n)} |\nabla' \rho_n^{1/2}| +\b_n^{-1}|B_n'|^2 dx,
\end{equation}
for any $\eps>0$. We momentarily fix a small  $\delta>0$ and estimate by the co-area formula \eqref{coarea},
\begin{align*}
\int_{Q_{1,1}} 2\sqrt{W_\eps(\rho_n)}|\nabla' \rho_n^{1/2}| dx&\ge \int_{\delta}^{1-\delta} \int_{-1}^1  2\sqrt{W_\eps(s^2)} \H^1(\partial \{\rho_n(\cdot,x_3)>s^2\}) dx_3 ds.
\end{align*}
In particular there exists $s_n\in[\delta,1-\delta]$ depending on $n$  such that
\begin{multline*}\int_{\delta}^{1-\delta} 2\sqrt{W_\eps(s^2)} \int_{-1}^1  \H^1(\partial \{\rho_n(\cdot,x_3)>s^2\}) dx_3 ds\\
 \ge  \lt(\int_{\delta}^{1-\delta} 2\sqrt{W_\eps(s^2)} ds\rt)\int_{-1}^1  \H^1(\partial \{\rho_n(\cdot,x_3)>s^2_n\}) dx_3.
\end{multline*}
Letting $\chi_n(x',x_3):=\b_n^{-1}(1-\chi_{\{\rho_n(\cdot,x_3)>s_n^2\}}(x'))$ this reads
\begin{equation}\label{boundBV}
 \int_{Q_{1,1}} \b_n^{-1/2}2\sqrt{W_\eps(\rho_n)}|\nabla' \rho_n^{1/2}| dx\ge   C_{\delta,\eps}\int_{Q_{1,1}} \b_n^{1/2} |D' \chi_n| dx_3,
\end{equation}
where $C_{\delta,\eps}:=\int_{\delta}^{1-\delta} 2\sqrt{W_\eps(s^2)} ds$.  

Let $\gamma_n\to 0$ to be chosen later. For $n$ large enough, if $\rho_n\le \gamma_n$, then $\chi_n=\beta_n^{-1}$ while if $\rho_n\ge 1-\gamma_n$, $\chi_n=0$ so that
\begin{align*}\int_{Q_{1,1}} |\chi_n-\beta_n^{-1}(1-\rho_n)| dx&\le\beta_n^{-1} \int_{\{\rho_n\le \gamma_n\}} \rho_n dx  +\beta_n^{-1} \int_{\{\rho_n\ge 1-\gamma_n\}} (1-\rho_n) dx \\
& \qquad +\beta_n^{-1}|\{\gamma_n<\rho_n<1-\gamma_n\}|
\\ &\le 2\beta_n^{-1}\gamma_n +\beta_n^{-1}|\{\gamma_n<\rho_n<1-\gamma_n\}|.\end{align*}
By definition of $W_\eps$ (recall \eqref{doublewell}, $\min_{[\gamma_n,1-\gamma_n]} W_\eps=\min(\frac{\gamma_n}{\eps}, \gamma_n^2)=\gamma_n^2$ so that using that $\int_{Q_{1,1}} \alpha_n^{2/3}\beta_n^{-2/3} W_\eps(\rho_n) dx\les 1$,
\[
\beta_n^{-1}|\{\gamma_n<\rho_n<1-\gamma_n\}|\le \beta_n^{-1} \gamma_n^{-2} \int_{Q_{1,1}} W_\eps(\rho_n) dx \les \beta^{-1/3}_n\gamma_n^{-2}\alpha_n^{-2/3}.
\]
 Therefore, if we choose $\gamma_n$ such that $\beta_n\gg \gamma_n\gg \alpha_n^{-1/3}\beta_n^{-1/6}$, which is possible since by hypothesis $\alpha_n \beta_n^{7/2}\to +\infty$, we obtain that
\[\lim_{n\to +\infty}\int_{Q_{1,1}} |\chi_n-\beta_n^{-1}(1-\rho_n)| dx=0.\]
Combining this with \eqref{bounddoublewell}, we obtain that 
\[
 \lim_{n\to +\infty}  \int_{Q_{1,1}}|\chi_n-\b_n^{-1}B_3^n| dx=0.
\]
By Fubini, this implies that, after passing to a subsequence in $n$, for a.e. $x_3\in (-1,1)$, if $ \b_n^{-1}B_3^n(\cdot,x_3)\weaklim \mu_{x_3}$ then also $\chi_n(\cdot,x_3)\weaklim  \mu_{x_3}$.
  Moreover, from $\int_{Q_1} \b_n^{-1} B^n_3(x',x_3) dx'=1$ for a.e. $x_3\in(-1,1)$,
 we obtain that $\lim_{n\to +\infty}\int_{Q_1} \chi_n(x',x_3) dx'=1$ for a.e. $x_3\in(-1,1)$. We thus can   use Lemma \ref{lowerbound2d}  to prove that $\mu_{x_3}=\sum_i \p_i \delta_{X_i}$ for some $\p_i>0$ and 
\begin{equation}\label{lowerboundperim}
 \liminf_{n\to +\infty} \int_{Q_{1,1}} \b_n^{1/2} |D'\chi_n| dx_3 \ge 2\sqrt{\pi} \int_{-1}^1 \sum_i \sqrt{\p_i} dx_3\, ,
\end{equation} where we used Fatou lemma.
 This shows (iii). Putting \eqref{estimMM}, \eqref{boundBV}, \eqref{lowerboundperim} and \eqref{liminftrans}  together we find
\[\liminf_{n\to +\infty} \wE(u_n,A_n)\ge \int_{-1}^1 2\sqrt{\pi} C_{\delta,\eps} \sum_i \sqrt{\phi_i} dx_3 +\int_{Q_{1,1}} \lt(\frac{dm}{d\mu}\rt)^2 d\mu,\]
for any $\eps$ and $\delta$. Since 
\[\lim_{\eps\to0}\lim_{\delta\to0}C_{\delta,\eps}= 2\int_0^1 (1-t^2) dt=\frac{4}{3},\]
and $K_*=\frac{8\sqrt{\pi}}{3}$, this concludes the proof of (iv).
\end{proof}

\section{Upper bound}\label{Sec:upperbound}
In this section we construct a recovery sequence for any  sequences $T_n, \alpha_n, \beta_n$
which obey (\ref{eqlbassumptcoeff}) and additionally the condition of quantization  of the total flux
\begin{equation}\label{defKtotal} 
 L_n^2 T_n\a_n \b_n\in 2\pi \N\,,
\end{equation}
where $L_n:=\tilde L \a^{-1/3}_n \b_n^{-1/6}$. Recalling  the 
form of the gradient term in the functional $\wE$ defined in (\ref{eqdefwe}),
to discuss quantization of the flux of individual domains it is convenient to introduce
\begin{equation}\label{defK} 
k_n:=\alpha_n^{1/3}\beta_n^{2/3}T_n\,.
\end{equation}
The global flux quantization (\ref{defKtotal})  then reads $k_n \tilde L^ 2\in 2\pi\N$ and, in the $\tilde L=1$ case we are considering here,
simplifies to $k_n\in 2\pi \N$.
Condition (\ref{eqlbassumptcoeff}) implies $k_n\to +\infty$ and in particular $k_n/\beta_n\to+\infty$, so that the quantization condition becomes
less and less stringent with increasing $n$. Aim of this section is to prove the following:
\begin{proposition}\label{gammalimsup-v2}
Assume $\tilde L=1$, (\ref{eqlbassumptcoeff})  and (\ref{defKtotal}). Then,
 for every $\mu$ with $I(\mu)<+\infty$ and $\mu_1=\mu_{-1}=dx'$ there exist sequences $u_n: Q_{1,1}\to \C$ and $A_n: Q_{1,1}\to \R^3$ such that  
\begin{equation*}
\limsup_{n\to +\infty} \,  \wE(u_n,A_n)\le I(\mu).
\end{equation*}
The fields $\rho_n:=|u_n|^2$ and $B_n:= \nabla \times A_n$ are $Q_1$-periodic and it holds
\[ \b_n^{-1}(1-\rho_n)\weaklim \mu \qquad \textrm{and } \qquad \b_n^{-1} B_n'\weaklim m,\]
where $m$ is the measure such that $I(\mu)=I(\m,m)$. 
\end{proposition}
The idea of the construction is to use the density result Theorem \ref{theodens} to separate the construction in two regions. In the bulk, the measure will be approximated by a finite
polygonal measure for which the construction is made in Section \ref{secconstrbulk}.
In the boundary layer, we plug in the construction of \cite{CoOtSer}, see Section \ref{secconstrboundary}, 
which is optimal up to a factor. Since the energy in the boundary
layer is small, its suboptimal effect disappears in  the limit.

We shall first construct 
the density $\rho$ and the magnetic field $B$. The appropriate energy is 
$\tilde F_{\alpha,\beta}:=\tilde F_{\alpha,\beta}^{(-T,T)}+\tilde F_{\alpha,\beta}^{\Ext}$, where
\begin{multline}\label{eqdeftildaF}
  \tilde F_{\alpha,\beta}^{(a,b)}(\rho,B):=\int_{Q_1\times(a,b)} \Bigl( \a^{-2/3}\b^{-1/3}\lt|\nabla'\rho^{1/2}\rt|^2+ 
 \a^{-4/3}\b^{-2/3}\lt|\partial_3\rho^{1/2}\rt|^2\\
+  \a^{2/3}\b^{-2/3}\left({B_3} -(1-\rho)\right)^2
+  \b^{-1}|{B'}|^2 \Bigr)dx\,,
\end{multline}
and
\begin{equation}\label{eqdeftildaFext} 
  \tilde F_{\alpha,\beta}^{\Ext}(B):=
  \alpha^{1/3}\beta^{7/6}
  \| \beta^{-1}B_3-1\|^2_{H^{-1/2}(Q_1\times \{\pm1\})}\,.
\end{equation}
Their sum  corresponds to the energy $\wE$, up to a reconstruction of $u$ and $A$
that will be discussed in Section \ref{recoveryGL}. A pair $(\rho,B)$ is admissible for $\tilde F_{\a,\b}^{(a,b)}$ if $\rho$ and $B$ are $Q_1$-periodic, $\div B=0$ and $\int_{Q_1} B_3(x',x_3)dx'=\beta$ for all $x_3$.

We say that a pair $(\rho,B)$ is $k$-quantized if there is a closed set
$\omega\subset Q_{1,1}$ 
such that $B=0$ outside $\omega$, $\rho=0$ in $\omega$, 
and the flux of $ \b^{-1}B_3$ over every connected component
of $\omega\cap\{x_3=z\}$ is an integer multiple of $2\pi /k$, for all $z\in(-1,1)$. 
\subsection{Construction in the bulk}
\label{secconstrbulk}
This section is concerned with the local construction of flux tubes. 
For notational simplicity we present this construction in $\R^2\times(a,b)$, without the periodicity assumption; since $\rho=1$ and $B=0$  outside
a small region, its periodic extension is immediate (see proof of Proposition 
\ref{propubregular} below).
We start from the optimal profile at the boundary of the individual tubes, with the lengths measured in units of the coherence length (recall \eqref{scalesepar2}). For the purpose of the upcoming constructions, we cut the profile at a lengthscale $R\gg1$ towards the normal region. 
\begin{lemma}\label{lemmaVR}
 Consider the functional $\displaystyle G(v):=2\sqrt\pi\int_0^{+\infty} |\dot{v}|^2+(1-v^2)^2 dt$. Then,
 \begin{equation*}
  \inf\left\{ G(v): v(0)=0, \lim_{t\to+\infty} v(t)=1\right\}=K_*=\frac83\sqrt\pi\,.
 \end{equation*}
Furthermore, for  all $R\ge3$ there is  $v_R\in C^\infty(\R;[0,1])$ with $v_R(t)=0$ for $t\le 0$, 
$v_R(t)=1$ for $t\ge R$, $|\dot{v}_R|\le 2$, and, setting
$K_R:=G(v_R)$, one has
\begin{equation*}
 \lim_{R\to\infty} K_R=\frac83\sqrt\pi,
\end{equation*}
and
$\int_0^{+\infty} t \left(|\dot{v}_R|^2+(1-v_R^2)^2\right) dt\lsim 1$.
\end{lemma}
\begin{proof}
The lower bound follows from the usual Modica-Mortola type computation
\begin{equation*}
\int_0^{+\infty} |\dot{v}|^2+(1-v^2)^2dt\ge 2\int_0^{+\infty} (1-v^2)\dot{v}dt=2\int_0^1 (1-s^2) ds = \frac43\,.
\end{equation*}
To prove the upper bound, we recall that $v(t):= \tanh t$ is the minimizer of $G$ under the constraint $v(0)=0$. A direct computation shows that $G(v)=K_*$. We then define for $R>0$,
\begin{equation*}
  \hat v_R(t):=
 \begin{cases}
 0 & \text{ if } t<1/R,\\
  \displaystyle\frac{\tanh (t-1/R)}{\tanh (R-2/R)}&\text{ if } t\in[1/R,R-1/R],\\
  1 & \text{ if } t>R-1/R,
 \end{cases}
\end{equation*}
and $v_R:=\psi_{1/R}\ast \hat v_R$, with $\psi_{1/R}\in C^\infty_c(-1/R,1/R)$ a mollifier.
By construction, this has the desired properties and verifies  $G(v_R)\to G(v)$ as $R\to\infty$.
\end{proof}

We start with the simple case in which the limiting measure comes from a Lipschitz curve; in practice this will be used only for affine or piecewise affine curves.
\begin{lemma}\label{lemmacurve3}
  Let $X: (a,b)\to \R^2$ be a Lipschitz curve, $\phi>0$, 
  $R>0$. We define
\begin{equation*}
 \rho(x):=  v_R^2\left(\frac{|x'-X(x_3)|-\sqrt{\beta\phi/\pi}}{\eta}\right),
\end{equation*}
where  $\eta:=\alpha^{-2/3}\beta^{1/6}\ll 1$ is the coherence length (see \eqref{scalesepar2}) and $v_R$ was introduced in Lemma \ref{lemmaVR}, and define $B$ through
\begin{equation*}
B_3(x):=\chi_{\B'(X(x_3),\sqrt{\beta \phi/\pi} )}(x')\,, \hskip1cm
B'(x):=B_3(x) \dot{X}(x_3) \,.
\end{equation*}
Then, $\rho B=0$ almost everywhere, $\div B=0$, and 
\begin{equation*}
\tilde F^{(a,b)}_{\alpha,\beta}(\rho,B)\le  \left(1+ \frac{C}{(\phi \alpha^{4/3}\beta^{2/3})^{1/2}}
+ \frac{C}{\phi \alpha^{4/3}\beta^{2/3}}\right)
\int_a^b \bigl( K_R\sqrt{\phi}+\phi |\dot{X}|^2\bigr) dx_3.
\end{equation*}
 The constant $C$  is universal. Moreover, if $\mu:=\phi \delta_{X(x_3)}\otimes dx_3$,
\begin{equation}\label{distestim}
 W_2^2(\b^{-1} B_3, \mu)\le |b-a|\frac{\beta \phi^2}{2\pi}.
\end{equation}

\end{lemma}
\begin{proof} The condition  $\rho B=0$ follows from $v_R(t)=0$ for $t\le 0$.
To check the divergence condition, pick $\psi\in C^1_c(\R^2\times(a,b))$ and compute
\begin{align*}
 \int_{\R^2\times(a,b)} (B_3\partial_3 \psi  + B'\cdot \nabla'\psi) dx
 &=\int_a^b  \int_{\B'(X(x_3), \sqrt{\beta\phi/\pi})} 
( \partial_3 \psi + \nabla'\psi \cdot \dot{X}(x_3)) dx'dx_3\\
& =\int_a^b  \int_{\B'(0, \sqrt{\beta\phi/\pi})} 
 \frac{d}{dx_3} \psi(X(x_3)+y', x_3) dy' dx_3=0\,.
\end{align*}
We now estimate the energy.  We start from the interfacial energy  at fixed $x_3$ (see \eqref{eqdeftildaF}),
\begin{align*}
E_1(x_3):=\int_{\R^2} \alpha^{-2/3}\beta^{-1/3} |\nabla'\rho^{1/2}|^2 + \alpha^{2/3}\beta^{-2/3}
 (B_3-(1-\rho))^2  dx'\,.
\end{align*}
If $|x'-X(x_3)|<  \sqrt{\beta\phi/\pi}$ then $\rho=0$  and $B_3=1$, whereas
if $|x'-X(x_3)|>\eta R +\sqrt{\beta\phi/\pi}$ then  $\rho=1$, $B_3=0$. 
In the intermediate region, we use $|\nabla'\rho^{1/2}|=|\dot{v}_R|/\eta$.
Passing to polar coordinates
and using $r=|x'-X(x_3)|$ as an integration variable,
\begin{align*}
E_1(x_3)=
 \int_{\sqrt{\beta\phi/\pi}}^{\sqrt{\beta\phi/\pi}+\eta R} \Bigl[
 \frac{\alpha^{-2/3}\beta^{-1/3}}{\eta^2} |\dot{v}_R|^2
 + \alpha^{2/3}\beta^{-2/3} (1-v_R^2)^2\Bigr]
 \lt(\frac{r-\sqrt{\beta\phi/\pi}}\eta\rt)2\pi r dr \,.
\end{align*}
We  change variables according to  $r=\sqrt{\beta\phi/\pi}+s\eta$, insert the definition of $\eta$, and by Lemma \ref{lemmaVR} obtain
\begin{align*}
E_1(x_3)&=
 \int_0^R \beta^{-1/2}
 \left[|\dot{v}_R|^2+ (1-v_R^2)^2\right]2\pi (\sqrt{\beta\phi/\pi}+s\eta) ds \\
& \le  K_R \sqrt\phi + C  \alpha^{-2/3}\beta^{-1/3}
 \le  \sqrt\phi \left(K_R+\frac{C}{\phi^{1/2}\alpha^{2/3}\beta^{1/3}}\right)\,.
\end{align*}

The other contributions to the energy are the cost of transport and the vertical part of the gradient,
\begin{align*}
E_2(x_3):=\int_{\R^2} \alpha^{-4/3}\beta^{-2/3} |\partial_3\rho^{1/2}|^2 + \beta^{-1}|B'|^2 dx'\,.
\end{align*}
By definition of $B'$, we have for the second term, 
\begin{align*}
\int_{\R^2}  \beta^{-1}|B'|^2 dx'= |\dot{X}(x_3)|^2 \phi\,.
\end{align*}
For the first one we use $|\partial_3\rho^{1/2}|\le |\dot{v}_R||\dot{X}|/\eta$ and change variables as above to obtain
\begin{align*}
\int_{\R^2} \alpha^{-4/3}\beta^{-2/3} |\partial_3\rho^{1/2}|^2  dx'
&\lsim \frac{ \alpha^{-4/3}\beta^{-2/3}}{\eta} |\dot{X}(x_3)|^2
\int_0^R (  \sqrt{\beta\phi/\pi}+s\eta)|v_R'|^2 ds\\
&\lsim  \alpha^{-4/3}\beta^{-2/3} |\dot{X}(x_3)|^2 (1 + \sqrt{\beta\phi}/\eta)\\
&=  |\dot{X}(x_3)|^2 \phi \left(
\frac{1}{\phi\alpha^{4/3}\beta^{2/3}}+
\frac{1}{\phi^{1/2}\alpha^{2/3}\beta^{1/3}}
\right)
\,.
\end{align*}
To prove \eqref{distestim} we consider the transport map $T(x',x_3)=(X(x_3),x_3)$ which gives
\[W_2^2(\b^{-1}B_3,\mu)\le \int_a^b \b^{-1} \int_{\B'(X(x_3),\sqrt{\b \phi/\pi})} |x'-X(x_3)|^2 dx'dx_3=|b-a| \frac{\b \phi^2}{2\pi}.\]
\end{proof}

We now turn to the construction around branching points. Since the total length around branching points is small, the construction here does not need to achieve the optimal constant but only the optimal scaling. The idea of the construction is the following.
We first transform disks into squares, then split the square into two rectangles and then retransform each rectangle into a disk. The construction is sketched in Figure \ref{fig1}.
We start with the transformation from a rectangle to a disk.
\begin{figure}
\centerline{\includegraphics[width=6cm]{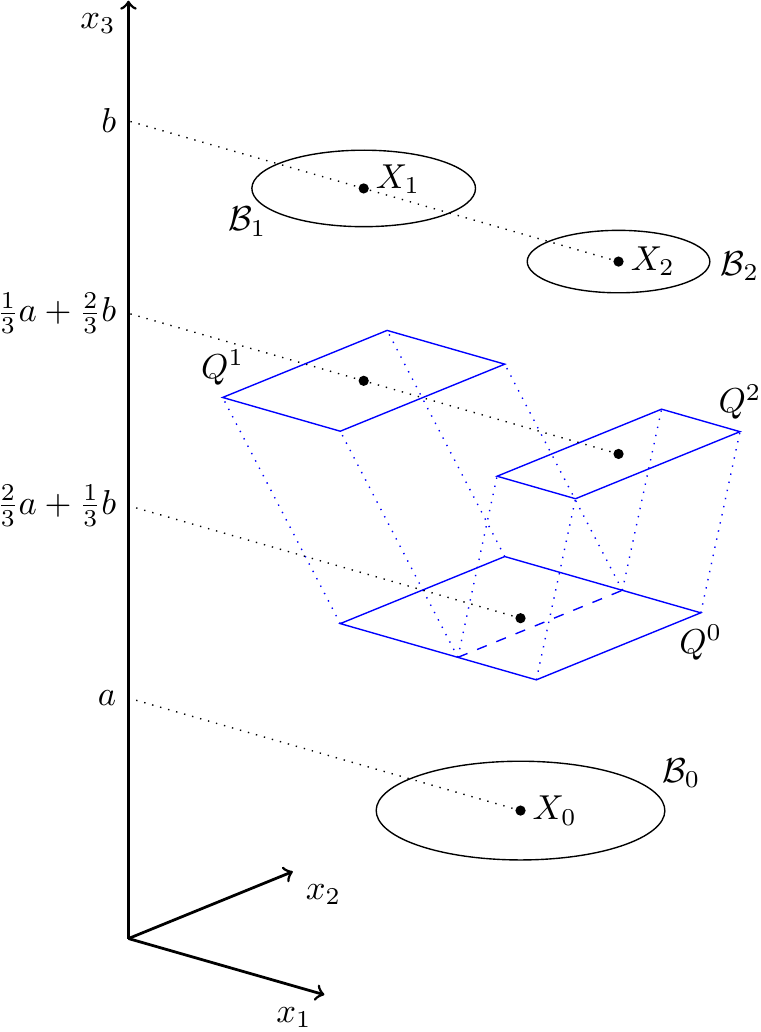}}
\caption{Construction around a branching point}
\label{fig1}
\end{figure}

\begin{lemma}\label{lemmacurveendpoints}
  Let $a\le b\in \R$, $\gamma>0$, $X\in \R^2$, $\phi>0$, 
  and $R\ge 3$. Let then $Q \subset \R^2$ be a rectangle with side lengths $w$ and $h$ centered in $X$, such that $wh=\beta\phi$,
 $w/h+h/w\le \gamma$  and $\a^{-4/3}\b^{-2/3}\le \phi$.  Let as before $\eta:=\alpha^{-2/3}\beta^{1/6}$ be the coherence length.
 
  Then there are $\rho\in L^\infty(\R^2\times [a,b];[0,1])$ and $B\in L^2(\R^2\times [a,b];\R^3)$ such that $\div B=0$, $\rho B=0$,
  with $B=0$ and $\rho=1$ on $(\R^2\setminus \B'(X,r))\times[a,b]$ for some $r\sim \eta R + \sqrt{\beta\phi}$, and
\begin{equation*}
\tilde F^{(a,b)}_{\alpha,\beta}(\rho, B) \lsim 
\sqrt{\varphi} |a-b| R^2+ \frac{r^2}{|a-b|}\varphi R^2,
\end{equation*}
where  the implicit  constants  only depend  on $\gamma$. 
Further, $\rho$ and $B$ satisfy the boundary conditions
\begin{equation*}
 \rho(x',a)=  v_R^2\left(\frac{|x'-X|-\sqrt{\beta\phi/\pi}}{\eta}\right),
\hskip1cm
 \rho(x',b)=  \min\left\{1,\frac{\dist^2(x', Q)}{\eta^2}\right\}\,,
\end{equation*}
and
\begin{equation*}
B_3(x',a)= \chi_{\B'(X, \sqrt{\beta\phi/\pi})}(x')\,,
\hskip3mm\text{ }\hskip3mm
B_3(x',b)= \chi_{Q}(x')\,.
 \end{equation*}
\end{lemma}
We note that our assumptions on the parameters $w, h, \phi,$ just mean $w\sim h\sim \sqrt{\beta \phi}\ge \eta$. That is, the thread diameter is large compared to the coherence length. 
Note that $r\sim \eta R+\sqrt{\beta \phi}$ behaves like the maximum  between the thread diameter $w\sim h$ and the cut-off scale $\eta R$. 
The proof of Lemma \ref{lemmacurveendpoints} is based on an explicit construction for a bilipschitz bijection with unit determinant  
that  transforms  a rectangle into a circle, which we first present. 
\begin{figure}
\centerline{\includegraphics[width=5cm]{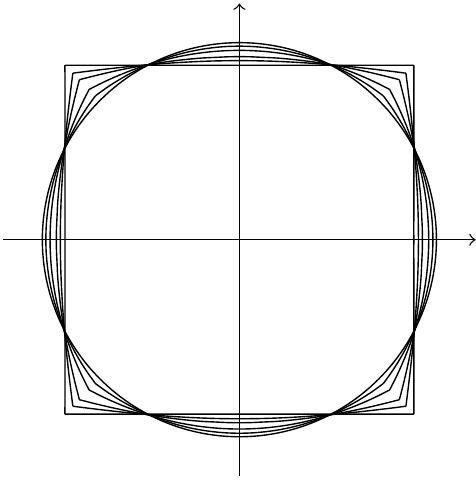}}
\caption{The construction in Lemma \ref{lemmasquarecircle}
transforms a rectangle into a circle, keeping $\det \nabla'u=1$.
Each curve corresponds to a different value of $x_3$,
and plots the solution of $\hat r(r,\theta)=t$, which
corresponds to $r(\theta)=t/[\lambda(x_3)\cos(\theta x_3)]$ for $\theta\in (-\pi/4,\pi/4)$.}
\label{fig-square}
\end{figure}
\begin{lemma}\label{lemmasquarecircle}
Assume that  $z_-,z_+,h,w>0$, $X\in\R^2$ are given, $z_-< z_+$. Then there is $u:\R^2\times[z_-,z_+]\to\R^2$ such that
 \begin{equation*}
  u(x',z_-)=x'\,,\hskip5mm
  u(X+(-\frac12w,\frac12w)\times(-\frac12h,\frac12h),z_+)=\B'(X,\sqrt{hw/\pi})\,,\hskip5mm
  \det \nabla' u =1 \text{ a.e.}\,.
 \end{equation*}
 The function $x\mapsto (u(x),x_3)$ is bilipschitz, its inverse is of 
 the form $y\mapsto (U(y),y_3)$, and $u(X,x_3)=X$ for all $x_3$. 
 If additionally $h/w+w/h\le \gamma$, then the bounds
 $|\nabla'u|+|\nabla'U|\lsim 1$, $|\partial_3 u|+|\partial_3 U|\lsim h/|z_+-z_-|$
 hold, with constants which only depend on $\gamma$. 
\end{lemma}
\begin{proof}
By scaling and translation we may assume that $hw=\pi$, $z_+=1$, $z_-=0$, $X=0$.
We can further assume $h=w=\sqrt\pi$, as the general case is obtained by taking the composition at each $x_3$ 
with the linear map $\diag(g(x_3), 1/g(x_3))$, where $g(x_3):=\sqrt{h/w}(1-x_3)+x_3$.
 We work in polar coordinates, and construct functions $\hat r$, $\hat\theta$ of $r$, $\theta$ and $x_3$ such that 
 \begin{equation*}
  u(r\cos\theta,r\sin\theta,x_3)=(\hat r\cos\hat\theta, \hat r\sin\hat\theta),
 \end{equation*}
 with $r\ge 0$ and $0\le\theta\le \pi/4$, and then extend by symmetry. 
 We set
 \begin{equation*}
  \hat \theta = f(\theta, x_3) \,,\hskip5mm \hat r = r \lambda(x_3) \cos(x_3\theta),
 \end{equation*}
 where $\lambda$ and $f$ are two functions still to be determined (see Figure \ref{fig-square}).
 The extension of the $\frac{1}{8}-$sectors by reflection is feasible provided that
 $f(0,x_3)=0$ and $f(\pi/4,x_3)=\pi/4$ for all $x_3$, the boundary data 
 are attained provided that 
 $f(\theta,0)=\theta$, $\lambda(0)=1$, $\lambda(1)=2/\sqrt\pi$. The latter ensures that indeed the straight segment $r \cos \theta=\frac{w}{2}=\frac{\sqrt{\pi}}{2}$ is mapped into the unit circle $\hat{r}=1$.
The determinant condition is equivalent to
\begin{equation*}
 1=\frac{\hat r}{r}\partial_r \hat{r}  \partial_\theta f= \lambda^2(x_3)\cos^2(x_3\theta) \partial_\theta f(\theta, x_3),
\end{equation*}
which can be solved (using $f(0,x_3)=0$) to give
\begin{equation*}
 f(\theta,x_3):=\frac{1}{x_3\lambda^2(x_3)}\tan (\theta x_3)\,,
\end{equation*}
smoothly extended  to $f(\theta,0)=\theta/\lambda^2(0)$.
The condition $f(\pi/4,x_3)=\pi/4$ determines $\lambda$,
\begin{equation*}
 \lambda(x_3):=\left(\frac{\tan (\pi x_3/4)}{\pi x_3/4}\right)^{1/2},
\end{equation*}
which obeys $\lambda(1)=2/\sqrt\pi$ and smoothly extends to $\lambda(0)=1$. Clearly $\lambda\sim 1$ so that $\partial_\theta f\sim 1$. This implies that the change of variables
defines a smooth deformation of the $\frac{1}{8}-$sector which smoothly depends on $x_3\in[0,1]$.
\end{proof}
\begin{proof}[Proof of Lemma \ref{lemmacurveendpoints}]
After a rotation and a translation, we may assume that $X=0$, $a=0$ and $b>0$; so that $Q=(-w/2,w/2)\times (-h/2,h/2)$.
 We treat two regions separately.

In the lower region $\R^2\times [0,b/2]$, we interpolate between 
$\B'_0:=\B'(0, \sqrt{\beta\phi/\pi})$ and $Q$.
To do this, let  $u$ and $U$ be the functions from Lemma \ref{lemmasquarecircle}, using it for the rectangle $Q$ and $z_-=0$, $z_+=b/2$.
The magnetic field is defined by
 \begin{equation*}
B_3(x):=\chi_{\B'_0}(u(x))\,,\hskip1cm
  B'(x):=\chi_{\B'_0}(u(x))\partial_3 U(u(x),x_3)\hskip3mm \text{ for } x_3\in[0,b/2]\,.
 \end{equation*}
The density is defined by
 \begin{equation*}
  \rho(x):=v_R^2\left(\frac{|u(x)|-\sqrt{\beta\phi/\pi}}{\eta}\right) \hskip3mm \text{ for } x_3\in[0,b/2]\,.
 \end{equation*}
  The condition $\rho B=0$ follows immediately, as well as the boundary data at $x_3=0$. 
  Since $B=0$ and $\rho=1$ whenever $|u(x)|\ge \eta R + \sqrt{\beta\phi/\pi}$
and since $|u(x)|\ge |x'|/\|\nabla'U\|_\infty$, 
we have $B=0$ and $\rho=1$ whenever $|x'| \ge \|\nabla'U\|_\infty  (\eta R + \sqrt{\beta\phi})$.

  In order to check $\div B=0$, we fix $\psi\in C^1_c(\B'_r\times(0,b/2))$. Performing then
 a change of variables at each $x_3$ gives, since $\det \nabla'U=1$,
\begin{alignat*}1
 \int_{\R^2\times(0,b/2)} \left[\partial_3 \psi B_3 + \nabla'\psi\cdot B' \right]dx &=
 \int_{\R^2\times(0,b/2)} (\partial_3\psi B_3 + \nabla' \psi \cdot  B')(U(y),y_3) dy
 \\
 &= \int_{\R^2\times(0,b/2)} \chi_{\B'_0} \left[\partial_3 \psi( U(y),y_3)  + \partial_3 U(y)\cdot \nabla'\psi( U(y),y_3))   \right] dy \\
 &= \int_{\B'_0} \int_{(0,b/2)} \frac{d}{dx_3} (\psi(U(y),y_3)) dy_3 dy'=0\,.
\end{alignat*}
By the properties of $u$ we obtain $B_3(x',b/2)=\chi_Q(x')$.

In the upper region $\R^2\times[b/2,b]$, we keep $B(x',x_3):=e_3 \chi_Q(x')$ concentrated on $Q$ and  linearly interpolate the profile between $\rho^{1/2}(x',b/2)$ and the profile $\rho^{1/2}(x',b)$
given in the statement:
\begin{equation*}
\rho^{1/2}(x):=\frac{2}{b}\lt[
\rho^{1/2}(x',b/2) (b-x_3)+ \rho^{1/2}(x',b)(x_3-b/2)\rt]
 \hskip3mm \text{ for } x_3\in(b/2,b)\,.
\end{equation*}
It is immediate to check that  $B_3$ and $\rho$ match continuously at all interfaces, $\rho B=0$ and $\div B=0$. This concludes the construction.

We estimate the energy similarly to the  proof of Lemma \ref{lemmacurve3}. Lemma \ref{lemmasquarecircle} gives $|\nabla'u|\lsim 1$, 
$|\partial_3 u|\lsim h/b$, $|\partial_3 U|\lsim h/b$,
with constants depending only on $\gamma$. This yields $|B'|\les h/b$. Furthermore,
 by Lemma \ref{lemmaVR} 
$|v_R|\le1$ and $|\dot{v}_R|\le 2$, so that
$\rho\le 1$, $|\nabla'\rho^{1/2}|\lsim 1/\eta $ and $|\partial_3\rho^{1/2}|\lsim h/(\eta b)$.

We start with the region $x_3\in [0,b/2]$. 
All integrals in $x'$ can be restricted to the set 
\begin{equation*}
 \Omega_{x_3}:=\{x': \sqrt{
 \beta\phi/\pi}\le |u(x',x_3)|\le \eta R+\sqrt{\beta\phi/\pi}\}.
 \end{equation*}
Since $\det \nabla u'=1$, 
we have 
\begin{equation*}|\Omega_{x_3}|=|\B'(0,\eta R+\sqrt{\beta\phi/\pi})\setminus \B'_0|
= \pi \eta^2R^2+2\sqrt{\beta\phi\pi}\eta R\lsim \eta \sqrt{\beta\phi} R^2,
\end{equation*} (in the last step we used the assumption on $\phi$ and $R\ge 3$). 
Therefore
\begin{align*}
   \tilde F_{\alpha,\beta}^{(0,b/2)}(\rho,B)
   &\lsim \int_{0}^{b/2}\int_{\Omega_{x_3}} \left( \frac{\alpha^{-2/3}\beta^{-1/3}}{\eta^2}
   + \frac{h^2\alpha^{-4/3}\beta^{-2/3}}{\eta^2 b^2} + \alpha^{2/3}\beta^{-2/3} + \frac{h^2}{\beta b^2}\right) dx' dx_3\\
   &\lsim \int_{0}^{b/2}|\Omega_{x_3}| dx_3 \left( \alpha^{2/3}\beta^{-2/3} +  \frac{h^2}{\beta b^2}\right) \\
   &\lsim \left( b + \frac{h^2}b \alpha^{-2/3}\beta^{-1/3}\right) R^2 \sqrt{\phi}\,\\
   &\lsim \sqrt{\phi} b R^2 +\frac{r^2}{b} \phi R^2,
   \end{align*}
where in the last line we have used that $h^2\les r^2$ and $\a^{-2/3}\b^{-1/3}\les \sqrt{\phi}$. The region $(b/2,b)$ is simpler, as the $|B'|^2$ term does not appear, the others
are the same with the exception of $|\partial_3 \rho^{1/2}|\les 1/b$, which is smaller by the factor $\eta/h\les 1$. 
\end{proof}

Using this building block we can finally produce the construction that will be used at branching points.
\begin{lemma}\label{Lembranch3}
  Let  $a<b\in \R$, $\gamma>0$, $\ell>0$, $R\ge 3$, $X_i\in\R^2$ and $\varphi_i>0$ for $i=0,1,2$ with $\phi_0=\phi_1+\phi_2$,
 $\phi_0/\phi_1+\phi_0/\phi_2\le \gamma$ and $\sqrt{\beta \phi_0}\ge \eta$,
where as above $\eta:=\alpha^{-2/3}\beta^{1/6}$.
 
 Then,
 if  $|X_0-X_1|,|X_0-X_2|\le \ell/4$,
 $\sqrt{\beta \varphi_0}\ll |X_1-X_2|$ and $ \eta R\ll \ell$,
  then there are $\rho\in L^\infty(\R^2\times [a,b];[0,1])$ and $B\in L^2(\R^2\times [a,b];\R^3)$ such that $\div B=0$, $\rho B=0$,
 $B=0$ and $\rho=1$ on $(\R^2\setminus \B'(X_0,3\ell/4))\times[a,b]$, 
\begin{equation*}
\tilde F^{(a,b)}_{\alpha,\beta}(\rho, B) \lsim 
\sqrt{\varphi_0} |a-b| R^2+ \frac{\ell^2}{|a-b|}\varphi_0R^2,
\end{equation*}
where the implicit constants only depend on $\gamma$, and which satisfy the boundary conditions
\begin{equation*}
\rho(x',a)=v_R^2\left(\frac{|x'-X_0|-\sqrt{\beta\phi_0/\pi}}{\eta}\right)
\hskip1mm\text{  }\hskip1mm
\rho(x',b)=\min_{i=1,2}v_R^2\left(\frac{|x'-X_i|-\sqrt{\beta\phi_i/\pi}}{\eta}\right)\,,
 \end{equation*}
and
\begin{equation*}
B_3(x',a)= \chi_{\B'(X_0, \sqrt{\beta\phi_0/\pi})}(x')\,,
\hskip3mm\text{ }\hskip3mm
B_3(x',b)=\sum_{i=1,2} \chi_{\B'(X_i, \sqrt{\beta\phi_i/\pi})}(x')\,.
 \end{equation*}
 \end{lemma}
Note that our assumptions on the parameters $\phi_0, \phi_1, \phi_2$ and $h$ just mean that $\eta\les \sqrt{\beta \phi_i}\ll |X_1-X_2|\les \ell$ and $\eta R\ll \ell$. That is, the thread diameter is at least as large as
 the coherence
length $\eta$ but small compared to the distance between the threads. Likewise, $\ell$ is large compared to the cut-off scale $\eta R$. 
\begin{proof}
After a translation and a rotation we may assume that $a=0$, $b>0$, $X_0=0$, $X_2-X_1=\zeta e_1$, with $\zeta>0$.
Let 
\[w_0:=h:=(\b \phi_0)^{1/2}, \qquad w_1:= \frac{\phi_1}{\phi_0} w_0 \qquad \textrm{ and } \qquad w_2:= \frac{\phi_2}{\phi_0}w_0.\]
Let then $Q^i:=X_i+[-w_i/2,w_i/2]\times[-h/2,h/2]$, for $i=0,1,2$. Notice that since $|X_i|\le \ell/4$ and $w_i\le w_0=h\ll \ell$, $Q^i\subset \B'(3\ell/8)$.\\ 
We  divide the interval $(0,b)$ in three parts (see again Figure \ref{fig1}). In $(0,b/3)$, 
we apply Lemma \ref{lemmacurveendpoints} to transform $\B'(\sqrt{\beta\phi_0/\pi})$ into $Q^0$ (in particular we have $\rho=1$ on $\lt(\R^{2}\backslash \B'(\frac{3\ell}{4})\rt)\times (0,b/3)$). In $(b/3,2b/3)$, we connect $Q^0$ to $Q^1$ and $Q^2$ by an explicit construction (see below). 
Finally, in $(2b/3,b)$ we apply again 
Lemma \ref{lemmacurveendpoints} to transform $Q^1$ and $Q^2$ back into  $\B'(X_1, \sqrt{\beta \phi_1/\pi})$ and $\B'(X_2, \sqrt{\beta \phi_2/\pi})$. Notice that in $(2b/3,b)$,
 if $\rho(x',x_3)\ne 1$ then by Lemma \ref{lemmacurveendpoints} and our hypothesis on the parameters, necessarily
\[
|x'|\le \max_i |X_i|+O(\eta R+\sqrt{\beta\phi_0})\le \frac{\ell}{4} +o(\ell).
\]
Thus,  $\rho=1$ on $\lt(\R^{2}\backslash \B'(\frac{3\ell}{4})\rt)\times (2b/3,b)$.\\
\begin{figure}\centering\resizebox{7cm}{!}{
 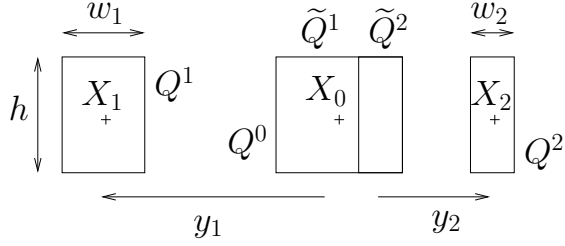
}
\caption{The construction in $(b/3,2b/3)$}\label{figseparsquares}
\end{figure}
It only remains to discuss the construction in the central region.
 Let $y_1:=X_1+w_2e_1/2$, $y_2:=X_2-w_1e_1/2$  and, for $i=1,2$,  $\tilde Q^i:=Q^i-y_i$, 
 so that up to a null set, $Q^0$ is the disjoint union of $\tilde Q^1$ and $\tilde Q^2$ (see Figure \ref{figseparsquares}). 
 Since $y_2-y_1=X_2-X_1-w_0e_1/2$,
 and by assumption $w_0=(\beta\phi_0)^{1/2} \ll|X_1-X_2|$, 
 we have  $(y_2-y_1)\cdot e_1>0$.\\
For $i=1,2$ and $x_3\in(b/3,2b/3)$ we set $\tilde{Q}^i(x_3):= \tilde{Q}_i+\frac{x_3-b/3}{b/3} y_i$. Since $(y_2-y_1)\cdot e_1>0$,  we have  $\tilde{Q}^1(x_3)\cap\tilde{Q}^2(x_3)=\emptyset$ for all $x_3$.  Furthermore, since $Q^i\subset \B'(3\ell/8)$, 
also $\tilde{Q}^i(x_3)\subset \B'(3\ell/8)$ for $x_3\in(b/3,2b/3)$. We finally let 
\begin{equation*}
\rho^{1/2}(x):=\min \lt\{ 1,\eta^{-1}\dist(x',\tilde{Q}^1(x_3)\cup \tilde{Q}^2(x_3))\rt\}\hskip3mm \text{ for } x_3\in[b/3,2b/3],
\end{equation*}
and correspondingly
 \begin{equation*}
(B',B_3)(x):=
  \sum_{i=1,2}
   \lt( \frac{y_i}{b/3},1\rt)\chi_{\tilde{Q}^i(x_3)}(x')\,.
    \end{equation*}
 All admissibility conditions are easily checked. In particular, $\rho=1$ if $\dist(x',\tilde{Q}^1\cup\tilde{Q}^2(x_3))\ge  \eta $ which holds if  $|x'|\ge 3\ell/4$. The energy estimate is immediate.
\end{proof}

\subsection{Boundary layer}
\label{secconstrboundary}

\begin{proposition}\label{propboundraylayer}
 Let $N\in\N$, $\alpha,\beta,T,t>0$ be given. Let $k:=\a^{1/3}\b^{2/3}T$
 and let $\phi_1, \dots, \phi_{N^2}$ be positive numbers such that
 \begin{equation*}
  k \phi_i \in 2\pi \N \, \qquad \textrm{and } \qquad\sum_{i=1}^{N^2} \phi_i=1.
 \end{equation*}
 Assume that in the regime $\beta\le \alpha$, $\alpha \sqrt{2}/T\le 1$, we have $t\gg \alpha^{-1}$, $1/N\ges \alpha^{-2/3}\beta^{-1/3}$ and $\phi_i\sim 1/N^2$ for every $i$.
  Then there are $\rho$ and $B$, admissible for $\tilde F_{\alpha,\beta}^{(0,t)}$ and $k$-quantized, such that
  \begin{equation}\label{eqboundarylint}
\tilde   F_{\alpha,\beta}^{(0,t)}(\rho,B)
\les tN + \frac{1}{N^2t},
\end{equation}
  \begin{equation}\label{eqboundarylext}
\alpha^{1/3}\beta^{7/6}
  \|\beta^{-1}{B_3}- 1\|_{H^{-1/2}(x_3=0)}^2\les \frac{\beta^{1/3}}{\alpha^{1/3}} + \frac{\beta^{1/2}}{N^2t}
  +\alpha^{1/3}\beta^{7/6},
  \end{equation}
and, denoting by $Q^i$ the squares centered in the $N^2$ points of the square grid of spacing $N^{-1}$ with $|Q^i|=\beta \phi_i$,
\begin{equation}\label{eqrandwb3bdly}
 B_3(x',t)=\sum_i  \chi_{Q^i}(x') \text{ and }
 \rho^{1/2}(x',t)=\min\{ 1,\eta^{-1} \dist(x',\cup_i Q^i) \}\,.
\end{equation}
\end{proposition}
Note that the assumptions on the parameters $N$ and $\phi_i$ mean that the sidelength $\sqrt{\beta \phi_i}$ of the squares is larger than the coherence length $\a^{-2/3}\b^{1/6}$ (see \eqref{scalesepar2}) and that both of them are small
compared to the distance (equal to $1/N$) between the squares. The assumption on the parameter $t$ means that the thickness $t$ of the boundary layer is large with respect to the coherence length, where we recall that vertical and horizontal lengths have different units.
\begin{proof}
The key construction is described in \cite[Lem. 4.7]{CoOtSer}. However, the notation and the scalings are different in that paper. Indeed, since  
there was no need to rescale $x_3$ by the thickness and $x'$ by the distance between the threads $T \alpha^{-1/3} \beta^{-1/6}$, in that paper length was measured in terms of the penetration length. 
We first let $r_i$ be a square of side $1/N$ centered on the square grid of spacing $N^{-1}$, 
 $Q^i\subset r_i$ be  a square with the same center and area given by
$ |Q^i|=\beta\phi_i$, and  $b_i:= \beta\phi_iN^2$.
Denoting with a star the quantities
 from \cite{CoOtSer}, we set $T_*:=Tt/2$, $\kappa_*:=\alpha \sqrt2/T$, $L_*:=T\alpha^{-1/3}\beta^{-1/6}$, 
 $r_i^*:=L_*r_i$, $\hat r_i^*:=L_*Q^i$, $d_0^*:=L_*/N$, $\rho_0^*:=\beta^{1/2}L_*/N$, $N_*=N$, and  $ b_i^*:=(|\hat r_i^*|/|r_i^*|)\kappa_*/\sqrt2  =b_i \kappa_*/\sqrt2$.

 \cite[Lem. 4.7]{CoOtSer} gives $Q_{L_*}$-periodic fields $\chi_*\in BV_\loc(\R^2\times (0,T_*);\{0,1\})$ and $B_*=(B'_*,B_3^*)\in L^2_\loc(\R^3;\R^3)$ 
 such that $\div B_*=0$, $B_*(1-\chi_*)=0$, 
and, 
 \begin{equation*}
 \chi_*(x',T_*)=\sum_i \chi_{\hat r_i^*}(x')  \qquad \textrm{and} \qquad B^*_3(x',T_*)=\frac{\kappa_*}{\sqrt2} \sum_i \chi_{\hat r_i^*}(x') \hskip5mm\text{ for } x'\in Q_{L_*}\,,
 \end{equation*}
with energy 
 \begin{equation}\label{estimstar}
 \frac{1}{L_*^2} \int_{Q_{L_*}\times(0,T_*)} \kappa_* |D\chi_*|+|B_*'|^2+\chi_* \lt(B_3^*-\frac{\kappa_*}{\sqrt2}\rt)^2 dx \les
E_*^{int}:=  \lt( \frac{\kappa_* \rho_0^* T_*}{\sqrt2 } +  \frac{\kappa_*^2(\rho_0^*d_0^*)^2}{2T_*}\rt)\lt(\frac{N}{L_*}\rt)^2
 \end{equation}
and
\begin{equation*}
 \frac{1}{L_*^2}  \|B_3^*-\sum_i b_i^*\chi_{q_j^*}\|^2_{H^{-1/2}(Q_{L_*})}\les
 E_*^{ext}:=\lt(
 \kappa_*(\rho_0^*)^2+
      \frac{\kappa_*^2(\rho_0^*)^3 d_0}{T_*}\rt)\lt(\frac{N}{L_*}\rt)^2\,.
\end{equation*}
We extend $\chi_*$ and $B_*$ to $Q_{L_*}\times(T_*,2T_*)$ by 
\[\chi_*(x):=\chi_*(x',T_*) \qquad \textrm{and} \qquad B_*=(0,B_3^*(x',T_*)) \hskip5mm\text{ for } x'\in Q_{L_*}\times(T_*,2T_*),\]
so that \eqref{estimstar} holds in $Q_{L_*}\times(0,2T_*)$.
We set 
\begin{equation*}
 \rho_*(x):=\min\{1,\kappa_*^2d(x,\omega_*)^2\},
\end{equation*}
where $\omega_*=\{x: \chi_*(x)=1\}$ (here $d(\cdot,\omega_*)$ is the 3D distance, periodic in the tangential directions). Since $\omega_*$ is invariant in the $x_3-$direction inside $(T_*,2T_*)$
and since $t\gg \a^{-1}$, giving $T_*\gg \kappa_*^{-1}$, for $x_3=2T_*$, 
\[\min\{1,\kappa_*^2d(x,\omega_*)^2\}=\min\{1,\kappa_*^2\dist(x',\omega_*\cap\{x_3=2T_*\})^2\} \qquad \textrm{for } x=(x',2T_*).\]
Hence,
\[\rho_*(x',2T_*)=\min_{i} \lt\{\min(1,\kappa_*^2\dist(x',\hat{r}^i_*)\rt\}.\]
The same computation as in the proof of \cite[Th. 4.9]{CoOtSer} leads to
 \begin{alignat*}1
 \frac{1}{L_*^2} \int_{Q_{L_*}\times(0,2T_*)}  |\nabla\rho_*^{1/2}|^2+|B_*'|^2+ (B_3^*-\frac{\kappa_*}{\sqrt2}(1-\rho_*))^2 dx \les
E_*^{int}\,.
 \end{alignat*}
 We scale back in the tangential direction to obtain
 \begin{equation*}
\rho(x',x_3):=\rho^*(L_*x',x_3)\quad \text{ and } \quad 
 (B',B_3)(x',x_3):=\frac{\sqrt2}{\kappa_*}\lt(\frac{B'_*}{L_*},B_3^*\rt)(L_*x',x_3)\,.
 \end{equation*}
We have $\div B=0$ and (\ref{eqrandwb3bdly}) holds.
Changing variables gives $\nabla'\rho^{1/2}(x)=L_*\nabla'\rho_*^{1/2}(L_*x',x_3)$
and $\partial_3\rho^{1/2}(x)=\partial_3 \rho^{1/2}_*(L_*x',x_3)$, so that
\begin{multline*}
\tilde F^{(0,t)}_{\alpha,\beta}(\rho,B)\le 
\alpha^{-4/3}\beta^{-2/3} \frac{1}{L_*^2} \int_{Q_{L_*}\times(0,2T_*)} |\nabla \rho_*^{1/2}|^2 + |B'_*|^2 +
\lt(B_3^*-\frac{\kappa_*}{\sqrt2}(1-\rho_*)\rt)^2 dx\\
\les\alpha^{-4/3}\beta^{-2/3}  E_*^{int}.
\end{multline*}
Since $E_*^{int}=\alpha^{4/3}\beta^{2/3} (Nt/2+2/(N^2t))$, this concludes the estimate for $\tilde F^{(0,t)}_{\alpha,\beta}$.

We now estimate the boundary term. By the embedding of $L^\infty$ into $H^{-1/2}$ (recall that since $\sum_i \phi_i =1$, $\int_{Q_1} \sum_i(b_i-\beta) \chi_{r_i} dx'=0$)
we have
\begin{equation*}
 \|\sum_i (b_i-\beta)\chi_{r_i}\|^2_{H^{-1/2}(Q_{1})}\les
 \|b_i-\beta\|_\infty^2  \,.
\end{equation*}
Rescaling the boundary estimate for $B_3^*$ leads to
\begin{equation*}
  \|B_3-\sum_i b_i\chi_{r_i}\|^2_{H^{-1/2}(Q_1)}\les
 \frac{1}{L_*^3}  \frac2{\kappa_*^2}\|B_3^*-\sum_i b_i^*\chi_{r_j^*}\|^2_{H^{-1/2}(Q_{L_*})}\les
\frac{\beta^{7/6}}{\alpha^{2/3}} + \frac{\beta^{4/3}}{\alpha^{1/3} N^2t}
 \end{equation*}
Adding terms and using that $\sum_i \chi_{r_i}\equiv 1$, we conclude that 
 \begin{equation*}
\alpha^{1/3}\beta^{7/6}
  \|\beta^{-1}{B_3}- 1\|_{H^{-1/2}(Q_1)}^2\les \frac{\beta^{1/3}}{\alpha^{1/3}} + \frac{\beta^{1/2}}{N^2t}
  +\frac{\alpha^{1/3}}{\beta^{5/6}} \|b_i-\beta\|_\infty^2  \,.
  \end{equation*}

\end{proof}

\subsection{Back to the full GL functional}\label{recoveryGL}
In this section, we relate the functional 
$\tilde F_{\alpha,\beta}$, as defined in (\ref{eqdeftildaF}--\ref{eqdeftildaFext}),
with the functional $\tilde E_T$, as defined in (\ref{eqdefwe}).

\begin{proposition}\label{thirdupperbound}
Let $k=\alpha^{1/3}\beta^{2/3}T$ and let $\rho,B$ be $k$-quantized admissible functions for $\tilde F_{\alpha,\beta}$ with $\rho B=0$.
Let $\omega:=\{\rho=0\}$ and then for $z\in(-1,1)$, $\omega_{z}:=\omega\cap \{x_3=z\}$. Assume that $\omega$ is closed, that $\omega^c$ is connected  and that for every $x_3\in(-1,1)$, $\omega_{x_3}^c$ is 
 also connected. 
Then, there is a pair $(u,A)$, admissible for $\tilde E_T$, such that
$\rho= |u|^2$, $B=\nabla\times A$, and 
\begin{equation}\label{limsupT}
 \tilde E_T(u,A)=\tilde F_{\a,\b}(\rho,B).
\end{equation}
\end{proposition}
\begin{proof}
We shall construct a function $A$ on $\R^2\times (-1,1)$ so that 
 $B=\nabla\times A$ everywhere, and then 
 a multivalued function  $\theta$ on 
  the set $\Omega^c:=\Z^2\times\{0\}+\omega^c$ such that $\nabla\theta=\b^{-1}kA$. Here $B$ and $\rho$ are $Q_1$-periodic, but 
 $A$ and $\theta$ not necessarily. 
 Setting
 $u:=\rho^{1/2} e^{i\theta}$ we then have 
 $\nabla_{\b^{-1}kA}u= e^{i\theta}\nabla\rho^{1/2}$, which directly implies (\ref{limsupT}). Notice that thanks to the hypothesis on $\omega$, the set $\Omega^c$ is a connected open set.

 We start from the construction of $A$.
By  \cite[Lem. 4.8]{CoOtSer} applied to $B-\beta e_3$
there is a $Q_1$-periodic potential $A_\mathrm{per}$ such that 
$\nabla \times A_\mathrm{per}=  B- \b e_3$ and $\Div A=0$. We define $ A(x):= A_\mathrm{per}(x)+\b x_1 e_2$ so that
$\nabla \times  A=B$. We remark that on any open set where $B=0$ the vector field $A$ is curl-free and divergence-free, therefore harmonic, 
and in particular smooth. In particular, since $\Omega^c$ is open and since by the Meissner condition, $B=0$ in $\Omega^c$, $A$ is smooth in $\Omega^c$. 

We now turn to the existence of $\theta$.
For a fixed level $x_3$, let $h+\omega^i_{x_3}$, $h\in\Z^2$, $i=1,..., I(x_3)$ be the connected components of $(\R^2\times\{x_3\}) \cap (\Z^2+\omega_{x_3})$. Denote the flux going 
through $\omega^i_{x_3}$ by
\begin{equation*}
\Phi^i(x_3):=\int_{\omega^i_{x_3}} B_3\, dx'.
\end{equation*}
By assumption we have
\begin{equation}\label{condfluxi}
\b^{-1}k \Phi^i(x_3)\in 2\pi \Z \qquad \textrm{ for every $i$. }
\end{equation}

Fix a smooth curve $\Gamma_0:=\{(\gamma_0(x_3),x_3) \ : \ x_3\in(-1,1)\}\subset \Omega^c$. For $x_3,y_3\in(-1,1)$, let $\Gamma_0^{x_3,y_3}:=\{(\gamma_0(t),t) \ : \ t\in(x_3,y_3)\}\subset \Gamma_0$.
For $x=(\gamma_0(x_3),x_3)$, let $\theta(x):= \b^{-1}k \int_{\Gamma_0^{0,x_3}}  A\cdot \tau d\H^1$. 
Now, for a generic $x=(x',x_3)\in \Omega^c$, let 
\[\theta(x):= \theta(\gamma_0(x_3),x_3)+\b^{-1}k\int_{\Gamma^x}  A\cdot \tau d\H^1,\]
where $\Gamma^x$ is any (horizontal) curve in $\Omega^c\cap(\R^2\times\{x_3\})$ connecting $x$ to $(\gamma_0(x_3),x_3)$. This gives a well defined $\theta:\Omega^c\to\R/2\pi\Z$  since
for every closed (horizontal) curve $\Gamma$ in $\Omega^c\cap(\R^2\times\{x_3\})$,
\begin{equation}\label{condfluxcurvebis}
  \b^{-1}k\int_{\Gamma}  A\cdot \tau d\H^1\in 2\pi \Z.
\end{equation}
Indeed, this follows  by  Stokes' Theorem and \eqref{condfluxi}. Let us show that $\nabla \theta=\b^{-1}k A $ in $\Omega^c$. Let $x=(x',x_3)$ be fixed and let $\Gamma^x$ 
be a fixed  simple smooth curve joining $x$ to $(\gamma_0(x_3),x_3)$ inside $\Omega^c\cap(\R^2\times\{x_3\})$. Since $\Omega^c$ is open, there is a simply connected neighborhood $V$ of
$\Gamma^x$  such that $V\subset \Omega^c$. Let then $y=(y',y_3)\in V$. Upon shrinking $V$, we may assume that $(\gamma_0(y_3),y_3)\in V$. Let $\Gamma^y\subset V$ be a 
smooth curve joining $y$ to $ (\gamma_0(y_3),y_3)$ and let $\Gamma^{x,y}\subset V$ be a smooth curve joining $x$ to $y$. By definition, we have
\[\theta(x)=\theta(\gamma_0(x_3),x_3)+ \b^{-1}k \int_{\Gamma^x}A\cdot \tau d\H^1, \quad \theta(y)=\theta(\gamma_0(y_3),y_3)+ \b^{-1}k\int_{\Gamma^y}  A\cdot \tau d\H^1\]
and
\[\theta(\gamma_0(y_3),y_3)=\theta(\gamma_0(x_3),x_3)+ \b^{-1}k\int_{\Gamma_0^{x_3,y_3}}  A\cdot \tau d\H^1,\]
so that 
\[
 \theta(y)-\theta(x)= \b^{-1}k\int_{\Gamma^y} A\cdot \tau d\H^1+ \b^{-1}k\int_{\Gamma_0^{x_3,y_3}} A\cdot \tau d\H^1-\b^{-1}k\int_{\Gamma^x}  A\cdot \tau d\H^1.
\]
However, by Stokes Theorem (and $B=0$ in $V$),
\[\int_{\Gamma^y}A\cdot \tau d\H^1+\int_{\Gamma_0^{x_3,y_3}} A\cdot \tau d\H^1-\int_{\Gamma^x}  A\cdot \tau d\H^1=\int_{\Gamma^{x,y}} A\cdot \tau d\H^1,\]
so that 
\[
 \theta(y)-\theta(x)=\b^{-1}k\int_{\Gamma^{x,y}}  A\cdot \tau d\H^1,
\]
proving that indeed, $\nabla \theta=\b^{-1}k A$ in $\Omega^c$.
\end{proof}

\subsection{Proof of the upper bound}
We start from a construction for an $N$-regular measure and finite $R$.
\begin{proposition}\label{propubregular}
 Let $\mu\in \MNreg(Q_{1,1})$ for some $N\in\N$, $R\ge 3$, and assume (\ref{eqlbassumptcoeff})
  and (\ref{defKtotal}) hold. Then, 
 there exist sequences $u_n: Q_{1,1}\to \C$ and $A_n: Q_{1,1}\to \R^3$ such that  
\begin{equation}\label{estimenerregular}
\limsup_{n\to +\infty} \,  \wE(u_n,A_n)\le 
\frac{K_R}{K_*}
I(\mu)
+ \frac{C}{N^{1/2}}(1+I(\mu)),
\end{equation}
with $|u_n|$ and $\nabla\times A_n$ $Q_1-$periodic. Here $K_R$ and $K_*$ are as in Lemma \ref{lemmaVR}, $C$ is universal. Moreover,
\begin{equation}\label{estimW2regular}
 \limsup_{n\to +\infty} W_2^2(\b^{-1}_n B^n_3,\mu)\les  N^{-3/2}.
\end{equation}
 \end{proposition}
\begin{proof}
We first  modify slightly $\mu$ in order to be able to 
 use Proposition \ref{propboundraylayer} for the boundary layer.
Set $\varepsilon_N:=N^{-3/2}$.
For $x_3\in [-1+\eps_N,1-\varepsilon_N]$
let $\hat{\mu}_{x_3}:=\mu_{x_3/(1-\varepsilon_N)}$ and for $x_3\in [-1,-1+\eps_N]\cup[1-\varepsilon_N,1]$, let $\hat{\mu}_{x_3}:=\mu_1=\mu_{-1}=N^{-2}\sum_l \delta_{X_l}$, where the $\{X_l\}_{l=1}^{N^2}$ form a regular grid with spacing $1/N$ (recall  Def. \ref{defregular}). 
We then have 
\[I(\hat{\mu})\le \frac{1}{1-\varepsilon_N} I(\mu)+2K_*\varepsilon_N N.\]
Moreover, since $|\hat \mu|(Q_1\times (1-\varepsilon_N,1))= \varepsilon_N$, the same on the other side, and
\[
 W_2^2((1-\varepsilon_N)\mu,\hat{\mu}\LL(Q_1\times (-1+\varepsilon_N,1-\varepsilon_N))\les \eps_N^2,
\]
so that 
\[
 W_2^2(\mu, \hat{\mu})\les \eps_N+\eps^2_N\les N^{-3/2},
\]
it is enough to prove the estimates for $\hat{\mu}$ instead of $\mu$. 

We start by characterizing the geometry of the construction, which will not depend on $n$.
The measure $\hat{\mu}$ is supported on finitely many polygonal curves, parametrized by
$X_i:[a_i,b_i]\to Q_1$, which are disjoint up to the endpoints and carry a flux $\phi_i>0$.
The endpoints 
are either on the boundary of $Q_{1,1}$, or they are triple points.
Those on the boundary constitute a regular grid. We let  
 $\phimin:=\min_i\phi_i$, $\phimax:=\max_i \phi_i$, and $\gamma:=8\phimax/\phimin$. 
Since there are finitely many curves, these quantities are finite 
and positive. We define as before $\eta_n:=\alpha_n^{-2/3}\beta_n^{1/6}$ to be the coherence length; by (\ref{eqlbassumptcoeff}) we have $\eta_n\to0$.

Let $y_j:=(Y_j,z_j)\in Q_{1,1}$ denote the internal endpoints of the curves.
 For $\ell>0$ sufficiently small one has that, for any $j$,
 the only curves which intersect $\B'_\ell(Y_j)\times(z_j-\ell,z_j+\ell)$ are
 those with an endpoint in $y_j$. Since $y_j$ is a triple point, there are three such curves $\Gamma_0$, $\Gamma_1$ and $\Gamma_2$, intersecting only at $y_j$.
 If we let  $M$ be the maximal slope of all curves and $t_\ell:=\ell/(8M)$, then no curve intersects
 $\partial\B'_{\ell/8}(Y_j)\times (z_j-t_\ell,z_j+t_\ell)$ (this means, they all ``exit'' from the 
 top and bottom faces).  Without loss of generality, we can assume that $(X_0,z_j-t_\ell)=\Gamma_0\cap\lt( \B'(Y_j,\ell)\times\{z_j-t_\ell\}\rt)$, $(X_1,z_j+t_\ell)=\Gamma_1\cap \lt(\B'(Y_j,\ell)\times\{z_j+t_\ell\}\rt)$
 and $(X_2,z_j+t_\ell)=\Gamma_2\cap \lt(\B'(Y_j,\ell)\times\{z_j+t_\ell\}\rt)$ (see Figure \ref{figtriplepoint}).
By definition of $t_\ell$, it holds $|X_0-Y_j|\le \frac{\ell}{8}$ and $|X_i-X_0|\le \frac{\ell}{4}$ for $i=1,2$.
 We set $\omega_j:=\B'_{\ell}(Y_j)\times (z_j-t_\ell,z_j+t_\ell)$ and let $\delta_\ell$ be the minimum distance
between any two curves outside  $\cup_j\omega_j$.
 \begin{figure}\centering\resizebox{8cm}{!}{
 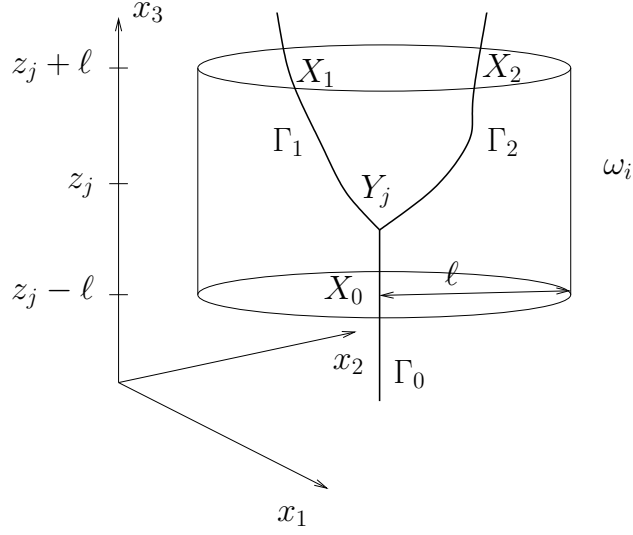
}
\caption{The construction in $(b/3,2b/3)$}\label{figtriplepoint}
\end{figure}

To explain the strategy, we first carry out a construction which ignores quantization. 
 For sufficiently large $n$ we have $ \eta_nR+\sqrt{\phimax \beta_n/\pi}<\delta_\ell/2$  and can thus use the construction
 of Lemma \ref{lemmacurve3} in a $\delta_\ell/2$-neighborhood of each curve (outside of  $\cup_j \omega_j$), extending by $\rho_n=1$ and $B_n=0$ to the complement.
  In the cylinders $\omega_j$, since the geometry is fixed, for $n$ sufficiently large the  conditions $\sqrt{\beta_n \phimin}\ge \eta_n$,
$\sqrt{\beta_n\phimax}\ll |X_1-X_2|$ and $ \eta_n R\ll \ell$ are satisfied and we can use  Lemma \ref{Lembranch3}.  We then have that $\rho_n=1$ outside $\B'(X_0, 3\ell/4)\times(z_j-t_\ell,z_j+t_\ell)\subset \omega_j$.

In order to obtain a quantized field, we define $k_n$  as in (\ref{defK}), which obeys
 $ k_n\in 2\pi\N$ and $k_n\to+\infty$.
Let $\hat{\mu}^n$ be the $k_n$-quantized approximation of $\hat{\mu}$, 
as given by Lemma \ref{lemmaquantize}.
For sufficiently large $n$ we have
 $\frac12 \hat{\mu}\le \hat{\mu}^n\le 2\hat{\mu}$ and for $|x_3|\ge 1-\varepsilon_N$,
\begin{equation*}
 \hat{\mu}_{x_3}^n=\sum_l\phi_l^n\delta_{X_l},
\end{equation*}
 where $|\phi_l^n-1/N^2|\le C(\hat{\mu})/k_n$, $\phi_l^n\le 2/N^2$.
The fluxes $\phi_i^n$ obey $\frac12\phi_i\le \phi_i^n\le 2\phi_i$.
 
We then construct $B_n$ and $\rho_n$
using  Lemma \ref{lemmacurve3} and Lemma \ref{Lembranch3} as discussed above.
The geometry is the one determined by $\hat{\mu}$. 
In particular, the points $y_j$,  and the constants
$M$, $\phimin$, $\phimax$, $\delta_\ell$ and $t_\ell$ do not depend on $n$, and only $\delta_\ell$ and $t_\ell$ depend on $\ell$.
Adding terms gives for sufficiently large $n$ (as a geometry dependent function of $\ell$ and $R$)
\begin{align*}
 \tilde F_{\alpha_n,\beta_n}^{(-1,1)}(\rho_n,B_n)\le &
\sum_i   \left(1+ \frac{C}{(\phi_i^n \alpha_n^{4/3}\beta_n^{2/3})^{1/2}}
+ \frac{C}{\phi_i^n \alpha_n^{4/3}\beta_n^{2/3}}\right)
\int_{a_i}^{b_i} \bigl( K_R\sqrt{\phi_i^n}+\phi_i^n |\dot{X}_i|^2\bigr) dx_3
\\
&+C(\gamma) \sum_j \sqrt{\phimax} t_\ell R^2+ \frac{\ell^2}{t_\ell}\phimax R^2,
\end{align*}
where the first sum runs over all curves and the second over the cylinders and where $C(\gamma)>0$ is a constant depending on $\gamma$.
In the limit $n\to+\infty$ we have $\phi_i^n\to\phi_i$, $\alpha_n^2\beta_n\to+\infty$ and
therefore, inserting the definition of $t_\ell$,
\begin{align*}
\limsup_{n\to+\infty} \tilde F_{\alpha_n,\beta_n}^{(-1,1)}(\rho_n,B_n)\le &
\sum_i   \int_{a_i}^{b_i} \bigl( K_R\sqrt{\phi_i}+\phi_i |\dot{X}_i|^2\bigr) dx_3
+C(\gamma) R^ 2 \ell \sum_j  \left(\frac{\sqrt{\phimax}}{M}  +M\phimax\right)\,.
\end{align*}
The sum over $j$ depends only on $\hat{\mu}$.
Therefore if $\ell$ is chosen sufficiently small we have 
\begin{equation*}
\limsup_{n\to +\infty} \,  \tilde F_{\alpha,\beta}^{(-1,1)}(\rho_n,B_n)\le \frac{K_R}{K_*} I(\mu)+\frac1{N^{1/2}}\,.
\end{equation*}
Moreover, thanks to \eqref{distestim} and Lemma \ref{lemmaquantize}, we have 
\begin{align*}
 W_2^2(\beta^{-1}_n B^n_3,\hat{\mu})&\les W_2^2(\beta^{-1}_n B^n_3,\hat{\mu}^n)+W_2^2(\hat{\mu},\hat{\mu}^n)\\
 &\les C(\hat{\mu})\b_n  + t_\ell \, \sharp\{Y_j\}+ C(\hat{\mu})k_n^{-1}\\
 &\les C(\hat{\mu})\b_n  + \frac{\ell}{M} \, \sharp\{Y_j\}+ C(\hat{\mu})k_n^{-1}.
\end{align*}
Again, since $ \sharp\{Y_j\}/M$ only depends on $\hat\mu$, if we choose $\ell$ sufficiently small then 
\begin{equation}\label{estimW2inside}
\limsup_{n\to+\infty} W_2^2(\beta^{-1}_n B^n_3,\hat{\mu})\le N^{-3/2}.
\end{equation}
We finally address the boundary layer, focusing for definiteness on the side $x_3>0$.
We first apply once more  Lemma \ref{lemmacurveendpoints}
to each curve for $x_3\in (1-\eps_N, 1-\eps_N/2)$, so that  the resulting 
 fields $B_n$ and $\rho_n$  obey for all $x_3\in (1-\eps_N/2,1)$,
\begin{equation*}
 B_n(x)=\sum_l e_3 \chi_{Q^l}(x') \text{ and }
 \rho_n(x)= \min\{1,\eta_n^{-1} \dist(x',\cup_l Q^l)\},
\end{equation*}
where $Q^l$ are squares centered in the $N^2$ points $X_l$ with $|Q^l|= \b\phi_l^n $ , and $|\phi_l^n-1 /N^2|\le C(\hat{\mu})/k_n$.
Then we modify $\rho_n$ and $B_n$ in the set $x_3\in (1-\eps_N/2,1)$ using Proposition \ref{propboundraylayer}.
This results in new fields $\hat \rho_n$, $\hat B_n$ which obey
  \begin{equation*}
  \tilde F_{\alpha_n,\beta_n}^{(1-\eps_N/2,1)}(\hat\rho_n,\hat B_n)+
  \alpha_n^{1/3}\beta_n^{7/6}
  \|\beta_n^{-1}{(\hat{B}_n)_3}- 1\|_{H^{-1/2}(x_3=1)}^2\les
  \frac{1}{N^{1/2}}+
  \frac{\beta_n^{1/3}}{\alpha_n^{1/3}} + \frac{\beta_n^{1/2}}{N^{1/2}}
  +\frac{\alpha_n^{1/3}}{\beta_n^{5/6}} \frac{C(\hat\mu)N^4}{k_n^2}  \,.
  \end{equation*}
By   (\ref{eqlbassumptcoeff})  and (\ref{defK}) 
all terms up to the first one tend to zero as $n\to+\infty$. Therefore
\begin{equation*}
\limsup_{n\to +\infty} \,  \tilde F_{\alpha_n,\beta_n}(\hat\rho_n,\hat B_n)\le \frac{K_R}{K_*}
I(\hat{\mu})
+\frac{C}{N^{1/2}}\,.
\end{equation*}
Finally, we pass from $(\hat\rho_n,\hat B_n)$ to $(u_n,A_n)$ via Proposition \ref{thirdupperbound} and 
conclude the proof of \eqref{estimenerregular}.
We also obtain \eqref{estimW2regular}  since 
\[W^2_2(\beta^{-1}_n (\hat{B}_n)_3, \hat{\mu})\les W^2_2(\beta^{-1}_n B^n_3, \hat{\mu})+N^{-3/2}.
\]
\end{proof}

It only remains to combine the different steps. 
\begin{proof}[Proof of Proposition \ref{gammalimsup-v2}]
 By density (Theorem \ref{theodens}) there is a sequence $\mu^N\in \MNreg(Q_{1,1})$ of $N$-regular measures
 converging weakly to $\mu$, with $\limsup_{N\to+\infty} I(\mu^N)\le I(\mu)$. 
 Fix $R\ge 3$ and let $(u_n^N,A_n^N)$ be as in Proposition 
 \ref{propubregular}.  Taking a diagonal subsequence 
 (first with $N$, then with $R$) we obtain
\begin{equation*}
\limsup_{n\to +\infty} \,  \wE(u_n^{N(n)},A_n^{N(n)})\le I(\mu),
\end{equation*}
 and  $\b_n^{-1}(B_n^{N(n)})_3\weaklim \mu $. From the compactness statement of Proposition \ref{gammaliminf} and uniqueness of $m$ one obtains $\b_n^{-1} (B_n^{N(n)})'\weaklim m$ and $\b_n^{-1}(1-|u_n^{N(n)}|^2)\weaklim \mu $.
\end{proof}

\section*{Acknowledgment}
M. G. thanks E. Esselborn, F. Barret and P. Bella for stimulating discussions about optimal transportation, and the FMJH PGMO fundation for partial support through the  project COCA. 
The work of S.C. was partially supported 
by the Deutsche Forschungsgemeinschaft through the Sonderforschungsbereich 1060 
{\sl ``The mathematics of emergent effects''}. 

\bibliographystyle{alpha-noname}
\bibliography{CGOS}

\end{document}